\documentclass[reqno,11pt]{article}

\usepackage{a4wide}
\usepackage{geometry}

\usepackage[strings]{underscore}
\usepackage{hyperref}
\usepackage{adjustbox}
\usepackage{graphicx}
\usepackage{subfigure}
\usepackage[toc,page]{appendix}
\usepackage{pst-all}
\usepackage[english]{babel}
\usepackage{pdfsync}
\usepackage{amsmath,amssymb,amsthm}
\usepackage{mdwlist}
\usepackage{mathrsfs}
\usepackage{mathtools}
\usepackage{bm}
\usepackage{dsfont}

\newcommand{\dint}{\mathrm{d}}
\newcommand{\dist}{\mathsf{d}}
\newcommand{\Lint}{\mathit{L}}
\newcommand{\X}{{\sf X}}
\newcommand{\Y}{{\sf Y}}
\newcommand{\mea}{\mathfrak{m}}

\newcommand{\mms}{(\X, \dist, \mea)}

\newcommand{\real}{\mathbb{R}}
\newcommand{\N}{\mathbb{N}}
\newcommand{\W}{\mathit{W}^{1,2}}
\newcommand{\WT}{\mathit{W}^{2,2}}

\newcommand{\ener}{\mathsf{E}}
\newcommand{\hess}{\mathrm{Hess}}
\newcommand{\Div}{{\rm{div}}}

\newcommand{\Test}{{\rm{Test}}}

\newcommand{\fct}{\mathrm{F}}

\newcommand{\V}{\mathit{V}}
\newcommand{\D}{\mathrm{D}}

\newcommand{\Sclass}{{S}^2}

\newcommand{\loc}{\mathsf{loc}}

\newcommand{\bpi}{\boldsymbol{\pi}}

\newcommand{\mae}{\mea\text{-a.e.}}

\newcommand{\Hil}{\mathscr{H}}

\newcommand{\T}{\mathit{T}}

\DeclarePairedDelimiter{\abs}{\lvert}{\rvert}

\def\Xint#1{\mathchoice 
  {\XXint\displaystyle\textstyle{#1}}% 
  {\XXint\textstyle\scriptstyle{#1}}% 
  {\XXint\scriptstyle\scriptscriptstyle{#1}}% 
  {\XXint\scriptscriptstyle\scriptscriptstyle{#1}}% 
  \!\int} 
\def\XXint#1#2#3{{\setbox0=\hbox{$#1{#2#3}{\int}$} 
  \vcenter{\hbox{$#2#3$}}\kern-.5\wd0}} 
\def\-int{\Xint -}

\numberwithin{equation}{subsection}

\newcommand{\R}{\mathbb{R}}

\newcommand{\mm}{{\mbox{\boldmath$m$}}}
\newcommand{\nn}{{\mbox{\boldmath$n$}}}

\newcommand{\ppi}{{\mbox{\boldmath$\pi$}}}

\newcommand{\sfd}{{\sf d}}

\newcommand{\sfh}{{\sf h}}

\newcommand{\Kliminf}{K\kern-3pt-\kern-2pt\mathop{\rm lim\,inf}\limits}  % Kuratowski liminf di insiemi
\newcommand{\supp}{\mathop{\rm supp}\nolimits}   % supporto 
\newcommand{\Lip}{\mathop{\rm Lip}\nolimits}          %Lipschitz constant
\renewcommand{\d}{{\mathrm d}}

\newcommand{\restr}[1]{\lower3pt\hbox{$|_{#1}$}} %restrizione funzione

\newcommand{\la}{\left<}                  % brackets
\newcommand{\ra}{\right>}
  
\newcommand{\nchi}{{\raise.3ex\hbox{$\chi$}}}
\newcommand{\weakto}{\rightharpoonup}
\newcommand{\limi}{\varliminf}

\newcommand{\fr}{\penalty-20\null\hfill$\blacksquare$}                      %quadratino nero alla fine del remark, se non vi piace, la cosa migliore e' `svuotare' la macro, cosi' non bisogna intervenire sul testo

\newcommand{\prob}[1]{\mathscr P(#1)}                   % misure di probabilita
\newcommand{\e}{{\rm{e}}}                           % mappa di valutazione, bisogna mettere `a mano' il tempo t
\renewcommand{\mm}{\mathfrak m}                                %misura di riferimento

\renewcommand{\nn}{\mathfrak n}                                %misura di riferimento

\renewenvironment{proof}{\removelastskip\par\medskip   % inizio e fine dimostrazione
\noindent{\em Proof.}\rm}{\penalty-20\null\hfill$\square$\par\medbreak}

\newtheorem{theorem}{Theorem}[section]

\newtheorem{corollary}[theorem]{Corollary}
\newtheorem{lemma}[theorem]{Lemma}
\newtheorem{proposition}[theorem]{Proposition}
\newtheorem{assumption}[theorem]{Assumption}

\newtheorem{thmdef}[theorem]{Theorem/Definition}

\newtheorem{definition}[theorem]{Definition}

\newtheorem{example}[theorem]{Example}
\newtheorem{remark}[theorem]{Remark}

\newcommand{\test}[1]{{\rm Test F}(#1)}

\newcommand{\h}{{\sfh}}

\newcommand{\CD}{{\sf CD}}
\newcommand{\RCD}{{\sf RCD}}

\newcommand{\HS}{{\lower.3ex\hbox{\scriptsize{\sf HS}}}}
\renewcommand{\H}{{\rm Hess}}
\renewcommand{\div}{{\rm div}}

\newcommand{\sx}{{\sf x}}
\newcommand{\sy}{{\sf y}}

\newcommand{\bs}{{\sf bs}}
\newcommand{\E}{{\sf E}}
\DeclareMathOperator*{\lw}{\text{weak-}L^2\text{-lim}}

\title{Partial derivatives in the nonsmooth setting}

\author{Nicola Gigli \thanks{SISSA, Trieste. Email: ngigli@sissa.it} \quad Chiara Rigoni  \thanks{Institut f\"ur Angewandte Mathematik, Universit\"at Bonn. Email: rigoni@iam.uni-bonn.de}}

\makeindex

\begin{document}
\maketitle

\begin{abstract}
We study partial derivatives on the product of two metric measure structures, in particular in connection with calculus via modules as proposed by the first named author in \cite{Gigli14}.

Our main results are:
\begin{itemize}
\item[i)] The extension to this non-smooth framework of  Schwarz's theorem about symmetry of mixed second derivatives
\item[ii)] A quite complete set of results relating the property $f\in W^{2,2}(\X\times\Y)$ on one side with that of $f(\cdot,y)\in W^{2,2}(\X)$ and $f(x,\cdot)\in W^{2,2}(\Y)$ for a.e.\ $y,x$ respectively on the other. Here $\X,\Y$ are $\RCD$ spaces so that second order Sobolev spaces are well defined.
\end{itemize}
These results are in turn based upon the study of Sobolev regularity, and of the underlying notion of differential, for a map with values in a Hilbert module: we mainly apply this notion to the map $x\mapsto\d_\sy f(x,\cdot)$ in order to build, under the appropriate regularity requirements, its differential $\d_\sx\d_\sy f$.
\end{abstract}

\tableofcontents
\section{Introduction}
Given  two smooth manifolds $M,N$ and points $p\in M$, $q\in N$, it is a classical fact in differential geometry that 
\begin{equation}
\label{eq:i1}
T_{(p,q)}(M\times N)\cong T_pM\times T_qN
\end{equation}
where the isomorphism is given by $(\d\pi_N,\d\pi_M)$, with $\pi_N,\pi_M$ being the projection onto the respective coordinates and $\d\pi_N,\d\pi_M$ denote their differentials. This being valid for any $(p,q)\in M\times N$ we can regard any vector field $v$ on $M\times N$ via its components, i.e.\ we can write  $v=(v_x,v_y)$ where $v_x$ is a vector field on $M$ parametrized by $q\in N$ (or, better, a section of the pullback of the tangent bundle of $M$ via $\pi_M$) and symmetrically $v_y$ is a vector field on $N$ parametrized by $p\in M$.

\bigskip

In the nonsmooth setting the analogue of \eqref{eq:i1} has been investigated in our earlier work \cite{GR17}. Here the concept of tangent space is interpreted in terms of the language of normed modules introduced in \cite{Gigli14} and is tightly linked to Sobolev calculus. In this sense it should not surprise that the validity of (the analogue of) \eqref{eq:i1} on metric measure spaces is related to the validity of appropriate tensorization properties of Sobolev spaces (this latter topic has been investigated only relatively recently: the first studies we are aware of have been conducted in \cite{AmbrosioGigliSavare11-2}, see also \cite{Ambrosio-Pinamonti-Speight15} and \cite{GH15} for more recent contributions). Without entering into details here, let us just say that these tensorization properties of Sobolev spaces we are alluding to are always satisfied if the spaces $\X,\Y$ are $\RCD$ spaces (see Proposition \ref{prop:rcdass}) and that in this case the appropriate analogue of \eqref{eq:i1} holds (see Theorem \ref{dectang}). This has been established in our earlier work \cite{GR17}.

In the present manuscript we push the investigation further, in particular in connection with higher order differentiations (in \cite{GR17} only first order derivatives have been considered). Our main results are:
\begin{itemize}
\item[i)] A generalization to this setting of the classical theorem by Schwarz about symmetry of second order derivatives: we shall see in Theorem \ref{thm:swap} that under fairly general assumptions the identity
\begin{equation}
\label{eq:i2}
\d_\sx\d_\sy f=\d_\sy\d_\sx f
\end{equation}
holds even in this framework
\item[ii)] A thorough study of the relation between second order regularity in the product of two $\RCD$ spaces and second order regularity in the factors which-among other things-shows that the expected formula
\begin{equation}
\label{eq:i3}
\H (f)=\Bigg(
\begin{array}{cc}
\H_\sx (f)&\d_\sx\d_\sy f\\
\d_\sy\d_\sx f&\H_\sy (f)
\end{array}
\Bigg)
\end{equation}
holds true  (see Theorems \ref{prop:dafatt}, \ref{prop:w22prod} and Propositions \ref{prop:dafatt2}, \ref{prop:h22prod}). The discussion here is complicated by the fact that on $\RCD$ spaces there is no single `second order Sobolev space' as it is not clear whether the $`H$' version (obtained by completion of the space of `smooth' functions) coincides with the $`W$' one (obtained via integration by parts). Along similar lines we also investigate differentiability properties of vector fields in relation to that of their components (see Theorems \ref{prop:dafattCov}, \ref{thm:wCprod}  and Proposition \ref{prop:hCprod}). 
\end{itemize}
A crucial aspect of our analysis, and perhaps the most important advance to the theory among those given in the current manuscript, is in the possibility of speaking of Sobolev regularity for functions with values in a (Hilbert) module, and in the related concept of differential. Notice indeed that in order to just state the identities \eqref{eq:i2} and \eqref{eq:i3}, it is crucial to know what $\d_\sx\d_\sy f$ is and given that for a function $f$ in two variables the object $\d_\sy f$ can be interpreted as the map sending $x$ to $\d_\sy f(x,\cdot)\in L^0(T^*\Y)$, it is imperative to be able to differentiate this sort of maps. This analysis is carried out in Section \ref{se:dermod}. \bigskip

{\bf Acknowledgement:} The second author gratefully acknowledges support by the European Union through the ERC-AdG 694405 RicciBounds.

\section{Notation and preliminary results}

For us a metric measure space $\mms$ will always be a complete and separable metric space $(\X, \dist)$ equipped with a reference non-negative (and non-zero) Borel measure $\mea$ which is finite on bounded sets.

\subsection{Sobolev spaces for locally integrable objects}
In this section we briefly recall some basic definitions of Sobolev-related objects, with particular focus on quantities which are only integrable on bounded sets. For the definition and properties of Sobolev functions and minimal weak upper gradients we refer to \cite{Cheeger00} (see also  \cite{Shanmugalingam00} and the more recent \cite{AmbrosioGigliSavare11, AmbrosioGigliSavare11-2} whose presentation we are going to follow) while for what concerns $L^\infty$ and $L^0$ normed modules we refer to \cite{Gigli14, GR17}.

Given a metric measure space $\mms$, by  $L^2_\loc(\X)$ we mean the space of (equivalence classes w.r.t.\ $\mm$-a.e.\ equality of) Borel functions $f:\X\to\R$ such that $\nchi_Bf\in L^2(\X)$ for every bounded Borel set $B\subset \X$. A sequence $(f_n)\subset L^2_\loc(\X)$ converges to $f$ in $L^2_\loc(\X)$ provided $\nchi_Bf_n\to \nchi_B f$ in $L^2(\X)$ for every bounded Borel set $B\subset \X$.

A test plan $\ppi$ is a Borel probability measure on $C([0,1],\X)$ such that
\[
\begin{split}
(\e_t)_*\ppi&\leq C\mm,\qquad\forall t\in[0,1],\\
\iint_0^1|\dot\gamma_t|\,\d t\,\d\ppi(\gamma)&<\infty,
\end{split}
\]
for some $C>0$.

The {Sobolev class} $\Sclass(\X)$ (resp.\ $\Sclass_\loc(\X)$) is the space of all Borel functions $f: \X\to \R$ for which there exists $G\in \Lint^2(\mea)$ (resp.\ $G\in \Lint^2_\loc(\mea)$) non-negative, called weak upper gradient, such that for every test plan $\bpi$ it holds
\[
\int|f(\gamma_1)-f(\gamma_0)|\,\d\pi(\gamma)\leq \iint_0^1G(\gamma_t)|\dot\gamma_t|\,\d t\,\d\bpi(\gamma).
\]

It can be proved that $f\in \Sclass_\loc(\X)$ and $G$ is a weak upper gradient  if and only if for every test plan $\ppi$ we have that for $\ppi$-a.e.\ $\gamma$ the map $t\mapsto f(\gamma_t)$ is in $W^{1,1}(0,1)$ and 
\[
\Big|\frac\d{\d t}f(\gamma_t)\Big|\leq G(\gamma_t)|\dot\gamma_t|\qquad a.e.\ t\in[0,1].
\]
In particular, this characterization implies the existence of a minimal weak upper gradient in the $\mm$-a.e.\ sense: we call it {\bf minimal weak upper gradient} and denote it by $|\D f|$. %With a simple cut-off argument, \eqref{eq:sobcurve} also shows that $f\in\Sclass_\loc(\X)$ if and only if $\eta f\in\Sclass(\X)$ for every $\eta$ Lipschitz and with bounded support. NON PROPRIO PERCHE' SENZA UN CONTROLLO SULL'INTEGRABILITA' DI f NON SAPPIAMO SE \eta |D f|+|f||D\eta| STIA IN L^2 AMMESSO CHE |Df| STIA IN L^2_loc

The { Sobolev space} $\W(\X)$ (resp.\ $\W_\loc(\X)$) is defined as $\Lint^2\cap\Sclass(\X)$ (resp.\ $\Lint^2_\loc\cap\Sclass_\loc(\X)$). It can be proved that $f\in \W_\loc(\X)$ if and only if $\eta f\in \W(\X)$  for every $\eta$ Lipschitz with bounded support. We recall that $\W(\X)$  is a Banach space when endowed with the norm
\[
\|f\|_{\W(\X)}^2:=\|f\|_{\Lint^2(\mea)}^2+\||\D f|\|_{\Lint^2(\mea)}^2.
\]
We say (\cite{Gigli12}) that $\mms$ is {\bf infinitesimally Hilbertian} provided $\W(\X)$ is a Hilbert space.

It is useful to recall that minimal weak upper gradients have the following important locality property:
\[
|\D f|=|\D g|\qquad \mm-a.e.\text{ on }\{f=g\}\qquad\qquad\forall f,g\in\Sclass_\loc(\X).
\]

From the notion of minimal weak upper gradient it is possible to extract the one of differential via the following result:
\begin{thmdef}\label{thm:defd}
There exists a unique couple $(\Lint^0(T^*\X),\dint)$, where $\Lint^0(T^*\X)$ is a $\Lint^0(\X)$-normed module and $\dint:\Sclass_\loc(\X)\to \Lint^0(T^*\X)$ is  a linear map, such that
\begin{itemize}
\item[i)] $|\dint f|=|\D f|$ $\mm$-a.e.\ for every $f\in \Sclass_\loc(\X)$,
\item[ii)] $\Lint^0(T^*\X)$ is generated by $\{\dint f : f\in\Sclass_\loc(\X)\}$, i.e.\ $\Lint^0$-linear combinations of objects of the form $\dint f$ are dense in $\Lint^0(T^*\X)$.
\end{itemize}
Uniqueness is intended up to unique isomorphism, i.e.\ if $(\mathscr{M},\dint')$ is another such couple, then there is a unique isomorphism $\Phi \colon \Lint^0(T^*\X)\to \mathscr{M}$ such that $\Phi(\dint f)=\dint'f$ for every $f\in\Sclass_\loc(\X)$.
\end{thmdef}

The space of vector fields $\Lint^0(T\X)$ is defined as the dual of the $\Lint^0$-normed module $\Lint^0(T^*\X)$. It can be equivalently characterized as the $\Lint^0$-completion of the dual $\Lint^2(T\X)$ of the $\Lint^2$-normed module $\Lint^2(T^*\X)$ (see \cite{Gigli14}, \cite{Gigli17}). $\Lint^2_\loc(T\X)\subset \Lint^0(T\X)$ is the space of $X$'s such that $|X|\in \Lint^2_\loc(\mea)$.

We say that  $X \in \Lint^2_\loc (T\X)$ has {\bf divergence} in $\Lint^2_\loc$, and  write $X \in \D(\Div_{\loc})$, if there exists $h \in \Lint^2_\loc(\X)$ such that
\[  
\displaystyle \int f h \ \dint \mea = - \int \dint f(X) \ \dint \mea, \quad \text{for every} \,\,\, f \in \W(\X) \,\,\, \text{with bounded support}. 
\]
In this case we call $h$ (which is unique by the density of $\W(\X)$ in $\Lint^2(\mea)$) the divergence of $X$, and denote it by $\div(X)$.

Let us now assume that $\mms$ is infinitesimally Hilbertian, so that the pointwise norms in $\Lint^0(T^*\X)$ and in $\Lint^0(T\X)$ induce pointwise scalar products. In this case the modules $\Lint^0(T^*\X)$ and $\Lint^0(T\X)$ are canonically isomorphic via the Riesz (musical) isomorphism
\[
\flat \colon \Lint^0(\mathit{T} \X) \rightarrow \Lint^0(\mathit{T}^\ast \X) \qquad \ \text{and} \qquad \ \sharp \colon  \Lint^0(\mathit{T}^\ast \X) \rightarrow \Lint^0(\mathit{T} \X)
\]
defined by
\[
X^\flat (Y):=\la X,Y\ra\qquad \text{and} \qquad \la\omega^\sharp,X\ra:=\omega(X)
\]
for every $X,Y\in \Lint^0(T\X)$ and $\omega\in \Lint^0(T^*\X)$. The {\bf gradient} of a function $f\in\W_\loc(\X)$ is defined as $\nabla f:=(\dint f)^\sharp\in \Lint^2_\loc(T\X)$.

We say that $f\in  \W_\loc(\X)$ has {\bf Laplacian} in $\Lint^2_\loc(\mea)$, namely $f \in \D(\Delta_\loc)$, if there exists $h \in \Lint^2_\loc(\mea)$ such that it holds
\[ 
\displaystyle \int g  h \,\dint \mea = - \int \la\nabla f,  \nabla g \ra\dint \mea, \quad \text{for every} \,\,\, g \in \W(\X) \,\,\, \text{with bounded support}
\]
(this is the same as requiring that $\nabla f\in \D({\Div}_\loc)$ with ${\Div}(\nabla f)=h$). In this case we call $h$ the Laplacian of $f$, and we denote it by $\Delta f$. If $f,h\in L^2(\X)$ we shall write $f\in D(\Delta)$ instead of $f \in D(\Delta_\loc)$, and in this case the Laplacian is equivalently defined as infinitesimal generator of the Dirichlet form 
\begin{equation}\label{def:E}
\ener(f):=\left\{\begin{array}{ll}
\dfrac12\displaystyle{\int|\d f|^2\,\d\mm}&\qquad\text{ if }f\in \W(\X),\\
+\infty&\qquad\text{ otherwise}.
\end{array}
\right.
\end{equation}
From the properties of the minimal upper gradient we deduce that $\ener$ is convex, lower semicontinuous and with dense domain, namely $\{ f : \ener(f) < \infty  \}$ is dense in $L^2(\mm)$. Hence, the classical theory of gradient flows of convex functions on Hilbert spaces ensures existence and uniqueness of a 1-parameter semigroup $(\h_t)_{t \ge 0}$ of continuous operators from $L^2(\mm)$ to itself such that for every $f \in L^2(\mm)$ the curve $t \mapsto \h_t(f) \in L^2(\mm)$ is continuous on $[0, \infty)$, absolutely continuous on $(0, \infty)$ and satisfies
\[
\dfrac{\d}{\d t} \h_t(f) = \Delta f, \qquad \text{for a.e. } t > 0,
\]
where it is part of the statement the fact that $\h_t(f) \in \D(\Delta)$ for every $f \in L^2(\mm)$ and $t > 0$ (see \cite{AmbrosioGigliSavare08}, and the references therein). We remark that in the case in which $\mms$ is infinitesimally Hilbertian, the Laplacian and the operators $\h_t$ are linear.

\subsection{Calculus tools on product spaces}\label{Sec:CalcPr}

Let $(\X,\sfd_\X,\mm_\X)$ and $(\Y,\sfd_\Y,\mm_\Y)$ be two metric measure spaces.  The product space $\X \times \Y$ will be always implicitly endowed with the product measure and the distance
\[
(\sfd_\X\otimes\sfd_\Y)^2\big((x_1,y_1),(x_2,y_2)\big):=\sfd_\X^2(x_1,x_2)+\sfd_\Y^2(y_1,y_2).
\]

In the following we will denote by $\pi_\X \colon \X\times\Y\to \X$, the canonical projection on the first coordinate. Observe that this is a map of local bounded deformation, meaning that \ for every bounded set $B\subset \X \times\Y$ there is a constant $C(B)>0$ such that $\pi_\X\restr B$ is $C(B)$-Lipschitz and $(\pi_\X)_*(\mm_\X\otimes\mm_\Y\restr B)\leq C(B)\mm_\X$: this in particular allows to construct a pullback module over $\X\times\Y$ starting from a  $L^0(\X)$-module $\mathscr M$ over $\X$ (see  \cite[Section 3.1]{GR17} for the definition of such pullback). Given the particular structure of the projection map, a very explicit construction can be given to such pullback, as we briefly discuss now.

Let us consider a $L^0(\X)$-module   $\mathscr M$ over $\X$. We have on one side  the pullback $([\pi_\X^\ast]\mathscr M, [\pi_\X^\ast])$  of $\mathscr M$ through $\pi_\X$ (see \cite[Theorem/Definition 3.2]{GR17} and notice that in particular $[\pi_\X^\ast]\mathscr M$ is a $L^0(\X\times\Y)$-module) and on the other the module $L^0(\Y, \mathscr{M})$ that we are going to define now.  $L^0(\Y, \mathscr{M})$ is the space of all the equivalence classes up to $\mm_\Y$-a.e. equality of strongly measurable (i.e., Borel and essentially separably valued) functions from $\Y$ to $\mathscr M$. This space canonically carries the structure of a $L^0(\X \times \Y)$-module. Indeed:
\begin{itemize}
\item[-] the multiplication of an element of $L^0(\Y, \mathscr M)$ by a function $f \in L^0(\X \times \Y)$ is defined as the map $\Y \ni y \mapsto f(\cdot, y)  v(\cdot, y) \in \mathscr M$. Recalling that $L^0(\X\times\Y)\sim L^0(\Y;L^0(\X))$ and approximating  $f \in L^0(\X\times \Y)$ with functions with finite range as maps from $\Y$ to $L^0(\X)$, it is not hard to see  that $y \mapsto f(\cdot, y)  v(\cdot, y)$ has essentially separable range, provided that $y \mapsto v(\cdot, y)$ does;
\item[-] the pointwise norm of $v \in L^0(\Y, \mathscr M)$ is obtained by composing the map $y \mapsto v(\cdot, y) \in \mathscr M$ with the pointwise norm on $\mathscr M$. Again the isomorphism  $L^0(\Y, L^0(\X)) \sim L^0(\X \times \Y)$ ensures that the so-defined map takes values in $ L^0(\X \times \Y)$. 
\end{itemize}
It is then clear that this pointwise norm induces a complete distance on $L^0(\Y,\mathscr M)$ via the formula
\[
\sfd_{L^0}(v,w):=\int 1\wedge|v-w|\,\d\mathfrak n,
\]
where $\mathfrak n\in\prob{\X\times\Y}$ has the same negligible sets of $\mm_\X\times\mm_\Y$ (the topology induced by $\sfd_{L^0}$ is independent on the choice of the particular $\nn$) and thus that, as claimed, $L^0(\Y,\mathscr M)$ is a $L^0(\X\times\Y)$-module.

By construction, $L^0(\Y,\mathscr M)$ is generated by constant maps and for any $v\in \mathscr M$ the map $\hat v\in L^0(\Y,\mathscr M)$ constantly equal to $v$ satisfies $|\hat v|=|v|\circ\pi_\X$. In other words, by the characterization of pullback of modules, the module $L^0(\Y,\mathscr M)$ and the map $v\mapsto \hat v$ can be identified with $([\pi_\X^\ast]\mathscr M, [\pi_\X^\ast])$, meaning that there is a unique isomorphism $\Phi:L^0(\Y,\mathscr M)\to [\pi_\X^\ast]\mathscr M$ such that $\Phi(\hat v)=[\pi_\X^\ast]v$ for any $v\in\mathscr M$.

In what will come next, we shall often implicitly use this identification in the case in which $\mathscr M = L^0(T^*\X)$ , namely $L^0(\Y,L^0(T^*\X))\sim [\pi_\X^*]L^0(T^*\X)$.\\

\bigskip

The fact that $\pi_\X$ is of local bounded deformation ensures that it can be used to pullback 1-forms (see \cite{Gigli14} and \cite[Section 3.1.2]{GR17}): for any $f \in \Sclass_\loc(\X_1)$ we have that
\begin{equation}
\label{eq:locpb1form}
f \circ \pi_\X \in \Sclass_\loc(\X \times \Y) \,\,\, \text{with} \,\,\, \abs{\d(f \circ \pi_\X)} = \abs{\d f} \circ \pi_\X, \quad \mm_\X \otimes \mm_\Y\text{-a.e.},
\end{equation}
and from this fact it follows that there exists a unique linear and continuous map $\pi_\X^\ast \colon L^0(\T^\ast \X)\to  L^0(\T^\ast(\X \times \Y))  $ with the property that
\[
\begin{split}
\pi_\X^\ast(\d f) &= \d (f \circ \pi_\X), \,\,\qquad \forall f \in \Sclass_\loc(\X),\\
\pi_\X^\ast(g \omega) &= g \circ \pi_\X\, \pi_\X^\ast\omega, \,\,\quad \forall g \in L^0(\X), \omega \in L^0(\T^\ast \X),\\
\abs{\pi_\X^\ast \omega} &= \abs{\omega} \circ \pi_\X, \,\,\qquad \mm_\X\otimes\mm_\Y\text{-a.e.}, \, \forall \omega  \in L^0(\T^\ast \X).
\end{split}
\]
All these constructions can be repeated with the roles of $\X$ and $\Y$ inverted.\\

\bigskip

By the universal property of the pullback of modules it is easy to see that the map $\pi_\X^*$ just described splits through the pullback map from $L^0(T\X)$ to $L^0(\Y,L^0(T\X))$ and a module morphism $\Phi_\X:L^0(\Y,L^0(T^*\X))\to L^0(T^*(\X\times  \Y))$. More precisely we have the following proposition (see {\cite[Proposition 3.7]{GR17}} for the proof):
\begin{proposition}\label{prop:defPhi}
Let $(\X,\sfd_\X,\mm_\X)$ and $(\Y,\sfd_\Y,\mm_\Y)$ be two metric measure spaces. There exists a unique $L^0(\X\times\Y)$-linear and continuous map $\Phi_\sx$ from $L^0(\Y,L^0(T^*\X))$ to $L^0(T^*(\X\times  \Y))$ such that
\[
\Phi_\sx(\widehat{\d g})=\d(g\circ\pi_\X)\qquad\forall g\in \Sclass_\loc(\X),
\]
where $\widehat{\d g}:\Y\to L^0(T^*\X)$ is the function identically equal to $\d g$. Such map preserves the pointwise norm.

Similarly, there is a unique $L^0(\X\times\Y)$-linear and continuous map $\Phi_\sy:L^0(\X,L^0(T^*\Y))\to L^0(T^*(\X\times  \Y))$ such that
\[
\Phi_\sy(\widehat{\d h})=\d(h\circ\pi_\Y)\qquad\forall h\in \Sclass_\loc(\Y),
\]
where $\widehat{\d h}:\X\to L^0(T^*\Y)$ is the function identically equal to $\d h$, and  such map preserves the pointwise norm.
\end{proposition}
A way to think at the above is the following. Say that we have two smooth manifolds $M_1$ and $M_2$ and a  map assigning to every $x_2\in M_2$ a 1-form $\omega(x_2)$ on $M_1$. Then we might think at such map as the 1-form $\Omega$ on $M_1\times M_2$ which at the point $(x_1,x_2)$ has value $(\omega(x_2)(x_1),0)$.  Here a way of thinking at $(\omega(x_2)(x_1),0)$ as element of the cotangent space at $(x_1,x_2)$ is the following: say that $\omega(x_2)(x_1)=\d_{x_1}f$ for some smooth function $f\colon M_1\to \R$ ($f$ depends on $x_2$, but such dependence is not emphasized here). Then we can think/define $(\omega(x_2)(x_1),0)$ as the differential of $f\circ \pi_1\colon M_1\times M_2\to\R$ at the point $(x_1,x_2)$. In the setting of metric measure spaces the assignment $\omega_\cdot\mapsto\Omega$ is the map $\Phi_\sx$ defined by Proposition \ref{prop:defPhi} above.

\bigskip

Without further informations on the structure of $\X,\Y$ it seems hard to find other relations between calculus on the base spaces and calculus on the product. In particular, one would expect the map $\Phi_\sx\oplus\Phi_\sy$ (see below for the precise definition and in particular Theorem \ref{dectang}) to be an isomorphism of modules: an investigation of this fact, carried out in \cite{GR17}, shows that it depends on the validity of the `Assumption \ref{ass}' below. The quotation marks are due because we don't know of any example for which such assumption is not satisfied, so perhaps there is a chance that Assumption \ref{ass} is rather a theorem; for our purposes it will be sufficient to know that we can actually prove that Assumption \ref{ass} holds if the given spaces are $\RCD$, see Proposition \ref{prop:rcdass} below.

In what follows we shall denote by $\d_\sx,\d_\sy,\d$ the differentials in $\X,\Y,\X\times\Y$ respectively.

\begin{definition}[Tensorization of the Cheeger energy]\label{def:tensch} Let $(\X,\sfd_\X,\mm_\X)$ and $(\Y,\sfd_\Y,\mm_\Y)$ be two metric measure spaces. We say that they  have the property of tensorization of the Cheeger energy provided for any $f \in \Lint^2(\X \times \Y)$ the following holds:  $ f\in \W(\X\times \Y)$ if and only if \smallskip
\begin{itemize}
\item[-] for $\mm_\X$-a.e. $x \in \X$  it holds $f(x, \cdot) \in\W (\Y)$  with $\displaystyle \int \abs{\d_\sy f(x,\cdot)}^2(y)\,\d (\mm_\X\otimes \mm_\Y)(x,y) < \infty$ \\
\item[-] for $\mm_\Y$-a.e. $y \in \Y$  it holds $f(\cdot, y) \in\W (\X)$ with $\displaystyle \int\abs{\d_\sy f(\cdot,y)}^2(x)\,\d (\mm_\X\otimes \mm_\Y)(x,y) < \infty$
\end{itemize}
and, in this case, we have
\begin{equation} \label{Chpr}
\abs{\d f}^2  = \abs{\d_\sx f}^2 + \abs{\dint_\sy f}^2 \qquad\mm_1\otimes \mm_2\text{-a.e.}.
\end{equation}
\end{definition}

\begin{definition}[Strong measurability of the sections]\label{def:smsec} Let $(\X,\sfd_\X,\mm_\X)$ and $(\Y,\sfd_\Y,\mm_\Y)$ be two metric measure spaces. We say that they have   the property of the strong measurability of the sections if for any $f \in \Sclass_\loc(\X \times \Y)$ the maps $\Y \ni y \mapsto \d_\sy f \in L^0(\T^\ast \X)$ and $\X \ni x \mapsto \d_\sx f \in L^0(\T^\ast \Y)$ are essentially separably valued.
\end{definition}
We shall be interested in spaces $\X,\Y$ satisfying the following:
\begin{assumption}\label{ass} $(\X,\sfd_\X,\mm_\X)$ and $(\Y,\sfd_\Y,\mm_\Y)$   are two metric measure spaces for which both the tensorization of Cheeger energy \ref{def:tensch} and the strong measurability of the sections \ref{def:smsec} hold.
\end{assumption}
For our purposes it is important to recall that if $\X,\Y$ are $\RCD$ spaces, then they satisfy such assumption, see Proposition \ref{prop:rcdass}.

\begin{remark}{\rm
We will refer to \cite{GR17} for the proofs of  the forthcoming results. Notice that there the assumption made  involves the density of a suitable class of functions in $\W(\X\times \Y)$, called ``product algebra'', in the strong topology of $\W(\X\times \Y)$. However, the reason why the density of the product algebra is needed in \cite{GR17} is exactly to show the strong measurability of the sections, therefore all the results proved in \cite{GR17} are still true once we work under Assumption \ref{ass}.} \fr
\end{remark}
The following lemma provides a link from differentials on $\X\times\Y$ to differentials on $\X$ and $\Y$, thus going in the opposite direction of Proposition \ref{prop:defPhi}:
\begin{lemma}[{\cite[Lemma 3.12]{GR17}}]
\label{le:perprod}
Let $(\X,\sfd_\X,\mm_\X)$ and $(\Y,\sfd_\Y,\mm_\Y)$   be two metric measure spaces satisfying Assumption \ref{ass}. 

Then for every $f \in \Sclass_\loc(\X \times \Y)$ we have that $f(\cdot, y) \in\Sclass_\loc(\X)$ for $\mea_2$-a.e.\ $y$ and the map $y\mapsto  \dint_\sx f(\cdot,y)$ (that below we shall simply denote by $\d_\sx f$) belongs to $\Lint^0(\Y, \Lint^0(\mathit{T}^{\ast}\X))$. Moreover, for $(f_n)\subset\Sclass_\loc(\X_1\times\X_2)$ we have
\[
\dint f_n\to \dint f\text{ in }\Lint^0(T^*(\X\times\Y))\qquad\Rightarrow\qquad \dint_\sx f_n\to \dint_\sx f \text{ in }\Lint^0(\Y,\Lint^0(T^*\X)).
\]
 Similarly for the roles of  $\X$ and $\Y$ inverted. Finally, the identity \eqref{Chpr} holds for any $f\in\Sclass_\loc(\X\times\Y)$.
 \end{lemma}
We now turn to the `full' relation between forms on $\X,\Y$ and forms on $\X\times\Y$. To this aim, notice that for given $L^0$-normed modules $\mathscr M_1,\mathscr M_2$  on the same space ${\rm Z}$, the product $\mathscr M_1\times\mathscr M_2$ is canonically a $L^0$-normed module on ${\rm Z}$ once it is endowed with the product topology, the multiplication by $L^0$-functions  given by $f(v_1,v_2):=(fv_1,fv_2)$ and the pointwise norm defined as
\[
|(v_1,v_2)|^2:=|v_1|^2+|v_2|^2.
\]
In particular, $L^0(\Y,L^0(T^*\X))\times L^0(\X,L^0(T^*\Y))$ is a $L^0(\X\times\Y)$-normed module and we can define $\Phi_\sx\oplus\Phi_\sy$ as
\[
\begin{array}{cccc}
\Phi_\sx \oplus \Phi_\sy \colon& \Lint^0 (\Y, \Lint^0(\mathit{T}^{\ast}\X)) \times \Lint^0(\X, \Lint^0(\mathit{T}^{\ast}\Y)) &\rightarrow  &\Lint^0 (\mathit{T}^{\ast}(\X \times \Y))\\
&(\omega, \sigma) &\mapsto& \Phi_\sx(\omega) + \Phi_\sy(\sigma)
\end{array}
\]
We then have the following result:
\begin{theorem}[{\cite[Theorem 3.13]{GR17}}]\label{dectang}
Let $(\X,\sfd_\X,\mm_\X)$ and $(\Y,\sfd_\Y,\mm_\Y)$   be two metric measure spaces satisfying Assumption \ref{ass}. 

Then $\Phi_\sx\oplus\Phi_\sy$  is an isomorphism of modules, i.e.\ it is $L^0(\X\times\Y)$-linear, continuous, 
 surjective and  for every $\omega \in \Lint^0 (\Y, \Lint^0(\mathit{T}^{\ast}\X))$ and $\sigma \in \Lint^0 (\X, \Lint^0(\mathit{T}^{\ast}\Y))$ the following identity holds
\[
|\Phi_\sx(\omega) + \Phi_\sy(\sigma)|^2 = |\omega|^2+|\sigma|^2\quad\mea_\X\otimes\mea_\Y-a.e..
\]
Moreover, for every   $f \in \Sclass_\loc(\X \times \Y)$ it holds:
\[
\dint f = \Phi_\sx ( \dint_\sx f) + \Phi_\sy ( \dint_\sy f).
\]
\end{theorem}
In other words, and in line with the discussion made after Proposition \ref{prop:defPhi}, this last theorem provides a decomposition of $L^0(T^\ast(\X\times\Y))$ in two submodules, the image of $\Phi_\sx$ and the image of $\Phi_\sy$, and we shall think at these as the decomposition of a 1-form on $\X\times\Y$ into its components cotangent to $\X$ and $\Y$ respectively: for brevity, given $\omega\in L^0(T^*(\X\times\Y))$ we shall write
\begin{equation}\label{eq:decome}
\omega=(\omega_\sx,\omega_\sy)
\end{equation}
meaning that $\omega=\Phi_\sx(\omega_\sx)+\Phi_\sy(\omega_\sy)$. Theorem \ref{dectang} above ensures that both $\omega_\sx$ and $\omega_\sy$ are uniquely determined by $\omega$. In the smooth case, given a 1-form $\omega$ on the product of two manifolds $M_1,M_2$ and $(x_1,x_2)\in M_1\times M_2$, the 1-form $\omega_\sx(x_2)$ at the point $x_1$ is the restriction of $\omega(x_1,x_2)$ to the kernel of the differential of $\pi_{M_2}$ in $T_{(x_1,x_2)}(M_1\times M_2)$ (this kernel being isomorphic to $T_{x_1}M_1$ via the differential of $\pi_{M_1}$).

\bigskip

Now recall that if $L^0(T\X)$ is separable, then the dual of $L^0(Y,L^0(T^*\X))$ be canonically identified with  $L^0(Y,L^0(T\X))$ via the coupling
\[
\left.\begin{array}{ll}
L^0(Y,L^0(T^*\X))\ni \omega\\
L^0(Y,L^0(T^*\X))\ni v
\end{array}\right\}\qquad\Rightarrow\qquad
y\,\,\mapsto\,\,\omega(y)(v(y))\in L^0(\X).
\]
Below we shall assume that $L^0(T\X)$ is separable (recall that this is always the case if $\X$ is infinitesimally Hilbertian, see \cite{ACM14}) and constantly identify $L^0(Y,L^0(T\X))$ with $L^0(Y,L^0(T^*\X))^*$ via the above isomorphism.  Same assumption and identification with the roles of $\X,\Y$ swapped.

With this said,  the $L^0(\X\times\Y)$-linear map $\Phi_\sx \colon L^0(\Y,L^0(T^*\X)) \to L^0(T^*(\X\times  \Y))$ has the adjoint
\[
\Phi_\sx^\ast \colon L^0(T(\X\times  \Y)) \to L^0(\Y,L^0(T\X))
\]
characterized by the property that for any $v \in  L^0(T(\X\times  \Y))$ and $\omega \in L^0(\Y,L^0(T^\ast\X))$ it holds
\[ 
\omega \big( \Phi^\ast_\sx(v)\big) = \big(\Phi_\sx(\omega)\big)(v).
\]
Similarly,  the operator $\Phi_\sy^\ast \colon L^0(T(\X\times  \Y)) \to L^0(\X,L^0(T\Y))$, adjoint of $\Phi_\sy$ is characterized by the fact that  for any $v \in  L^0(T(\X\times  \Y))$ and $\eta \in L^0(\X,L^0(T^\ast\Y))$ we have
\[ 
\eta \big( \Phi^\ast_\sy(v)\big) = \big(\Phi_\sy(\eta)\big)(v).
\]
Since the adjoint of $\Phi_\sx\oplus\Phi_\sy$ is 
\[
(\Phi_\sx^*,\Phi_\sy^*):L^0(T(\X\times  \Y)) \to L^0(\Y,L^0(T\X))\times L^0(\X,L^0(T\Y)),
\]
from Theorem \ref{dectang}, we deduce that $(\Phi_\sx^*,\Phi_\sy^*)$ is an isomorphism of modules. 

Given $v\in L^0(T(\X\times\Y))$, we shall write $v_\sx,v_\sy$ in place of $\Phi_\sx^*(v),\Phi_\sy^*(v)$ respectively and also often write
\begin{equation}\label{eq:decv}
v=(v_\sx,v_\sy)
\end{equation}
thus implicitly identifying $L^0(T(\X\times  \Y))$ with $L^0(\Y,L^0(T\X))\times L^0(\X,L^0(T\Y))$ through $(\Phi_\sx^*,\Phi_\sy^*)$. Notice that in the smooth setting $v_\sx,v_\sy$ are the components of the vector field $v$ along $\X,\Y$ respectively.

\bigskip

We conclude recalling that if $\X,\Y$ are infinitesimally Hilbertian, then also $\X\times\Y$ is so and the decomposition given in Theorem \ref{dectang} is orthogonal:
\begin{proposition}[{\cite[Proposition 3.14]{GR17}}] Let $(\X,\sfd_\X,\mm_\X)$ and $(\Y,\sfd_\Y,\mm_\Y)$   be two metric measure spaces infinitesimally Hilbertian and satisfying Assumption \ref{ass}.

Then $\X\times\Y$ is also infinitesimally Hilbertian and for every $\omega\in \Lint^0 ( \Y, \Lint^0(\mathit{T}^* \X) )$ and $\eta\in \Lint^0 ( \X, \Lint^0(\mathit{T}^* \Y) )$ we have
\begin{equation}
\label{eq:orto}
\la \Phi_\sx(\omega),\Phi_\sy(\eta)\ra=0\quad\mea_\X\otimes\mea_\Y-a.e..
\end{equation}
\end{proposition}

\subsection{Other differential operators in the product space}\label{sec:CalcPr2}

Up to now we have seen how the differential behaves under products of spaces. We shall now investigate other differentiation operators (divergence and Laplacian) and for simplicity we shall stick to the case of infinitesimally Hilbertian spaces.

As for the case of differentials, we shall denote by $\div_\sx,\Delta_\sx$ the divergence and Laplacian in the space $\X$ (and similarly for $\Y$) while we keep the un-labeled versions $\div,\Delta$ for the operators in the product space.
\begin{proposition}[{\cite[Proposition 3.15]{GR17}}]\label{divvf}  Let $(\X,\sfd_\X,\mm_\X)$ and $(\Y,\sfd_\Y,\mm_\Y)$   be two metric measure spaces infinitesimally Hilbertian and satisfying Assumption \ref{ass}.

Then $v \in D( \div_{\loc}, \X)$ if and only if $(\hat v, 0) \in D (\Div_{\loc}, {\X \times \Y})$, where $\hat v\in \Lint^0(\Y,\Lint^0(T\X))$ is the function identically equal to $v$ (and we are adopting the convention in \eqref{eq:decv}),  and in this case 
\[
\div\big((\hat v, 0)\big) = \div_\sx (v) \circ \pi_\X.
\]
\end{proposition}
\begin{proposition}[{\cite[Proposition 3.16]{GR17}}]
\label{divprod}
  Let $(\X,\sfd_\X,\mm_\X)$ and $(\Y,\sfd_\Y,\mm_\Y)$   be two metric measure spaces infinitesimally Hilbertian and satisfying Assumption \ref{ass}.

Let $v = (v_\sx, v_\sy) \in \Lint^2(\mathit{T}(\X \times \Y))$ be such that:
\begin{itemize}
\item[-] $v_\sx \in D(\div_\sx, \X)$ for $\mea_\Y$-a.e. $y \in \Y$ with $ \int  |{\div_\sx (v_\sx)} |^2  \ \dint (\mea_\X \otimes \mea_\Y) < \infty$,
\item[-] $v_\sy \in D(\div_\sy, \Y)$ for $\mea_\X$-a.e. $x \in \X$ with $ \int   |{\div_\sy (v_\sy)}|^2  \ \dint (\mea_\X \otimes \mea_\Y) < \infty$.
\end{itemize}
Then $v \in D (\Div,\X\times\Y)$ and 
\begin{equation}
\label{eq:divprod}
\div(v) = \div_\sx(v_\sx) \circ \pi_\X + \div_\sy(v_\sy) \circ \pi_\Y.
\end{equation}
\end{proposition}
\begin{remark}{\rm
The viceversa of Proposition \ref{divprod} - obviously - does not hold, indeed the vector field on $\R^2$ defined by
\[ 
X(x_1, x_2) = \text{sgn}(x_1 - x_2) (\mathbf{e}_1 + \mathbf{e}_2),  
\]
where $\text{sgn}$ is the sign function, is easily seen to have null divergence, while the distributional divergence of its components are Dirac masses (that get canceled out in considering the divergence in the product space).

In other words, the problem is that if one has  integrability for the two quantities on the right hand side of \eqref{eq:divprod}, then the left hand side has the same integrability properties, but from the integrability of the latter one cannot deduce informations on the integrability of both addends in the right hand side. 
}\fr\end{remark}

A direct application of Proposition \ref{divvf} and Proposition \ref{divprod} gives the following result on the Laplacian on the product space:
\begin{proposition}[{\cite[Corollary 3.17]{GR17}}]\label{cor:lapprod}
 Let $(\X,\sfd_\X,\mm_\X)$ and $(\Y,\sfd_\Y,\mm_\Y)$   be two metric measure spaces infinitesimally Hilbertian and satisfying Assumption \ref{ass}. Then:
\begin{itemize}
\item[i)] $f \in D(\Delta_\loc, \X)$ if and only if $f \circ \pi_\X \in D(\Delta_\loc, \X \times \Y)$ and in this case 
\[
\Delta(f \circ \pi_\X) = (\Delta_\sx f) \circ \pi_\X.
\]
\item[ii)] Let $f \in \W(\X \times \Y)$ be such that
\begin{itemize}
\item for $\mea_\X$-a.e.\ $x \in \X$, $f^{x}:=f(x,\cdot) \in D(\Delta, \Y)$ with $\int |\Delta_\sy f^{x}|^2(y)\, \dint (\mea_\X\times\mea_\Y)(x,y) < \infty$,\\
\item for $\mea_\Y$-a.e.\ $y \in \Y$, $f_{y}:=f(\cdot,y) \in D(\Delta, \X)$ with $\int|\Delta_\sx f_{y}|^2(x) \, \dint  (\mea_\X\times\mea_\Y)(x,y) < \infty$.
\end{itemize}
Then $f \in D(\Delta, \X \times \Y)$ and
\[
\Delta f (x,y) = \Delta_\sx f_{y} (x) + \Delta_\sy  f^{x}(y)\qquad\mea_1\times\mea_2-a.e.\ (x,y).
\]
\end{itemize} 
\end{proposition}

\section{Partial first order derivatives on the product space}
\subsection{Partial first order derivatives for functions in two variables}
Aim of this section is to review, in preparation of the subsequent results, the properties of those functions of two variables which are differentiable with respect one of them. We start with the following:
\begin{definition}[{The space $L^0(\Y,S^{2}_{\loc}(\X))$ and the operator $\d_\sx$}]\label{def:l0s2}
Let $(\X,\sfd_\X,\mm_\X),(\Y,\sfd_\Y,\mm_\Y)$ be two metric measure spaces. Then $L^0(\Y,S^{2}_{\loc}(\X))\subset L^0(\Y,L^0(\X))\sim L^0(\X\times\Y)$ is the space of  functions $f\in L^0(\X\times\Y)$ such that $f(\cdot,y)\in S^2_\loc(\X)$ for $\mm_\Y$-a.e.\ $y\in\Y$, and the map $\Y\ni y\mapsto \d_\sx f\in L^0(T^*\X)$ belongs to $L^0(\Y;L^0(T^*\X))$.

We then define $\d_\sx:L^0(\Y,S^{2}_{\loc}(\X))\to L^0(\Y,L^0(T^*\X))$ by $\d_\sx f(y):=\d_\sx f(\cdot,y)$ for $\mm_\Y$-a.e.\ $y\in\Y$.

We shall also denote by $L^2(\Y,S^2(\X))\subset L^0(\Y,S^{2}_{\loc}(\X))$ the space of functions $f$ such that $\int |\d_\sx f|^2\,\d(\mm_\X\otimes\mm_\Y) < \infty$.
\end{definition}

The definition above guarantees that $\d_\sx f$ is a well defined element of $ L^0(\Y,L^0(T^*\X))$ (i.e.\ that it is Borel and essentially separably valued).

%\begin{remark}{\rm
%Recall that, as we have seen at the beginning of Section \ref{Sec:CalcPr}, the fact that the canonical projection $\pi_\X \colon \X \times \Y \to \X$ is a map of local bounded deformation, shows that $L^0(\Y,L^0(T^*\X))$ canonically carries the structure of a $L^0(\X \times \Y)$-normed module, and allows us to identify $L^0(\Y,L^0(T^*\X))$ with the pull-back module of $L^0(T^\ast \X)$ through the map $\pi_\X$, namely $L^0(\Y,L^0(T^*\X)) \sim [\pi_\X^\ast] L^0(T^\ast \X)$.}\fr
%\end{remark}

It is worth noticing that if the space $W^{1,2}(\X)$ is separable (as it is the case if it is Hilbert - recall \cite[Proposition 7.6]{ACM14}), then elements of $L^0(\Y,S^2_\loc(\X))$ can be easily singled out: \black
\begin{proposition}\label{prop:l0s2eq}
Let $(\X,\sfd_\X,\mm_\X),(\Y,\sfd_\Y,\mm_\Y)$ be two metric measure spaces with $W^{1,2}(\X)$ separable and $f\in L^0(\X\times\Y)$.

Then $f\in L^0(\Y,S^2_\loc(\X))$ if and only if for $\mm_\Y$-a.e.\ $y\in \Y$ we have $f(\cdot,y)\in S^2_\loc(\X)$. 
\end{proposition}
\begin{proof} We start with the `only if' and to this aim define $\eta^m:=1\wedge(m-\sfd_\X(x,\bar x))^+$ where $\bar x\in\X$ is a fixed point. Then it follows from our assumption that for every $n,m\in\N$ we have that for  $\mm_\Y$-a.e.\ $y\in\Y$ the function $x\mapsto \eta^m(x)f^n(x,y)$, where $f^n:=n\wedge f\vee (-n)$, belongs to $W^{1,2}(\X)$. Fix $y\in\Y$ for which this holds and let first $m\to\infty$ and then $n\to\infty$ to deduce that $x\mapsto f(x,y)$ belongs to $S^2_\loc(\X)$, giving the claim.

For the `if' we notice that by the very definition of $L^0(\Y,S^2_\loc(\X))$ and up to replace $f$ by $\eta^mf^n$ with $\eta^m,f^n$ as above and letting first $m\to \infty$ and then $n\to\infty$, it is sufficient to deal with the case  $f\in L^2(\X\times\Y)$ with $f(\cdot,y)\in W^{1,2}(\X)$ for $\mm_\Y$-a.e.\ $y\in\Y$. 

In this case to conclude it is sufficient to prove that $f\in L^0(\Y,W^{1,2}(\X))$ that is to say, as $W^{1,2}(\X)$ is separable, that $y\mapsto f(\cdot,y)\in W^{1,2}(\X)$ is (the equivalence class up to $\mm_\Y$-a.e.\ equality of) a Borel map. Since    $y\mapsto f(\cdot,y)\in L^2(\X)$ is Borel by \eqref{eq:locpb1form} and Assumption \ref{ass},  to conclude it is sufficient to show that the Borel structure on $W^{1,2}(\X)$ induced by the $W^{1,2}$-distance coincides with that induced by the $L^2$-distance. Since $\|f\|_{L^2}\leq \|f\|_{W^{1,2}}$ it is clear that $L^2$-open sets are also $W^{1,2}$-open, whence the same is true for Borel ones. For the opposite inclusion we notice that since the $W^{1,2}$-norm is $L^2$-lower semicontinuous, we have that a $W^{1,2}$-closed ball is a $L^2$-closed set, thus a $W^{1,2}$-open ball is the union of a countable number of $L^2$-closed sets, hence $L^2$-Borel. The conclusion follows.
\end{proof}
\begin{remark}[Partial differentiation as example of $D$-structure]{\rm
In \cite{GT01} the authors present an axiomatic approach to the theory of Sobolev spaces over abstract metric measure spaces, introducing the notion of $D$-structure. Given a metric measure space $(\X, \sfd, \mm)$ and an exponent $p \in (1, \infty)$, a $D$-structure on $(\X, \sfd, \mm)$ is any map $D$ associating to any function $u \in L^p_\loc(\mm)$ a family of non-negative Borel functions $D[u]$, called pseudo-gradients, such that a proper list of axioms is fulfilled (see \cite[Section 1.2]{GT01}). Intuitively, these pseudo-gradients provide a control from above on the variation of $u$, in a suitable sense. 

This allows to define the space $W^p(\X, \sfd, \mm, D)$ as the set of all functions in $L^p(\mm)$ admitting a pseudo-gradient in $L^p(\mm)$. Moreover, using standard techniques of functional analysis, to any Sobolev function $u \in W^p(\X, \sfd, \mm, D)$ it can be associated a uniquely determined minimal object $\underline D u \in D[u] \cap L^p(\mm)$, called minimal pseudo-gradient of $u$.

In \cite{GP18}, the authors prove that, under suitable locality assumptions, these $D$-structures give rise to a first-oder differential structure, namely to a natural notion of cotangent module $L^p(T^\ast \X; D)$, whose properties are analogous to the ones of the cotangent module $L^2(T^\ast \X)$: the main result of such paper shows the existence of an abstract differential $\d \colon W^{1, p}(\X, \sfd, \mm, D) \to L^p(T^\ast \X; D)$, which is a linear operator such that for any $u \in W^{1, p}(\X, \sfd, \mm, D)$ the pointwise norm $\abs{\d u} \in L^p(\mm)$ of $\d u$ coincides with $\underline D u$ in the $\mm$-a.e. sense.

We point out that an example of $D$-structure satisfying all the locality conditions as investigated in \cite{GP18} is the one that assigns to any $u\in L^2(\Y,W^{1,2}(\X))\subset L^2(\X\times \Y)$ the set $D[u]:=\{G\in L^2(\X\times\Y): G\geq |\d_\sx u|\ \mm_\X\times\mm_\Y-a.e.\}$.
}\fr\end{remark}

\black

The operator $\d_\sx$ defined on $L^0(\Y,S^{2}_{\loc}(\X))$ trivially  inherits all the calculus rules of $\d_\sx :S^2_\loc(\X)\to L^0(T^*\X)$:

\begin{proposition}[Calculus rules]\label{prop:calpar}Let $(\X,\sfd_\X,\mm_\X),(\Y,\sfd_\Y,\mm_\Y)$ be two metric measure spaces. Then the following hold:
\begin{itemize}
\item[i)]\emph{Closure}. Let $(f_n)\subset L^2(\Y;S^2(\X))$ be $\mm_\X\otimes\mm_\Y$-a.e.\ converging to $f_\infty$. Assume also that for some open sets $\Omega^k\subset \X$ (resp.\ Borel subsets $B^k\subset\Y$) with $\cup_k\Omega^k=\X$ (resp.\ $\mm_\Y(\Y\setminus\cup_kB^k)=0$) and $\omega\in L^0(\Y;L^0(T^*\X))$ we have that  $(\nchi_{\Omega^k}\nchi_{B^k}\d_\sx f_n)$ converges to  $\nchi_{\Omega^k}\nchi_{B^k}\omega$ in the weak topology of the Banach space $L^2(\Y,L^2(T^*\X))$ for any $k\in\N$ (notice that this holds in particular if $\d_\sx f_n\weakto \omega$ in $L^2(\Y,L^2(T^*\X))$).

Then $f_\infty\in  L^2(\Y;S^2(\X))$ and $\d_\sx f_\infty=\omega$.

The same conclusion holds if $(f_n)\subset L^2(\Y;W^{1,2}(\X))$ converges to $f_\infty$ in the weak topology of $L^2(\X\times\Y)$.
\item[ii)]\emph{Locality}. For any $f,g\in L^0(\Y,S^{2}_{\loc}(\X))$ we have
\begin{equation}
\label{eq:locdx}
\d_\sx f=\d_\sx g\qquad\mm_\X\otimes\mm_\Y-a.e.\ on\ \{f=g\}.
\end{equation}
\item[iii)]\emph{Chain rule}. Let $f\in L^0(\Y,S^{2}_{\loc}(\X))$ and $N\subset\R$ Borel and $\mathcal L^1$-negligible. Then $\d_\sx f=0$ $\mm_\X\otimes\mm_\Y$-a.e.\ on $f^{-1}(N)$. Moreover, for $\varphi:\R\to\R$ Lipschitz we have $\varphi\circ f\in  L^0(\Y,S^{2}_{\loc}(\X))$ and 
\[
\d_\sx(\varphi\circ f)=\varphi'\circ f\d_\sx f,
\] 
(notice that by what just claimed, the actual definition of $\varphi'\circ f$ on  the set $f^{-1}(\{\text{non differentiability points of }\varphi\})$ is irrelevant because on such set $\d_\sx f$ is zero).
\item[iv)]\emph{Leibniz rule}. Let $f,g\in L^0(\Y,S^{2}_{\loc}(\X))\cap L^\infty_\loc(\X\times\Y)$. Then $fg\in  L^0(\Y,S^{2}_{\loc}(\X))$ as well with
\[
\d_\sx (fg)=f\d_\sx g+g\d_\sx f.
\]
If $g$ only depends on the $y$ variable, then the above holds without assuming that $f\in L^\infty_\loc(\X\times\Y)$.
\end{itemize}
\end{proposition}
\begin{proof} Locality, Chain rule and Leibniz rule follow directly from the definition of $\d_\sx$ and the analogous properties of functions depending solely on the $x$ variable. For the closure we notice that first using Mazur's lemma, then passing to subsequences and finally with a diagonal argument we can assume that for any $k\in\N$ we have $\nchi_{\Omega^k}\d_\sx f_n(y)\to \nchi_{\Omega^k}\omega(y)$ in $L^2(T^*\X)$ for $\mm_\Y$-a.e.\ $y\in\Y$. Thus the claim follows from the closure of the differential operator on functions depending solely on the $x$ variable and on the fact that the result does not depend on the subsequence chosen. The second claim about the closure follows along similar lines by using Mazur's lemma to reduce to strong convergence in $L^2(\X\times\Y)$ and then passing to a subsequence to achieve $\mm_\X\otimes\mm_\Y$-a.e.\ convergence.
\end{proof}

The definitions given allow to reformulate Theorem \ref{dectang} in the following way:
\begin{proposition}\label{prop:Wprod}
Let $(\X,\sfd,\mm_\X)$ and $(\Y,\sfd_\Y,\mm_\Y)$ be two metric measure spaces satisfying Assumption \ref{ass} and let $f\in S^2_\loc(\X\times\Y)$. 

Then $y\mapsto f(\cdot,y)$ is in $L^0(\Y,S^2_\loc(\X))$ and, symmetrically, $x\mapsto f(x,\cdot)$ is in $L^0(\X,S^2_\loc(\Y))$, Moreover  we have
\begin{equation}
\label{eq:difflegati}
\d f=\Phi_\sx(\d_\sx f)+\Phi_\sy(\d_\sy f).
\end{equation}
\end{proposition}
\begin{proof}
This is simply a restatement of Theorem \ref{dectang}.
\end{proof}
In what follows we shall adopt the more compact (and more similar to the one adopted in the smooth setting and in \eqref{eq:decome}) notation
\begin{equation}
\label{eq:coppia}
\d f=(\d_\sx f,\d_\sy f)
\end{equation}
in place of \eqref{eq:difflegati}: this should hopefully clarify that we are dealing with partial derivatives and that $\d_\sx f,\d_\sy f$ are the two components of $\d f$. Recalling the definition of $\Phi_\sx$, this means that for $f\in S^2_\loc(\X)$ we have
\[
\d(f\circ\pi_\X)=(\d_\sx f,0). 
\]
Similarly, if the spaces under consideration are also infinitesimally Hilbertian (as it often will be the case) then we shall adopt a similar notation for the gradients to the one in \eqref{eq:decv}, so that $\nabla f = (\nabla_\sx f,\nabla_\sy f)$ and, for $f\in S^2_\loc(\X)$, $\nabla(f\circ\pi_\X)=(\nabla_\sx f,0)$. 

In this direction it will be useful to notice that what just said and Proposition \ref{divprod} allow to write
\begin{equation}
\label{eq:divdivx}
\div(h\nabla (g\circ\pi_\X))=\div_\sx(h\nabla_\sx g)\qquad\mm_\X\otimes\mm_\Y-a.e.
\end{equation}
for any $h\in \Lip_\bs(\X\times\Y)$ and $g\in D(\Delta_\X)$.

We now show that under appropriate integrability assumptions for both the function and the differential, belonging to $L^0(\Y;S^2_\loc(\X))$ can be checked via integration by parts. Notice that this requires also that Assumption \ref{ass} holds and that  $\X$ is infinitesimally Hilbertian (this latter point can be slightly weakened, but we won't push in this direction).
\begin{proposition}\label{prop:l2div}
Let $(\X,\sfd_\X,\mm_\X),(\Y,\sfd_\Y,\mm_\Y)$ be two metric measure spaces satisfying Assumption \ref{ass} and with $\X$ being infinitesimally Hilbertian. Also, let   $f\in L^2(\X\times\Y)$. Then the following are equivalent:
\begin{itemize}
\item[i)] $f\in L^2(\Y,W^{1,2}(\X))$,
\item[ii)] There is $A\in L^0(\Y;L^0(T^*\X))$ with $|A|\in L^2(\X\times \Y)$ such that
\begin{equation}
\label{eq:partX}
-\int f\div(h\nabla (g\circ\pi_\X))\,\d(\mm_\X\otimes\mm_\Y)=\int h\la A,\d_\sx g\ra \,\d(\mm_\X\otimes\mm_\Y)
\end{equation}
holds for any $h\in \Lip_\bs(\X\times\Y)$ and $g\in D(\Delta_\X)$. 
\end{itemize}
Moreover, if this is the case the choice $A=\d_\sx f$ is the only one for which $(ii)$ holds.
\end{proposition}
\begin{proof}\ \\
\noindent{$(i)\Rightarrow(ii)$}  From \eqref{eq:divdivx} and the definition of $\div_\sx$  we obtain
\[
\begin{split}
-\int f\div(h\nabla (g\circ\pi_\X))\,\d(\mm_\X\otimes\mm_\Y)&=-\int \Big(\int f(\cdot,y)\div_\sx(h(\cdot,y)\nabla_\sx g)\,\d\mm_\X\Big)\d\mm_\Y(y)\\
&=\int\Big(\int h \la \d_\sx f,\d_\sx g\ra\,\d\mm_\X\Big)\d\mm_\Y,
\end{split}
\]
which proves $(ii)$ with $A:=\d_\sx f$.

\noindent{$(ii)\Rightarrow(i)$}   For every $n\in\N$, let $(A^n_i)$ be a Borel partition of $\Y$ made of at most countable sets, with $\mm_\Y(A^n_i)\in(0,\infty)$, ${\rm diam}(A^n_i)\leq \frac1n$ and so that $(A^{n+1}_i)$ is a refinement of $(A^n_i)$. Put $f^n_i:=\mm_\Y(A^n_i)^{-1}\int_{A^n_i}f(\cdot,y)\,\d\mm_\Y(y)\in L^2(\X)$ and $f^n(x,y):=\sum_i\nchi_{A^n_i}(y)f^n_i(x)\in L^2(\X\times\Y)$. It is clear that $f^n\to f$ in $L^2(\X\times\Y)$ and thus  the  lower semicontinuity of the Cheeger-Dirichlet energy $\E_\X$ on $\X$ easily gives
\begin{equation}
\label{eq:limEX}
\int \E_\X(f(\cdot,y))\,\d \mm_\Y(y)\leq \limi_{n\to\infty }\int \E_\X(f^n(\cdot,y))\,\d \mm_\Y(y).
\end{equation}
For every $n,i$ let $(h^{n}_{i,m})\subset \Lip_\bs(\X\times\Y)$ be a sequence of functions 1-Lipschitz in the $x$ variable, with values in $[0,1]$ and such that $(h^n_{i,m}),(|\d_\sx h^n_{i,m}|)$ converge $\mm_\X\otimes\mm_\Y$-a.e.\ to $\nchi_{\X\times A^n_i}$ and $0$, respectively, as $m\to\infty$. Also, for $t>0$ let us put $g^n_{i,t}:=\h_{\X,t}(f^n_i)\in D(\Delta_\X)$. Then an application of the dominate convergence theorem shows that for every $n,i,t$ we have
\[
\begin{split}
-\int f\div_\sx(h^n_{i,m}\nabla_\sx(g^n_{i,t}))\,\d(\mm_\X\otimes\mm_\Y)&\quad\to\quad -\int_{\X\times A^n_i} f\Delta_\X g^n_{i,t}\,\d(\mm_\X\otimes\mm_\Y)\\
\int h^n_{i,m}\la A,\d_\sx g^n_{i,t}\ra \,\d(\mm_\X\otimes\mm_\Y)&\quad\to\quad \int \la A,\nchi_{\X\times A^n_i}\d_\sx g^n_i\ra \,\d(\mm_\X\otimes\mm_\Y)
\end{split}
\]
as $m\to\infty$, so that taking into account our assumption and the identity \eqref{eq:divdivx} we obtain
\begin{equation}
\label{eq:parti}
-\int_{\X\times A^n_i} f\Delta_\X g^n_{i,t}\,\d(\mm_\X\otimes\mm_\Y)=\int \la A,\nchi_{\X\times A^n_i}\d_\sx g^n_i\ra \,\d(\mm_\X\otimes\mm_\Y).
\end{equation}
Now observe that from the closure of the differential it is easy to justify the following computation:
\[
\begin{split}
-\int_{\X\times A^n_i} f\Delta_\X g^n_{i,t}\,\d(\mm_\X\otimes\mm_\Y)&=-\int_{A^n_i}\int_\X f(\cdot,y) \Delta_{\X}\h_{\X,t}(f^n_i)\,\d\mm_\X\,\d\mm_\Y(y)\\
&=-\int_{A^n_i}\int_\X \h_{\X,t/2}(f(\cdot,y)) \Delta_{\X}\h_{\X,t/2}(f^n_i)\,\d\mm_\X\,\d\mm_\Y(y)\\
&=\int_\X\int_{A^n_i}\la\d_\sx \h_{\X,t/2}(f(\cdot,y)),\d_\sx\h_{\X,t/2}(f^n_i)\ra\,\d\mm_\Y(y)\,\d\mm_\X\\
&=\mm_\Y(A^n_i)\int_\X|\d_\sx\h_{\X,t/2}(f^n_i)|^2\,\d\mm_\X\\
&=2\mm_\Y(A^n_i)\E_\X\big(\h_{\X,t/2}(f^n_i)\big).
\end{split}
\]
On the other hand by Young's inequality and the fact that $\E_\X$ is decreasing along the heat flow we have
\[
\begin{split}
\int \la A,\nchi_{\X\times A^n_i}\d_\sx g^n_i\ra \,\d(\mm_\X\otimes\mm_\Y)&\leq \frac12\int_{\X\times A^n_i}|A|^2\,\d(\mm_\X\otimes\mm_\Y)+\mm_\Y(A^n_i)\E_\X(g^n_i)\\
&\leq \frac12\int_{\X\times A^n_i}|A|^2\,\d(\mm_\X\otimes\mm_\Y)+\mm_\Y(A^n_i)\E_\X\big(\h_{\X,t/2}(f^n_i)\big).
\end{split}
\]
Coupling these last two (in)equalities with \eqref{eq:parti} we deduce that
\[
\mm_\Y(A^n_i)\E_\X(f^n_i)=\lim_{t\downarrow0}\mm_\Y(A^n_i)\E_\X\big(\h_{\X,t/2}(f^n_i)\big)\leq \frac12\int_{\X\times A^n_i}|A|^2\,\d(\mm_\X\otimes\mm_\Y).
\]
Summing over $i$ and recalling \eqref{eq:limEX} we conclude that 
\[
\int_\Y\E_\X(f(\cdot,y))\,\d \mm_\Y(y)\leq \frac12\int_{\X\times\Y}|A|^2\,\d(\mm_\X\otimes\mm_\Y),
\] 
showing that $f\in L^2(\Y,W^{1,2}(\X))$ (recall Proposition \ref{prop:l0s2eq}),  as desired.

To conclude that $A=\d_\sx f$ is the only choice for which \eqref{eq:partX} holds, it is enough to show that the vector space generated by elements of the form $h\d_\sx g$ with $h,g$ as in the statement is dense in $L^2(\Y;L^2(T^*\X))$. To see this, notice that arguing as we just did the closure of such space contains all the elements of the form $\nchi_{B\times C}\d_\sx g$ for $B,C$ Borel subsets of $\X,\Y$ respectively and $g\in W^{1,2}(\X)$. Then the conclusion follows from the fact that $L^2(T^*\X)$ is generated by differentials of $W^{1,2}$-functions and the fact that piecewise constant maps from $\Y$ to $L^2(T^*\X)$ are dense in $L^2(\Y;L^2(T^*\X))$.
\end{proof}

\subsection{Sobolev maps with values in a Hilbert module}\label{se:dermod}
In this section we study the differential of a function on $\X$  with values in a Hilbert module $\Hil$ over $\Y$. The prototype case we want to cover is that of $\Hil:=L^2(T^*\Y)$ as we aim to give a meaning, in the non-smooth context, to the Sobolev regularity of $x\mapsto \d_\sy f$ and to its differential $\d_\sx \d_\sy f$.

For the purpose of the present discussion, a Hilbert module on $\Y$ is a $L^2$-normed $L^\infty$-module on $\Y$. Given such a module $\Hil$ and its dimensional decomposition $(E_i)_{i\in \N\cup\{\infty\}}$ (see \cite{Gigli14}), a local Hilbert base $(e_i)_{i\in\N}$ is a collection of elements of the $L^0$-completion of $\Hil$ such that for every $n\in\N$ we have $\la e_i,e_j\ra=\delta_{ij}$ on $E_{n}$ for every $i,j<n$.

\begin{definition}[The space $S^2_\loc(\X;\Hil_\loc)$]\label{def:s2hl}
Let $(\X,\sfd_\X,\mm_\X),(\Y,\sfd_\Y,\mm_\Y)$ be two metric measure spaces, with $\X$ infinitesimally Hilbertian. Moreover, let $\Hil$ be a separable Hilbert module over $\Y$, and $(e_i)_{i \in \N}$ be a local Hilbert base  of $\Hil$ with $|e_i|\in L^\infty(\Y)$ for any $i \in \N$.

Then the space $S^2_\loc(\X;\Hil_\loc)\subset L^0(\X;\Hil)$ is the space of functions $f$ such that:
\begin{itemize}
\item[i)] the map $y\mapsto \la f(\cdot),e_i\ra(y)$ is in $L^0(\Y;S^2_\loc(\X))$, for every $i \in \N$,
\item[ii)] the function $|\d_\sx f|\colon\X\times\Y\to[0,\infty]$ defined by
\begin{equation}
\label{eq:defdsx}
|\d_\sx f|^2:=\sum_i|\d_\sx (\Phi_i\circ f)|^2
\end{equation}
belongs to $L^0(\Y;L^2_\loc(\X))$ (i.e.\ is such that $|\d_\sx f|(\cdot,y)\in L^2_\loc(\X)$ for $\mm_\Y$-a.e.\ $y\in\Y$),  where $\Phi_i:\Hil\to L^0(\Y)$ is given by $\Phi_i(v):=\la v,e_i\ra$.
\end{itemize}
By $W^{1,2}(\X;\Hil)\subset S^2_\loc(\X;\Hil_\loc)$ we denote the space of functions $f\in S^2_\loc(\X;\Hil_\loc)$ with $|f|,|\d_\sx f|\in L^2(\X\times\Y)$. 
\end{definition}

\begin{remark}\label{rem:relconprima}{\rm
Picking $\Hil:=L^2(\Y)$ and comparing the above with Definition \ref{def:l0s2} we see that  $S^2_\loc(\X;L^2_\loc(\Y))=L^0(\Y,S^2_\loc(\X))$.

Then, taking into account that the infinitesimal Hilbertianity of $\X$ yields the separability of $W^{1,2}(\X)$ and Proposition \ref{prop:l0s2eq} it is clear that $W^{1,2}(\X;L^2(\Y))\sim L^2(\Y;W^{1,2}(\X))$ as both spaces are made of those functions $f\in L^2(\X\times \Y)$ such that $f(\cdot,y)\in W^{1,2}(\X)$ for $\mm_\Y$-a.e.\ $y$ and so that $|\d_\sx f|\in L^2(\X\times \Y)$.
}\fr\end{remark}

\begin{remark}{\rm
A priori, Definition \ref{def:s2hl} depends on the particular fixed local Hilbert base $(e_i)_{i \in \N}$ but we shall see in Proposition \ref{prop:modpart} below that this is not the case. Such independence is also based on the assumption that $\X$ is infinitesimally Hilbertian, which is needed in order to ensure that $\abs{\d_\sx f}$ is well defined (i.e.\ that it does not depend on the particular base chosen). In fact, the proof of Proposition \ref{prop:modpart} below is based on the fact that we can make use of the parallelogram identity in the chain of equalities \eqref{eq:modpart}, and this is actually guaranteed by the fact that $\X$ is infinitesimally Hilbertian. 
}\fr\end{remark}

\begin{proposition}\label{prop:modpart}
Let $(\X,\sfd_\X,\mm_\X),(\Y,\sfd_\Y,\mm_\Y)$ be two metric measure spaces, with $\X$ infinitesimally Hilbertian and let $\Hil$ be a separable Hilbert module over $\Y$.

Then the space $ S^2_\loc(\X;\Hil_\loc)$ does not depend on the particular local base $(e_i)$ of $\Hil$ chosen and for $f\in  S^2_\loc(\X;\Hil_\loc)$  the function $|\d_\sx f|$ defined by \eqref{eq:defdsx} also does not depend on the base.

Moreover the following holds:
\begin{itemize}
\item[i)] \emph{Locality.} For any $f\in S^2_\loc(\X;\Hil_\loc)$ we have
\[
|\d_\sx f|=0\qquad\mm_\X\otimes\mm_\Y-a.e.\ on\ \{f=0\}.
\]
\item[ii)]\emph{Subadditivity.} For any $f,g\in S^2_\loc(\X;\Hil_\loc)$ and $\alpha,\beta\in\R$ we have
\[
|\d_\sx(\alpha f+\beta g)|\leq |\alpha||\d_\sx f|+|\beta||\d_\sx g|\qquad\mm_\X\otimes\mm_\Y-a.e..
\]
\item[iii)]\emph{Lower semicontinuity.} Let $(f_n)\subset \W(\X;\Hil)$ be weakly $L^2(\X;\Hil)$-converging to $f\in L^2(\X;\Hil)$ and such that $|\D_\sx f_n|\weakto G$ in the weak topology of $L^2(\X\times\Y)$ for some $G$. Then $f\in  \W({\X;\Hil})$ and $|\D_\sx f|\leq G$ $\mm_\X\otimes\mm_\Y$-a.e.. 
\end{itemize}
\end{proposition}
\begin{proof} Let $\{e_i\}$ be a local Hilbert base of $\Hil$ such that $f\in S^2_\loc(\X;\Hil_\loc)$ in the sense of the  Definition \ref{def:s2hl} above, and let $\{\tilde e_j\}$ be another local Hilbert base of $\Hil$. Put $f_n:=\sum_{i\leq n}e_i\Phi^i\circ f$ so that $f_n\to f$ $\mm_\X$-a.e.\ as $n\to\infty$ and, as direct consequence of the definitions, $f_n\in S^2_\loc(\X;\Hil_\loc)$. 

We have $\langle f_n,\tilde e_j\rangle=\sum_{i\leq n}\langle\tilde e_j,e_i\rangle\Phi^i\circ f_n$, which is trivially an element in $L^0(\Y;S^2_\loc(\X))$ (notice that $\Phi^i\circ f_n\in S^2_\loc(\X)$ by assumption and $\langle\tilde e_j,e_i\rangle\in L^0(\Y)$, hence both functions can be seen as elements of $L^0(\Y,S^2_\loc(\X))$ and then apply the Leibniz rule in Proposition \ref{prop:calpar}), and
\begin{equation}\label{eq:modpart}
\begin{split}
\sum_j|\d_\sx\langle f_n,\tilde e_j\rangle|^2&=\sum_j\sum_{i,k\leq n}\langle\tilde e_j,e_i\rangle\langle\tilde e_j,e_k\rangle\, \la \d_\sx(\Phi^i\circ f_n), \d_\sx(\Phi^k\circ f_n)\ra\\
&=\sum_{i,k\leq n} \la \d_\sx(\Phi^i\circ f_n), \d_\sx(\Phi^k\circ f_n)\ra\underbrace{\sum_j\langle\tilde e_j,e_i\rangle\langle\tilde e_j,e_k\rangle}_{=\la e_i,e_k\ra=\delta_{ik}}=\sum_{i\leq n}|\d_\sx(\Phi^i\circ f_n)|^2,
\end{split}
\end{equation}
where the order of summation can be swapped because $i,k$ run over a finite set. In particular we have
\begin{equation}
\label{eq:boundnn}
\sum_j|\d_\sx\langle f_n,\tilde e_j\rangle|^2\leq \sum_{i}|\d_\sx(\Phi^i\circ f)|^2\quad\mm_\X\otimes\mm_\Y-a.e.\ \qquad\forall n\in\N.
\end{equation}
It follows that for every bounded Borel subset $B\subset\X$ we have that for $\mm_\Y$-a.e.\ $y\in\Y$ the sequence of functions $(\nchi_B|\d_\sx\langle f_n,\tilde e_j\rangle|(\cdot,y))$ is bounded in $L^2(\X)$. Also, for $\mm_\Y$-a.e.\ $y\in\Y$ the functions $(\langle f_n(\cdot),\tilde e_j\rangle(y))$ converge to $\langle f(\cdot),\tilde e_j\rangle(y)$ $\mm_\X$-a.e., hence what just said and the lower semicontinuity of minimal weak upper gradients gives that $\langle f(\cdot),\tilde e_j\rangle(y)\in S^2_\loc(\X)$ with $|\d_\sx \langle f,\tilde e_j\rangle|(\cdot,y)\restr B\leq G_j(\cdot, y)$ $\mm_\X$-a.e., where $G_j$ is any weak $L^2$-limit of some subsequence of  $(\nchi_B|\d_\sx\langle f_n,\tilde e_j\rangle|(\cdot,y))$. Then from  Lemma \ref{le:normwl} below, \eqref{eq:boundnn} and the arbitrariness of $B$ we deduce that
\[
\sum_j|\d_\sx \langle f,\tilde e_j\rangle|^2\leq\sum_{i}|\d_\sx(\Phi^i\circ f)|^2\quad\mm_\X\otimes\mm_\Y-a.e..
\]
Swapping the roles of the bases $(e_i)$ and $(\tilde e_j)$ we conclude.

Finally, properties $(i),(ii),(iii)$ are direct consequences of the definitions and the analogous properties for $L^2(\Y)$-valued functions (for $(iii)$ we also use Lemma \ref{le:normwl} below).
\end{proof}

\black
\begin{lemma}\label{le:normwl}
Let $(\X,\sfd,\mm)$ be a metric measure space and for every $i\in\N$ let $(g^i_n)\subset L^2(\X)$ be a  sequence of non-negative functions such that $\sqrt{\sum_i|g^i_n|^2}\weakto G$ in $L^2(\X)$ for some $G\geq 0$ as $n\to\infty$. Also, let $g^i\in L^2(\X)$ be non-negative and such that $g^i\leq G^i$ $\mm$-a.e.\ for any weak $L^2$-limit $G^i$ of $(g^i_n)$. Then
\[
\sqrt{\sum_n|g^i|^2}\leq G\qquad\mm-a.e..
\]
\end{lemma}
\begin{proof} With a diagonalization argument, up to pass to a subsequence we can assume that $g^i_n\weakto G^i$ for some $(G^i)\subset L^2(\X)$. Now let $A\subset \ell_2$ be countable and dense in the set of non-negative sequences of $\ell_2$-norm $\leq 1$ and notice that
\begin{equation}
\label{eq:A}
\sqrt{\sum_i|h_i|^2}=\sup_{(a_i)\in A}\sum_ia_ih_i\quad\text{for any sequence of non-negative numbers $h_i$}.
\end{equation}
For any $(a_i)\in A$, $\mm$-a.e.\ we have
\[
\begin{split}
\sum_ia_ig^i\leq \sum_ia_iG^i=\lim_{N\to\infty}\sum_{i=1}^Na_iG^i=\lim_{N\to\infty}\lw_{n\to\infty} \sum_{i=1}^Na_ig^i_n\leq\lw_{n\to\infty}\sum_{i}a_ig^i_n.
\end{split}
\]
Now observe that since for every $n\in\N$ we have that $\mm$-a.e.\ it holds $\sum_{i}a_ig^i_n\leq \sqrt{\sum_i|g^i_n|^2}$, the same relation is in place for the respective weak $L^2$-limits, thus from the above we get $\sum_ia_ig^i\leq G$ $\mm$-a.e.. Then the arbitrariness of $(a_i)\in A$ and  \eqref{eq:A} give the conclusion.
\end{proof}
To the Sobolev space $S^2_\loc(\X;\Hil_\loc)$ we can canonically associate a differentiation operator:
\begin{theorem}[Module-valued partial derivatives]
Let $(\X,\sfd_\X,\mm_\X),(\Y,\sfd_\Y,\mm_\Y)$ be two metric measure spaces, with $\X$ infinitesimally Hilbertian and let $\Hil$ be a separable Hilbert module over $\Y$. 

Then there exists a unique couple $(L^0(T^*\X;\Hil),\d_\sx)$ where $L^0(T^*\X;\Hil)$ is a $L^0(\X\times\Y)$-normed module, $\d_\sx:  S^2_\loc(\X;\Hil_\loc)\to L^0(T^*\X;\Hil^0)$ is linear and satisfies
\begin{itemize}
\item[i)] for any $f\in  S^2_\loc(\X;\Hil_\loc)$ the pointwise norm of $\d_\sx f$ coincides with $|\d_\sx f|$  $\mm_\X\otimes\mm_\Y$-a.e.,
\item[ii)] $L^0(\X\times\Y)$-linear combinations of elements of the form $\d_\sx f$ for $f\in   S^2_\loc(\X;\Hil_\loc)$ are dense in $L^0(T^*\X;\Hil)$.
\end{itemize}
Uniqueness here is intended up to unique isomorphism, i.e.\ if $(\tilde{ \mathscr M},\tilde \d_\sx)$ is another such couple, then there is a unique isomorphism $\Phi:L^0(T^*\X;\Hil)\to \tilde{ \mathscr M}$ such that $\Phi\circ\d_\sx=\tilde \d_\sx$.

An explicit example of couple as above is given by  the module 
\begin{equation}
\label{eq:defl0h}
L^0(T^*\X;\Hil):=L^0(\Y;L^0(T^*\X))\otimes L^0(\X;\Hil)
\end{equation}
and the operator
\begin{equation}
\label{eq:defdx}
 S^2_\loc(\X;\Hil_\loc) \ni f\quad \mapsto\quad \d_\sx f:=\sum_i\d_\sx (\Phi_i\circ f)\otimes \hat e_i,
\end{equation}
where $(e_i)\subset\Hil$ is a local Hilbert base of $\Hil$, $\hat e_i\in L^0(\X;\Hil)$ is the function constantly equal to $e_i$ and  $\Phi_i:\Hil\to L^0(\Y)$ is given by $\Phi_i(\cdot):=\la\cdot, e_i\ra$, for any $i\in\N$.
\end{theorem}
\begin{proof} Uniqueness is a simple consequence of the definitions, see e.g.\ the arguments in \cite[Theorem/Definition 3.2]{GR17} and notice that they can be easily adapted -  we omit the details. Existence also follows by mimicking the construction in \cite{Gigli14}, \cite{GR17} or, alternatively,  by the explicit construction we provide below.

To check that the couple made by  the module $L^0(\Y; L^0(T^*\X))\otimes L^0(\X;\Hil^0)$ and the operator $\d_\sx$ as defined by \eqref{eq:defdx} satisfies the requirements, we must first check that $\d_\sx$ is well defined, i.e.\ that the series in \eqref{eq:defdx} converges in $L^0(\Y;L^0(T^*\X))\otimes L^0(\X;\Hil^0)$. To see this, notice that the addends are pointwise orthogonal and that $\{ \hat e_i=0\}\subset\{\Phi^i\circ f=0\}$, thus for any $N,M\in \N$, $N<M$ we have
\begin{equation}
\label{eq:normadx}
\Big|\sum_{i=N}^M\d_\sx (\Phi_i\circ f)\otimes \hat e_i\Big|^2= \sum_{i=N}^M|\d_\sx (\Phi_i\circ f)|^2,
\end{equation}
and since for $\mm_\Y$-a.e.\ $y\in\Y$ and bounded Borel set $B\subset\X$ the series $\sum_i|\d_\sx (\Phi_i\circ f)|^2(\cdot,y)\nchi_B$ is convergent in $L^2(\X)$, this is sufficient to conclude that the series in   \eqref{eq:defdx} converges in $L^0(T^*\X;L^0(\Y))\otimes L^0(\X;\Hil^0)$. 

Then the fact that $\d_\sx$ is linear is obvious from the definition, while the fact that $(i)$ holds follows from \eqref{eq:normadx} above. 

To check $(ii)$, notice that for any $f\in L^0(\Y;S^2_\loc(\X))$ and  $i\in\N$, the function $F := fe_i \colon \X\to\Hil$, defined by $F(x)=f(x,\cdot)e_i$ for $\mm_\X$-a.e.\ $x\in\X$, satisfies $\Phi^j\circ F=0$ for $j\neq i$ and $\Phi^i\circ F=f\nchi_{\X\times\{e_i\neq 0\}}$. In particular, directly from Definition \ref{def:l0s2} we have that $\Phi^i\circ F\in L^0(\Y;S^2_\loc(\X))$ with $\d_\sx (\Phi^i\circ F)=\nchi_{\X\times\{e_i\neq 0\}}\d_\sx f$. It is then clear, by the very Definition \ref{def:s2hl}, that $F\in S^2_\loc(\X;\Hil_\loc)$ with
\[
\d_\sx F=(\nchi_{\X\times\{e_i\neq 0\}}\d_\sx f)\otimes \hat e_i=\d_\sx f\otimes (\nchi_{\X\times\{e_i\neq 0\}}\hat e_i)=\d_\sx f\otimes \hat e_i.
\]
The conclusion follows from the fact that differentials of functions in $L^0(\Y;S^2_\loc(\X))$ generate $L^0(\Y;T^*\X))$ (trivial consequence of the fact that differentials of functions on $S^2_\loc(\X)$ generate $L^0(T^*\X)$) and the family $(\hat e_i)$ generates $L^0(\X;\Hil)$.
\end{proof}
The explicit representation as in \eqref{eq:defdx} allows some manipulation of the object $\d_\sx f$. For instance, if $\Hil'$ is another Hilbert module over $\Y$ and $T\colon L^0(\X;\Hil)\to L^0(\X;\Hil')$ is a module morphism, then $T$ induces a map, still denoted by $T$ (with a slight abuse of notation, as perhaps ${\rm Id}\otimes T$ would be the canonical choice), from $L^0(T^*\X;L^0(\Y))\otimes L^0(\X;\Hil)$ to $L^0(T^*\X;L^0(\Y))\otimes L^0(\X;\Hil')$ by acting on the `second factors'. More precisely, the map
\[
L^0(T^*\X;L^0(\Y))\otimes L^0(\X;\Hil)\ni \sum_{i=1}^n\omega_i\otimes v_i\quad\mapsto\quad \sum_{i=1}^n\omega_i\otimes (Tv_i)\in L^0(T^*\X;L^0(\Y))\otimes L^0(\X;\Hil') 
\]
which is trivially well defined by the $L^0(\X\times\Y)$-linearity of $T$, can be shown (see below) to be continuous, and thus can be uniquely extended to a continuous map from $L^0(T^*\X;L^0(\Y))\otimes L^0(\X;\Hil)$ to $L^0(T^*\X;L^0(\Y))\otimes L^0(\X;\Hil')$. To check the continuity one possibility is to notice that we can always rewrite a finite sum such as  $\sum_{i=1}^n\omega_i\otimes v_i$ so that the $\omega_i$'s are pointwise orthogonal and once this is done we have
\[
\Big|\sum_{i=1}^n\omega_i\otimes (Tv_i)\Big|^2=\sum_{i=1}^n|\omega_i|^2|Tv_i|^2\leq|T|^2\sum_{i=1}^n|\omega_i|^2|v_i|^2=|T|^2\,\Big|\sum_{i=1}^n\omega_i\otimes v_i\Big|^2.
\]
Thus for $T$ as above it makes sense to speak about $T(\d_\sx f)$ for $f\in S^2_\loc(\X;\Hil_\loc)$. 

We shall mostly apply this construction in two cases: either when - as in formula \eqref{eq:chaindx} - $T:\Hil\to\Hil'$ is a module morphism which induces by post-composition a module morphism from $L^0(\X;\Hil)$ to $L^0(\X;\Hil')$, or when - as in formula \eqref{eq:leibdx} - a given element $w\in L^0(\X;\Hil)$ is considered and $T:L^0(\X;\Hil)\to L^0(\X\times\Y)$ is given by $T(v):=\la v,w\ra$.

\bigskip

We now collect some properties of the newly defined   differentiation operator:
\begin{proposition}\label{thm:closured} Let $(\X,\sfd_\X,\mm_\X),(\Y,\sfd_\Y,\mm_\Y)$ be two metric measure spaces, with $\X$ infinitesimally Hilbertian and let $\Hil$ be a separable Hilbert module over $\Y$.   Then:
\begin{itemize}
\item[i)] \emph{Closure.} Let $(f_n) \subset S^2_\loc(\X;\Hil_\loc))$ be $\mm_\X$-a.e.\  converging to some $f \in L^0(\X ;\Hil)$. Assume also that for some open sets $\Omega^k\subset \X$ (resp.\ Borel subsets $B^k\subset\Y$) with $\cup_k\Omega^k=\X$ (resp.\ $\mm_\Y(\Y\setminus\cup_kB^k)=0$) and $\omega\in L^0(T^*\X;\Hil)$ we have that  $(\nchi_{\Omega^k}\nchi_{B^k}\d_\sx f_n)$ converges to $\nchi_{\Omega^k}\nchi_{B^k}\omega$ in the weak topology of  $ L^2(T^*\X;\Hil)$ (this happens or instance if $(\d_\sx f_n\weakto \omega)$ in $ L^2(T^*\X;\Hil)$). 

Then $f \in  S^2_\loc(\X;\Hil_\loc))$ with $\d_\sx f = \omega$.
\item[ii)] \emph{Locality.} For any $f,g\in S^2_\loc(\X;\Hil_\loc)$ we have
\[
\d_\sx f=\d_\sx g\qquad\mm_\X\otimes\mm_\Y-a.e.\ on\ \{f=g\}.
\]
\item[iii)] \emph{Leibniz rule.} For any $f,g\in S^2_\loc(\X;\Hil_\loc)$ with $|f|,|g|\in L^\infty_\loc(\X\times\Y)$ we have $\la f,g\ra\in S^2_\loc(\X;L^0_\loc(\Y))$ with 
\begin{equation}
\label{eq:leibdx}
\d_\sx\la  f, g\ra=\la \d_\sx f, g\ra+\la  f,\d_\sx g\ra.
\end{equation}
\item[iv)]\emph{Chain rule.} Let $\Hil'$ be another Hilbert module over $\Y$, $T:\Hil\to\Hil'$ a module morphism with $|\T|\in L^\infty(\Y)$ and $f\in S^2_\loc(\X;\Hil_\loc)$. Then $T\circ f\in S^{2}_\loc(\X;\Hil'_\loc)$ and
\begin{equation}
\label{eq:chaindx}
\d_\sx (T\circ f)=T(\d_\sx f).
\end{equation}
\end{itemize}
\end{proposition}
\begin{proof} The closure property $(i)$ is a direct consequence of the the very definition \eqref{eq:defdx} and of the closure of the differential of functions in $L^0(\Y;S^2_\loc(\X))$ established in point $(i)$ of Proposition \ref{prop:calpar}. Similarly,  point $(ii)$ is a direct consequence of the analogous property \eqref{eq:locdx} and Definition \eqref{eq:defdx}.

For what concerns $(iii)$, let $(e_i)$ be a local Hilbert base of $\Hil$, write $f_i:=\Phi^i\circ f$, $g_i:=\Phi^i\circ g$ for brevity and notice that $\la f,g\ra=\sum_if_ig_i$ with the series converging $\mm_\X\otimes\mm_\Y$-a.e.. Then observe that from the Leibniz rule proved in point $(iv)$ of Proposition \ref{prop:calpar} we have that
\[
\d_\sx\sum_{i=1}^Nf_ig_i=\sum_{i=1}^Ng_i\d_\sx f_i+\sum_{i=1}^Nf_i\d_\sx g_i,
\]
thus by the closure of the differential the conclusion will follow if we prove that for any $\Omega\subset\X$ open and bounded there is a sequence $k\mapsto B^k\subset\Y$ of Borel sets with $\mm_\Y(\Y\setminus \cup_kB^k)=0$ such that $\nchi_\Omega\nchi_{B^k}(\sum_{i=1}^Ng_i\d_\sx f_i+\sum_{i=1}^Nf_i\d_\sx g_i)\to \nchi_\Omega\nchi_{B^k}(\la \d_\sx f, g\ra+\la  f,\d_\sx g\ra)$ in $L^2(T^*\X;L^2(\Y))$ as $N\to\infty$ for every $k\in\N$. We pick $B^k:=\{y\in\Y\ :\ \int_\Omega |\d_\sx f|^2(\cdot,y)+|\d_\sx g|^2(\cdot,y)\,\d\mm_\X<k\}\cap B_k(\bar y)$ for some fixed $\bar y\in\Y$, and notice that the assumption $f,g\in  S^2_\loc(\X;\Hil_\loc))$ gives $\mm_\Y(\Y\setminus \cup_kB^k)=0$. Then we observe that
\[
\begin{split}
\int_{\Omega\times B^k}\Big|\sum_{i=N}^Mg_i\d_\sx f_i\Big|^2\,\d(\mm_X\times\mm_\Y)&\leq \int_{\Omega\times B^k}{\sum_{i=N}^M|g_i|^2}{\sum_{i=N}^M|\d_\sx f_i|^2}\,\d(\mm_X\times\mm_\Y)\\
&\leq \||g|\|_{L^\infty(\Omega\times B^k)} \int_{\Omega\times B^k}{\sum_{i\geq N}|\d_\sx f_i|^2}\,\d(\mm_X\times\mm_\Y)
\end{split}
\]
and that by construction it holds 
\[ 
\int_{\Omega\times B^k}{\sum_{i\geq 0}|\d_\sx f_i|^2}\,\d(\mm_X\times\mm_\Y)= \int_{\Omega\times B^k}{\sum_{i\geq 0}|\d_\sx f_i|^2}\,\d(\mm_X\times\mm_\Y)\leq k\,\mm_\X(\Omega)\mm_\Y(B_k(\bar y))<\infty.
\]
These show that $\nchi_\Omega\nchi_{B^k}\sum_{i=1}^Ng_i\d_\sx f_i\to \nchi_\Omega\nchi_{B^k}\sum_{i=1}^\infty g_i\d_\sx f_i$ in $L^2(T^*\X;L^2(\Y))$ as $N\to\infty$ for every $k\in\N$, so that the conclusion follows noticing that $\la\d_\sx f,g\ra=\sum_ig_i\d_\sx f_i$ (from $\d_\sx f=\sum_i\d_\sx f_i\otimes\hat e_i$ and $g=\sum_i g_i\hat e_i$)  and arguing similarly for the other addend.

We pass to $(iv)$ and start claiming that for $f\in S^2_\loc(\X;L^0_\loc(\Y))$ and $v\in\Hil$ with $|v|\in L^\infty(\Y)$, the function $\X\ni x\mapsto f(x,\cdot)v\in \Hil$ belongs to $S^2_\loc(\X;\Hil_\loc))$ with $\d_\sx (fv)=\d_\sx f\otimes v$. To see this, notice that we can write $v=|v|e_1$, where $e_1\in \Hil$ is the first element of some local Hilbert base, use the fact that $f|v|\in S^2_\loc(\X;L^0_\loc(\Y))$ with $\d_\sx(f|v|)=|v|\d_\sx f$ (point $(iii)$ of Proposition \ref{prop:calpar}) and conclude recalling the very definitions of $S^2_\loc(\X;\Hil_\loc)$ and of $\d_\sx$ given in  \eqref{eq:defdx}.

Now denote by $(e_i),(e'_j)$ local Hilbert basis of $\Hil,\Hil'$ respectively and put for brevity $f_i:=\Phi^i\circ f\in S^2_\loc(\X;L^0_\loc(\Y))$ and $a_{ij}:=\langle Te_i,e'_j\rangle\in L^\infty(\Y)$. Then what we just proved ensures that for any $i,j$ we have $f_ia_{ij}e'_j\in S^{2}_\loc(\X;\Hil'_\loc)$ with $\d_\sx(f_ia_{i,j} e'_j)=\d_\sx f_i\otimes (a_{ij} e'_j)$. Adding up in $j$ and using the closure of the differential proved in point $(i)$ above (notice that $\sum_j|a_{ij}|^2=|e_i|^2\in L^\infty(\Y)$) we deduce that $T(f_ie_i)=f_iT(e_i)=\sum_jf_ia_{ij}e'_j$ belongs to $S^2_\loc(\X;\Hil'_\loc)$ with 
\[
\d_\sx \big(T(f_ie_i)\big)=\sum_j\d_\sx f_i\otimes (a_{ij} e'_j)=\d_\sx f_i\otimes \Big(\sum_ja_{ij} e'_j\Big)=\d_\sx f_i\otimes T(e_i)=T\big(\d_\sx f_i\otimes e_i\big),
\]
where the last identity comes from the very definition of the action of $T$ on $L^0(T^*\X;\Hil)=L^0(T^*\X;L^0(\Y))\otimes L^0(\X;\Hil)$. The conclusion \eqref{eq:chaindx} now follows adding up in $i$ and using again the closure of the differential.
\end{proof}

We conclude the section with a statement similar to Proposition \ref{prop:l2div} which allows to check whether $f$ belongs to $W^{1,2}(\X;\Hil)$ via integration by parts:
\begin{proposition}\label{prop:perdxdy}
Let $(\X,\sfd_\X,\mm_\X),(\Y,\sfd_\Y,\mm_\Y)$ be two metric measure spaces satisfying Assumption \ref{ass} with $\X$ being infinitesimally Hilbertian, $\Hil$ an Hilbert module over $\Y$ and $f\in L^2(\X;\Hil)$. Then the following are equivalent:
\begin{itemize}
\item[i)] $f\in W^{1,2}(\X;\Hil)$,
\item[ii)] There is $A\in L^0(T^*\X;\Hil)$ with $|A|_\HS\in L^2(\X\times \Y)$ and a generating set $D\subset \Hil$ made  of bounded elements such that
\begin{equation}
\label{eq:partXX}
-\int \la f,v\ra \div(h\nabla (g\circ\pi_\X))\,\d(\mm_\X\otimes\mm_\Y)=\int h\la A, \d_\sx g \otimes v\ra \,\d(\mm_\X\otimes\mm_\Y)
\end{equation}
holds for any $v\in D$,  $h\in \Lip_\bs(\X\times\Y)$ and $g\in D(\Delta_\X)$. 
\end{itemize}
Moreover, if this is the case the choice $A=\d_\sx f$ is the only one for which $(ii)$ holds and \eqref{eq:partXX} holds for every bounded element $v\in\Hil$.
\end{proposition}
\begin{proof}\ \\
\noindent{$(i)\Rightarrow(ii)$} Let $v\in\Hil$ be with $|v|\in L^\infty$ and $T:\Hil\to L^0(\Y)$ be given by $T(\cdot):=\la \cdot,v\ra$. Then point $(iv)$ in Proposition \ref{thm:closured} gives that $\la f,v\ra\in W^{1,2}(\X;L^2(\Y))$ with $\d_\sx (\la f,v\ra)=\la \d_\sx f,v\ra\in L^2(\Y;L^2(T^*\X))$. Thus we can  apply \eqref{eq:partX} of Proposition \ref{prop:l2div}  to $\la f,v\ra$ with $g,h$ as in the statement to get
\[
\begin{split}
-\int \la f,v\ra \div(h\nabla (g\circ\pi_\X))\,\d(\mm_\X\otimes\mm_\Y)&=\int h\langle\la \d_\sx f,v\ra,\d_\sx g \rangle\,\d(\mm_\X\otimes\mm_\Y)\\
&=\int h\la \d_\sx f,\d_\sx g \otimes v\ra \,\d(\mm_\X\otimes\mm_\Y),
\end{split}
\]
which is the claim with $A=\d_\sx f$.

\noindent{$(ii)\Rightarrow(i)$} Applying Proposition \ref{prop:l2div} to $\la f,v\ra$ for $v\in D$ we obtain that $\la f,v\ra\in W^{1,2}(\X;L^2(\Y))$ with $\d_\sx \la f,v\ra=\la A,v\ra$ (this expression meaning that $\la\d_\sx \la f,v\ra,\omega\ra=\la A,\omega\otimes v\ra$ for any $\omega\in L^0(\Y;L^0(T^*\X))$). By the linearity of the differential and taking also into account the second claim in the Leibniz rule in point $(iv)$ of Proposition \ref{prop:calpar} we deduce that for $v_i\in D$ and $a_i\in L^\infty(\Y)$ we have $\la f,\sum_ia_iv_i\ra\in W^{1,2}(\X;L^2(\Y))$ with 
\[
\d_\sx \big\langle f,\big(\sum_ia_iv_i\big)\big\rangle=\sum_ia_i\d_\sx \langle f,v_i\rangle=\sum_ia_i\la A, v_i\ra=\langle A, \sum_ia_iv_i\rangle,
\]
where the sums here are finite.
From the fact that $D$ generates $\Hil$ it is easy to deduce that for any $h\in\Hil$ with $|h|\in L^\infty(\Y)$ there exists a sequence of $L^\infty$-linear combinations of elements of $D$ whose pointwise norm is uniformly bounded and $\mm_\Y$-a.e.\ converging to $h$. Hence from what we  proved above and the closure of the differential we conclude that for any $h\in\Hil$ with $|h|\in L^\infty$ we have 
\begin{equation}
\label{eq:Ah}
\la f,h\ra\in W^{1,2}(\X;L^2(\Y))\quad\text{ with }\quad \d_\sx\la f,h\ra=\langle A, h\rangle.
\end{equation}
In particular this holds for $h=e_i$ with $(e_i)$ a local Hilbert base of $\Hil$. Hence from the very definition of $W^{1,2}(\X;\Hil)$, the identity  $\sum_i|\la A,e_i\ra|^2=|A|_\HS^2$ and the assumption $|A|_\HS\in L^2(\X,\times\Y)$ we conclude that $f\in W^{1,2}(\X;\Hil)$. 

To obtain that $A=\d_\sx f$ notice that from \eqref{eq:defdx} we have $\la \d_\sx f,e_i\ra=\d_\sx\la f,e_i\ra$ and recall \eqref{eq:Ah} to deduce that $\d_\sx\la f,h_i\ra=\la A,e_i\ra$ for any $i$. The claim follows.
\end{proof}

\subsection{Schwarz's theorem on symmetry of mixed second derivatives}
Let $\Hil^1,\Hil^2$ be two Hilbert modules over the same space (in our case that would be $\X\times\Y$). Then there is a natural \emph{transposition} map $A\mapsto A^{\rm tr}$ from $\Hil^1\otimes\Hil^2$ to $\Hil^2\otimes\Hil^1$, defined as the only $L^\infty$-linear and continuous extension of the map sending $v_1\otimes v_2$ to $v_2\otimes v_1$ for any $v_i\in\Hil^i$, $i=1,2$. It is trivial to check that the transposition is well defined by this as well as the fact that transposing twice returns the identity.

\bigskip

Now consider infinitesimally Hilbertian spaces $\X,\Y$ and a function $f\in L^2(\Y;W^{1,2}(\X))$. Then in principle we have $\d_\sx f\in  L^2(\Y;L^2(T^*\X))$. Now assume that in fact we have more regularity and that in fact it holds $\d_\sx f\in W^{1,2}(\Y;L^2(T^*\X))$. Then recalling \eqref{eq:defl0h} we see that the differential $\d_\sy\d_\sx f$ is in 
\begin{equation}
\label{eq:dxdy}
L^2(T^*\Y;L^2(T^*\X))\subset L^0(T^*\Y;L^0(T^*\X))\cong L^0(\X;L^0(T^*\Y))\otimes L^0(\Y;L^0(T^*\X)).
\end{equation}
In these circumstances and assuming also that $f\in L^2(\X;W^{1,2}(\Y))$ (see Example \ref{ex:facile} below), it is natural to wonder if we also have $\d_\sy f\in W^{1,2}(\X;L^2(T^*\Y))$ and whether mixed derivatives commute. Inspecting the spaces where $\d_\sx\d_\sy f$ and $\d_\sx\d_\sy f$ belong, we see that the correct way to phrase this question is to ask whether
\begin{equation}
\label{eq:schwarz}
\d_\sx\d_\sy f=(\d_\sy\d_\sx f)^{\rm tr}.
\end{equation}
holds, where here the transposition is acting on $L^0(\X;L^0(T^*\Y))\otimes L^0(\Y;L^0(T^*\X))$ (recall \eqref{eq:dxdy}).

In the following theorem we prove that this is actually the case under natural assumptions. Before stating it let us remark that for $\X,\Y$ infinitesimally Hilbertian we have
\[
W^{1,2}(\X\times\Y)=W^{1,2}(\X;L^2(\Y))\cap W^{1,2}(\Y;L^2(\X)),
\]
as can be seen from Proposition \ref{prop:Wprod} together with the equivalence $W^{1, 2}(\X; L^2(\Y)) \sim L^2(\Y; W^{1, 2}(\X))$ pointed out in Remark \ref{rem:relconprima}. In particular, for $f\in W^{1,2}(\X\times\Y)$ both $\d_\sx f$ and $\d_\sy f$ are defined and these objects belong to $L^2(\Y;L^2(T^*\X))$ and $L^2(\X;L^2(T^*\Y))$ respectively.

In the course of the forthcoming proof and also in subsequent part of the paper, given $f_1\in L^0(\X)$ and $f_2\in L^0(\Y)$ we shall denote by $f_1\otimes f_2\in L^0(\X\times\Y)$ the function given by
\[
f_1\otimes f_2(x,y):=f_1(x)f_2(y).
\]
We then have:
\begin{theorem}\label{thm:swap}
Let  $(\X,\sfd_\X,\mm_\X),(\Y,\sfd_\Y,\mm_\Y)$ be two infinitesimally Hilbertian metric measure spaces satisfying Assumption \ref{ass}.  Let $f\in W^{1,2}(\X\times\Y)$ and assume that  $\d_\sx f\in W^{1,2}(\Y;L^2(T^*\X))$. 

Then $\d_\sy f\in W^{1,2}(\X;L^2(T^*\Y))$ as well and \eqref{eq:schwarz} holds.
\end{theorem}
\begin{proof}
Let $g_1\in D(\Delta_\sx)\cap\Lip(\X)$ (resp.\ $g_2\in D(\Delta_\sy)\cap\Lip(\Y)$) and $h=h_1\otimes h_2\in \Lip_\bs(\X\times\Y)$ for $h_1\in \Lip_\bs(\X)$ and $h_2\in \Lip_\bs(\Y)$. Then from Proposition \ref{prop:perdxdy} applied with the roles of $\X,\Y$ swapped and $\Hil:=L^2(T^*\X)$ together with the assumption $\d_\sx f\in W^{1,2}(\Y;L^2(T^*\X))$ we conclude that
\begin{equation}
\label{eq:swap1}
-\int\langle\d_\sx f,\d_\sx g_1\rangle\,\div(h\nabla (g_2\circ\pi_\Y))\,\d(\mm_\X\otimes\mm_\Y)=\int h\langle \d_\sy\d_\sx f,\d_\sy g_2\otimes\d_\sx g_1\rangle\,\d(\mm_\X\otimes\mm_\Y).
\end{equation}
Recalling \eqref{eq:divdivx} we get
\[
\div(h\nabla (g_2\circ\pi_\Y))=\div_\sy(h\nabla_\sy g_2)=h\Delta_\sy g_2+\la\d_\sy  h,\d_\sy g_2\ra=h\Delta_\sy g_2+h_1\la\d_\sy  h_2,\d_\sy g_2\ra,
\]
where in the last step we used the identity $\d_\sy h=h_1\d_\sy h_2$ which follows directly from Definition \ref{def:l0s2}. It is  clear that the rightmost side in the above belongs to $L^2(\Y;W^{1,2}(\X))$ with differential $\d_\sx$ given by $h_2\Delta_\sy g_2 \,\d_\sx h_1+\la\d_\sy h_2,\d_\sy g_2\ra\d_\sx h_1$, thus we have
\[
\begin{split}
-\int\langle\d_\sx f,\d_\sx g_1\rangle&\,\div(h\nabla( g_2\circ\pi_\Y))\,\d(\mm_\X\otimes\mm_\Y)\\
&=-\int \int\langle\d_\sx f,\d_\sx g_1\rangle\,\big(h\Delta_\sy g_2+h_1\la\d_\sy  h_2,\d_\sy g_2\ra\big)\,\d\mm_\X\,\d\mm_\Y\\
&=\int \int  f\big(h\Delta_\sx g_1\Delta_\sy g_2+h_1\Delta_\sx g_1\la\d_\sy  h_2,\d_\sy g_2\ra\\
&\qquad\qquad+h_2\Delta_\sy g_2  \la\d_\sx g_1,\d_\sx h_1\ra+\la\d_\sy h_2,\d_\sy g_2\ra\la\d_\sx g_1,\d_\sx h_1\ra\big)\,\d\mm_\X\,\d\mm_\Y\\
&=\cdots\text{(by the symmetry of the last expression)}\\
&=-\int\langle\d_\sy f,\d_\sy g_2\rangle\,\div(h\nabla (g_1\circ\pi_\X))\,\d(\mm_\X\otimes\mm_\Y).
\end{split}
\]
Hence from \eqref{eq:swap1} we get
\[
-\int\langle\d_\sy f,\d_\sy g_2\rangle\,\div(h\nabla (g_1\circ\pi_\X))\,\d(\mm_\X\otimes\mm_\Y)=\int h\langle \d_\sy\d_\sx f,\d_\sy g_2\otimes\d_\sx g_1\rangle\,\d(\mm_\X\otimes\mm_\Y)
\]
and the claim follows by the arbitrariness of $g_1,g_2,h$, by applying  Proposition \ref{prop:perdxdy} again (with $D:=\{\nabla_\sy g_2:g_2\in D(\Delta_\sy)\cap\Lip(\Y)\}$).
\end{proof}

\begin{example}\label{ex:facile}{\rm Let $\X=\Y=[0,1]$ with the Euclidean distance and measure and consider the function $f:=\nchi_{[0,1]\times[0,1/2]}\in L^2([0,1]\times[0,1])$. It trivially belongs to $L^2(\Y;W^{1,2}(\X))$ with $\d_\sx f=0$, thus evidently $ \d_\sx f\in W^{1,2}(\Y;L^2(T^*\X))$ with $\d_\sy\d_\sx f=0$. Yet, obviously $f\notin L^2(\X;W^{1,2}(\Y))$, showing that the assumption $f\in L^2(\X;W^{1,2}(\Y))$ is necessary in Theorem \ref{thm:swap}.
}\fr\end{example}

\section{Second order partial derivatives on the product space}
\subsection{Basic material}
\subsubsection{Reminders about calculus on $\RCD$ spaces}

Here we briefly recall those notions about calculus on $\RCD$ spaces that will be used in the rest of the manuscript. These will be used throughout the text without further notice.

From now on we fix a metric measure space $\mms$, satisfying the $\RCD(K, \infty)$ condition for some $K \in \R$, meaning that it is an infinitesimally Hilbertian metric measure space satisfying the $\CD(K, \infty)$ condition of Lott-Villani \cite{Lott-Villani09} and Sturm \cite{Sturm06I}. This in particular means that the energy functional introduced in \eqref{def:E} is a quadratic form: we call {\bf heat flow} $(\h_t)$ its gradient flow in $\Lint^2(\mea)$, which is linear and self-adjoint. It is useful to recall the following standard a priori estimates:
\begin{align}
\ener(\h_t f) &\le \frac{1}{4 t} \| f \|^2_{\Lint^2(\mea)}, \label{apriori1}\\
\| \Delta \h_t f   \|^2_{\Lint^2(\mea)} &\le \frac{1}{2 t^2} \| f \|^2_{\Lint^2(\mea)}, \label{apriori2}
\end{align}
valid for every $t > 0$ and every $f \in \Lint^2(\mea)$. Moreover, it holds
\begin{equation}\label{eq:maxh}
\| \h_t f \|_{L^2(\mm)} \le \| f\|_{L^2(\mm)}, \qquad \forall t \ge 0, \, \forall f \in L^2 (\mm).
\end{equation}

Finally, a crucial property of the heat flow which is strongly related to the lower curvature bound is the {\bf Bakry-\'Emery contraction estimate} (see \cite{Gigli-Kuwada-Ohta10} and \cite{AmbrosioGigliSavare11-2}):
\begin{equation}\label{BE}
\abs{\d \h_t f}^2 \le e^{-2 K t} \h_t(\abs{\d f}^2), \qquad \mae, \quad \forall t \ge 0, \quad \forall f \in \W(\X)
\end{equation}
from which it follows the $L^\infty-\Lip$ regularization estimate
\begin{equation}
\label{eq:linftylip}
\||\d \h_t f|\|_{L^\infty}\leq\frac1{\sqrt{2I_{2K}(t)}} \|f\|_{L^\infty},\qquad\forall f\in L^2(\mm),\ t>0,
\end{equation}
where $I_{K}(t):=\int_0^te^{Ks}\,\d s$.

Then the space of {\bf test functions} $\test{\X}$ on $\X$ is defined as
\[
\test\X:=\Big\{f\in L^\infty\cap W^{1,2}(\X)\cap D(\Delta)\ :\ |\d f|\in L^\infty(\X),\ \Delta f\in W^{1,2}(\X)\Big\}.
\]
It turns out that $\test{\X}$ is an algebra dense in $W^{1,2}(\X)$ (in particular gradients of test functions generate the whole $L^2(T\X)$) and that $|\d f|^2\in W^{1,2}(\X)$ for any $f\in\test{\X}$ (see \cite{Savare13}). Another direct consequence of the Ricci curvature lower bound is the following regularity result, which allows to pass from a Sobolev information to a metric one (see \cite{AmbrosioGigliSavare11-2}, \cite{Gigli13}, \cite{Gigli13over}):
\begin{equation}\label{eq:SobtoLip}
\text{any} \, f \in \test{\X} \, \text{ has a Lipschitz representative }\, \bar f \colon \X \to \R \, \text{ with } \, \Lip(\bar f) \le \| \abs{\nabla f} \|_{L^\infty(\mm)}.
\end{equation}
Moreover, it is also useful to know that
\begin{equation}
\label{eq:l2test}
f\in L^2 \cap L^\infty (\mm)\qquad\Rightarrow\qquad \h_tf\in \test{\X},\ \forall t>0,
\end{equation}
which in particular guarantees the density of $\test{\X}$ in $W^{1, 2}(\X)$, and that
\begin{equation}
\label{eq:testcutoff}
\forall B\subset\X\text{ bounded there is $f\in\test{\X}$ with bounded support identically 1 on $B$.}
\end{equation}

The space $W^{2,2}_\loc(\X)$ is the space of $f\in W^{1,2}_\loc(\X)$ for which there exists $A\in L^2_\loc((T^*)^{\otimes 2}\X)$ such that
\begin{equation}
\label{eq:defhess}
\begin{split}
2\int& hA(\nabla g,\nabla \tilde g)\,\d\mm\\
&=-\int\la\nabla f,\nabla g\ra\div(h\nabla\tilde g)+\la\nabla f,\nabla \tilde g\ra\div(h\nabla g)+h\la\nabla f,\nabla(\la\nabla g,\nabla\tilde g\ra)\ra\,\d\mm
\end{split}
\end{equation}
for every $g,\tilde g,h\in\test{\X}$ with bounded support. In this case we call $A$ the {Hessian} of $f$ and denote it by  $\H f$. If $f\in \W(\X)$ and $\H f\in  L^2((T^*)^{\otimes 2}\X)$ we say that $f\in W^{2,2}(\X)$ (it is not hard to check that this definition coincides with the one proposed in \cite{Gigli14}).  The space $W^{2,2}(\X)$ is a separable Hilbert space when endowed the norm
\[
\|f\|_{W^{2,2}(\X)}^2:=\|f\|_{L^2(\X)}^2+\||\D f|\|_{L^2(\X)}^2+\||\H f|_\HS\|_{L^2(\X)}^2.
\]
We have $D(\Delta)\subset W^{2,2}(\X)$ with 
\begin{equation}
\label{eq:hesscont}
\int |\H(f)|_\HS^2\,\d\mm_\X\leq \int|\Delta f|^2-K|\d f|^2\,\d\mm_\X\qquad\forall f\in D(\Delta).
\end{equation}
The space $H^{2,2}(\X)$ is the $W^{2,2}(\X)$-closure of $D(\Delta)$ (or, equivalently, of $\test{\X}$) and, similarly, $H^{2,2}_\loc(\X)$ as the $W^{2,2}_\loc(\X)$-closure of $\test\X$, i.e.: $f\in H^{2,2}_\loc(\X)\subset W^{2,2}_\loc(\X)$ provided there exists a sequence $(f_n)\subset \test\X$ such that $f_n,\d f_n,\H {f_n}$ converge to $f,\d f,\H f$ in $L^2_\loc(\X),L^2_\loc(T^*\X),L^2_\loc((T^*)^{\otimes 2}\X)$ respectively. 

We shall also often use the identity
\[
2\H f(\nabla g,\nabla h)=\d\la\la\d f,\d g\ra,\d h\ra+\d\la\la\d f,\d h\ra,\d g\ra-\d\la\d f,\la\d g,\d h\ra\ra
\]
valid for any $f,g,h\in\test\X$.

\bigskip
The space of Sobolev vector fields $W^{1,2}_{C,\loc}(T\X)$ is defined as the space of $v\in L^2_\loc(T\X)$ for which there is $T\in L^2_\loc(T^{\otimes 2}\X)$ such that
\[
\int h T : (\nabla g \otimes \nabla\tilde g)\,\d\mm=\int -\la v,\nabla\tilde g\ra\div(h\nabla g) +h\H {\tilde g}(X,\nabla g)\,\d\mm
\]
for every $h,g,\tilde g\in\test\X$ with bounded support. In this case we call $T$ the {\bf covariant derivative} of $v$ and denote it by $\nabla_C v$. If $|v|,|\nabla_C v|_\HS\in L^2(\X)$ we shall say that $v\in W^{1,2}_C(T\X)$; this definition of $W^{1,2}_C(T\X)$ coincides with the one given in \cite{Gigli14}. Vector fields of the form $g\nabla f$ for $f,g\in\test\X$ are in $W^{1,2}_C(T\X)$: we denote by $\Test V(\X)$  the linear span of such vector fields, and by $H^{1,2}_{C}(T\X)$ its $W^{1,2}_C$-closure . The space $H^{1,2}_{C,\loc}(T\X)$ is then equivalently defined either as the subspace of $L^2_\loc(T\X)$ made of vectors $v$ of such that $fv\in H^{1,2}_C(T\X)$ for every $f\in\test\X$ with bounded support or as the $W^{1,2}_{C,\loc}$-closure of $H^{1,2}_C(T\X)$, i.e.\ as the space of vector fields $v\in W^{1,2}_{C,\loc}(T\X)$ such that there is $(v_n)\subset H^{1,2}_C(T\X)$ such that  $v_n\to v$ and $\nabla_C v_n\to\nabla_C v$ in $L^2_\loc(T\X)$ and $L^2_\loc(T^{\otimes 2}\X)$ as $n\to\infty$.

A useful result we are going to use later on is the following (we refer to \cite[Theorem 3.4.2-(iv) and (v)]{Gigli14} for its proof):
\[
f \in \WT_\loc(\X) \quad \Rightarrow \quad \nabla f \in \W_{C, \loc}(\T\X), \text{ with } \nabla (\nabla f) = (\hess(f))^\sharp.
\]
Moreover, $\Test V(\X) \subset \W_C(\T \X)$ with
\[
\nabla v = \sum_i \nabla g_i \otimes \nabla f_i + g_i (\hess(f_i))^\sharp, \qquad \text{for} \,\, v = \sum_i g_i \nabla f_i.
\]
\medskip

\subsubsection{Product of $\RCD$ spaces}

We start recalling  the following fact, which allows to use the results obtained in Sections \ref{Sec:CalcPr} and \ref{sec:CalcPr2} in the case of interest for us:
\begin{proposition}\label{prop:rcdass}
Let $\X,\Y$ be two $\RCD(K,\infty)$ spaces. Then they satisfy Assumption \ref{ass} and the product $\X\times\Y$ is $\RCD(K,\infty)$ as well.
\end{proposition}
\begin{proof}
See \cite[Theorem 5.1]{AmbrosioGigliSavare12}.
\end{proof}
A direct consequence of the tensorization of the Cheeger energy is that the heat flow tensorizes as well. In what follows we will denote by $\h_t^\X,\h_t^\Y,\h_t$ the heat flows on $\X,\Y,\X\times\Y$ respectively. We will also apply the operator $\h_t^\X$ to functions in two variables: in this case we mean that `$y$' variable is `frozen', i.e.:
\[
\h_t^\X f(x,y):=\h_t^\X (f(\cdot, y))(x)\qquad\mm_\X\otimes\mm_\Y-a.e.\ x,y.
\]
We then have:
\begin{corollary}\label{cor:th}
Let $\X,\Y$ be two $\RCD(K,\infty)$ spaces. Then for every $f\in L^2(\X\times\Y)$ we have
\begin{equation}
\label{eq:tenscal}
\h_tf=\h_t^\X(\h_t^\Y f)=\h_t^\Y(\h_t^\X f).
\end{equation}
In particular, for $f=f_1\otimes f_2$ with $f_1\in L^2(\X)$ and $f_2\in L^2(\Y)$ we have
\begin{equation}
\label{eq:tenscal2}
\h_tf=\h_t^\X f_1\otimes \h_t^\Y f_2.
\end{equation}
\end{corollary}
\begin{proof}
See {\cite[Section 6.4]{AmbrosioGigliSavare11-2}} and {\cite[Section 5.1]{AmbrosioGigliSavare12}}.
\end{proof}
\subsubsection{Setting up the problem}

Let us fix two $\RCD(K,\infty)$ spaces $\X,\Y$. Much like in the previous chapter we shall put subscripts $\sx,\sy$ to differentiation operators to specify that they act on the $x,y$ variable respectively. Thus, for instance,  with some abuse of notation we shall denote by $\H_\sx(f)$ both the Hessian of a function $f\in W^{2,2}(\X)$ and the map $y\mapsto \H_\sx(f(\cdot,y))$, if $f$ is defined on $\X\times\Y$ and is such that $f(\cdot,y)\in W^{2,2}(\X)$ for $\mm_\Y$-a.e.\ $y\in\Y$. It may also occur that for $f\in  W^{2,2}(\X)$ we write $\H_\sx(f)$ to denote the constant map $y\mapsto \H_\sx(f)\in L^2((T^*)^{\otimes 2}\X)$: this choice will alleviate the notational burden, and it will be clear from the context when we will be doing so.

\bigskip

With this said,  in Section \ref{sec:Hess} we prove that the Hessian of a function $f$ on $\X\times\Y$ admits a decomposition of the form
\[
\H (f)=\Bigg(
\begin{array}{cc}
\H_\sx (f)&\d_\sx\d_\sy f\\
\d_\sy\d_\sx f&\H_\sy (f)
\end{array}
\Bigg)
\]
meaning that the right hand side is a well defined tensor in $L^2$ if and only if so it is the left hand side, and similarly in Section \ref{sec:CovDer} we prove that the Covariant Derivative of a vector field $v = (v_{\sx}, v_{\sy})$ in the tangent module over $\X \times \Y$ admits a decomposition of the form
\[
\nabla_C v  =\Bigg(
\begin{array}{cc}
\nabla_{C, \sx} v_{\sx} &\quad\d_\sx v_{\sy}\\
\d_\sy v_\sx&\quad \nabla_{C, \sy} v_{\sy}
\end{array}
\Bigg).
\] 
\black
To make this rigorous we start from Theorem \ref{dectang} and notice that $\Phi_\sx\colon L^0(\Y;L^0(\T^*\X))\to L^0(T^*(\X\times\Y))$ induces a map,  denoted by $\Phi_\sx^{\otimes 2}$, from $L^0(\Y;L^0(\T^*\X))\otimes L^0(\Y;L^0(\T^*\X))$ to $L^0(T^*(\X\times\Y))\otimes L^0(T^*(\X\times\Y))=L^0((T^*)^{\otimes 2}(\X\times\Y))$ by acting component-wise. More precisely, the $L^0(\X\times\Y)$-linear extension of the map
\[
\omega_1\otimes\omega_2\quad\mapsto\quad \Phi_\sx(\omega_1)\otimes \Phi_\sx(\omega_2)\qquad\forall \omega_1,\omega_2\in L^0(\Y;L^0(\T^*\X))
\]
(which is easily seen to be well defined) can be extended by continuity to a norm-preserving map. This is a consequence of the fact that $\Phi_\sx$ itself is norm-preserving (see Proposition \ref{prop:defPhi}), so that  we have
\[
\begin{split}
\Big|\sum_i\Phi_\sx(\omega_{i,1})\otimes\Phi_\sx(\omega_{i,2})\Big|^2&=\sum_{i,j}\la\Phi_\sx(\omega_{i,1}),\Phi_\sx(\omega_{j,1})\ra\la\Phi_\sx(\omega_{i,2}),\Phi_\sx(\omega_{j,2})\ra\\
&=\sum_{i,j}\la \omega_{i,1}, \omega_{j,1}\ra\la \omega_{i,2}, \omega_{j,2}\ra=\Big|\sum_i\omega_{i,1}\otimes\omega_{i,2}\Big|^2
\end{split}
\]
for any finite choice of $\omega_{i,1},\omega_{i,2}\in L^0(\Y;L^0(\T^*\X))$.

In a similar way we can define the maps 
\[
\begin{split}
\Phi_\sy^{\otimes 2}&:L^0(\X;L^0(\T^*\Y))\otimes L^0(\X;L^0(\T^*\Y))\quad \to\quad L^0((T^*)^{\otimes 2}(\X\times\Y))\\
\Phi_\sx\otimes\Phi_\sy&:L^0(\Y;L^0(\T^*\X))\otimes L^0(\X;L^0(\T^*\Y))\quad \to\quad L^0((T^*)^{\otimes 2}(\X\times\Y))\\
\Phi_\sy\otimes\Phi_\sx&:L^0(\X;L^0(\T^*\Y))\otimes L^0(\Y;L^0(\T^*\X))\quad \to\quad L^0((T^*)^{\otimes 2}(\X\times\Y))
\end{split}
\]
and prove that they are $L^0(\X\times\Y)$-linear and norm preserving. 

Then, from the fact that $\Phi_\sx\oplus\Phi_\sy$ is an isomorphism of modules, we deduce that 
\[
(\Phi_\sx\oplus\Phi_\sy)^{\otimes 2}=(\Phi_\sx^{\otimes 2})\oplus(\Phi_\sx\otimes\Phi_\sy)\oplus(\Phi_\sy\otimes\Phi_\sx)\oplus(\Phi^{\otimes 2}_\sy)
\] is also an isomorphism of modules, in the sense made precise by the following statement:
\begin{proposition}
Let $\X,\Y$ be $\RCD(K,\infty)$ spaces. Then, with the notation introduced above, the following holds.

For any $A\in  L^0((T^*)^{\otimes 2}(\X\times\Y))$ there are unique $A_{\sx\sx},A_{\sx\sy},A_{\sy\sx},A_{\sy\sy}\in  L^0((T^*)^{\otimes 2}(\X\times\Y))$ in the images of $\Phi_\sx^{\otimes 2}, \Phi_\sx\otimes\Phi_\sy,\Phi_\sy\otimes\Phi_\sx,\Phi^{\otimes 2}_\sy$, respectively, such that
\begin{equation}
\label{eq:decA}
A=A_{\sx\sx}+A_{\sx\sy}+A_{\sy\sx}+A_{\sy\sy}.
\end{equation}
Moreover, we have $\la A_{ij},A_{i'j'}\ra=0$ whenever $(i,j)\neq (i',j')$ and thus
\begin{equation}
\label{eq:decA2}
|A|^2=|A_{\sx\sx}|^2+|A_{\sx\sy}|^2+|A_{\sy\sx}|^2+|A_{\sy\sy}|^2\qquad\mm_\X\otimes\mm_\Y\text{-a.e.}.
\end{equation}
\end{proposition}
\begin{proof} We start proving that the images of $\Phi_\sx^{\otimes 2}$ and $\Phi_\sx\otimes\Phi_\sy$ are pointwise orthogonal. Since we  know that these maps are continuous, it is sufficient to check that $\la\Phi_\sx^{\otimes 2}(V), \Phi_\sx\otimes\Phi_\sy(W)\ra=0$ for arbitrary $V,W$ in dense subsets of the respective source spaces. We thus pick $V=\sum_iv_{1,i}\otimes v_{2,i}$ and $W=\sum_jv_{3,j}\otimes w_j$ with $v_{1,i},v_{2,i},v_{3,j}\in L^0(\Y;L^0(\T^*\X))$ and $w_j\in L^0(\X;L^0(\T^*\Y))$. Then
\[
\begin{split}
\la\Phi_\sx^{\otimes 2}(V), \Phi_\sx\otimes\Phi_\sy(W)\ra&=\sum_{i,j}\la\Phi_\sx(v_{1,i}){\otimes }\Phi_\sx(v_{2,i}), \Phi_\sx(v_{3,j})\otimes\Phi_\sy(w_{j})\ra\\
&=\sum_{i,j}\la\Phi_\sx(v_{1,i}),\Phi_\sx(v_{3,j})\ra\underbrace{\la \Phi_\sx(v_{2,i}),\Phi_\sy(w_{j})\ra}_{\equiv 0\atop\text{ by \eqref{eq:orto}}}=0
\end{split}
\]
as desired. Similarly for other couple of images. This is enough to prove that whenever \eqref{eq:decA} holds, the identity \eqref{eq:decA2} holds as well and in turn this proves uniqueness of the $A_{ij}$'s as in the statement.

Thus it remains to prove the existence of $A_{11},A_{12},A_{21}$, and $A_{22}$. Since \eqref{eq:decA2} shows that the maps $A\mapsto A_{ij}$, $i,j\in\{\sx,\sy\}$, are continuous, we can validate such existence just for $A$'s running in a dense subset of $L^0((T^*)^{\otimes 2}(\X\times\Y))$. We thus pick $A$ of the form $A=\sum_i\omega_i\otimes\omega_i'$ for $\omega_i,\omega_i'\in L^0(T^*(\X\times\Y))$, and recall that, from the surjectivity part of the statement of Theorem \ref{dectang}, for any $i$ we have $\omega_i=\Phi_\sx(\omega_{i,\sx})+\Phi_\sy(\omega_{i,\sy})$, and similarly $\omega_i'=\Phi_\sx(\omega'_{i,\sx})+\Phi_\sy(\omega'_{i,\sy})$, for appropriate $\omega_{i,\sx},\omega'_{i,\sx}\in  L^0(\Y;L^0(\T^*\X))$ and $\omega_{i,\sy},\omega'_{i,\sy}\in  L^0(\X;L^0(\T^*\Y))$.% (in the notation as in xxx this reads as $\omega_i=(\omega_{i,\sx},\omega_{i,\sy})$ and $\omega_i=(\omega'_{i,\sx},\omega'_{i,\sy})$).

It is then clear from the definitions that with the choices
\[
\begin{array}{rll}
&A_{\sx\sx}:=\Phi_\sx^{\otimes 2}\Big(\sum_i\omega_{i,\sx}\otimes\omega'_{i,\sx}\Big),\qquad\qquad\qquad &A_{\sx\sy}:=\Phi_\sx\otimes\Phi_\sy\Big(\sum_i\omega_{i,\sx}\otimes\omega'_{i,\sy}\Big),\\
&A_{\sy\sx}:=\Phi_\sy\otimes \Phi_\sx\Big(\sum_i\omega_{i,\sy}\otimes\omega'_{i,\sx}\Big),\qquad\qquad\qquad  &A_{\sy\sy}:=\Phi_\sy^{\otimes 2}\Big(\sum_i\omega_{i,\sy}\otimes\omega'_{i,\sy}\Big),\\
\end{array}
\]
the equality \eqref{eq:decA} holds, thus giving the conclusion.
\end{proof}

\subsection{Hessian on the product space}\label{sec:Hess}
\subsubsection{From regularity on the factors to regularity on the product}
In this section we prove that if a function of two variables is such that the functions obtaining by freezing one variable are in $W^{2,2}$ and with appropriate integrability assumptions, then the function itself belongs to $W^{2,2}(\X\times\Y)$, see Theorem \ref{prop:dafatt} for the precise statement (and Proposition \ref{prop:dafatt2} for a suboptimal version about the space $H^{2,2}(\X\times\Y)$).

It will be useful to consider the following algebra of functions:
\begin{definition}[The space $\Test_\times\fct(\X\times\Y)$]
Let $\X,\Y$ be two $\RCD(K,\infty)$ spaces. 

Then $\Test_\times\fct(\X\times\Y)$ is the space of finite linear combinations of functions of the form $g_1\otimes g_2$, with $g_1\in\test{\X}$ and $g_2\in\test{\Y}$.
\end{definition}

From the tensorization of the Cheeger energy in Definition \ref{def:tensch} (which holds in this setting because of Proposition \ref{prop:rcdass}), and the tensorization of the Laplacian given in  Corollary \ref{cor:lapprod}, it is easy to see that $\Test_\times(\X\times\Y)\subset\test{\X\times\Y}$. This newly defined space of functions will be useful in proving the main result of this section, Theorem \ref{prop:dafatt}, because, as we will see in Proposition \ref{le:w22tp}, to check whether a function is in $W^{2,2}(\X\times\Y)$ it is sufficient to verify the defining property  of the Hessian only for functions in the smaller space $\Test_\times(\X\times\Y)$, where computations are easier. In turn, this latter fact will be a direct consequence of the following density result:

\begin{lemma}\label{le:apprtestx}
Let $\X,\Y$ be two $\RCD(K,\infty)$ spaces and $g\in L^2\cap L^\infty({\X\times\Y})$ (this in particular holds if $g\in \test{X\times\Y}$). Then we can find  functions $g_{n,t}\in \Test_\times\fct(\X\times\Y)$, parametrized by $n\in\N$ and $t > 0$, such that for every $t>0$ it holds:
\begin{itemize}
\item[i)] $\sup_{n}\|g_{n,t}\|_{L^\infty}<\infty$,
\item[ii)] $\sup_n\Lip(g_{n,t})<\infty$,
\item[iii)] $(g_{n,t}),(\Delta g_{n,t})$ converge to $\h_tg,\Delta\h_t g$ respectively in $L^2(\X\times\Y)$ as  $n\to\infty$,
\item[iv)] $(\nabla g_{n,t}),(\nabla(|\nabla g_{n,t}|^2))$ converge to $\nabla \h_t g,\nabla(|\nabla\h_t g|^2)$ respectively in $L^2(T(\X\times\Y))$ as  $n\to\infty$.
\end{itemize}
\end{lemma}
\begin{proof} It is easy to see that we can find a sequence of equibounded functions  $(g_n)$ $L^2$-converging to $g$ of the  form $g_n=\sum_{i}^{N_n}\alpha_i \nchi_{A_{n,i} \times B_{n,i}}$, where the sets $A_i \subset \X$ and $B_i \subset \Y$ are Borel and bounded and $\alpha_{i} \in \real$. We then put $g_{n,t}:=\h_t(g_n)$ and notice that Corollary \ref{cor:th}  and \eqref{eq:l2test} ensures that $g_{n,t}\in \Test_\times\fct(\X\times\Y)$ for any $n\in\N$, $t>0$.

Then property $(i)$ follows from \eqref{eq:maxh} and $(ii)$ from \eqref{eq:linftylip}, \eqref{eq:SobtoLip} and $(i)$. The first convergence in $(iii)$ follows from the continuity of the heat flow as a map from $L^2$ into itself and the second by also taking into account the a priori estimate \eqref{apriori2}. The fact that $\nabla g_{n,t}\to\nabla \h_tg$ in $L^2(T(\X\times\Y))$ follows from \eqref{apriori1}, which ensures that $\h_t:L^2\to W^{1,2}$ is a continuous operator. 

It remains to prove $L^2$-convergence of $(\nabla(|\nabla g_{n,t}|^2))$ to $\nabla(|\nabla \h_tg|^2)$. To see this, start observing that $\nabla(|\nabla f|^2)=2\H( f)(\nabla f)$ for any $f\in\test{\X\times\Y}$, and notice that what already proved together with  \eqref{eq:hesscont} show that $\H(g_{n,t})\to \H(\h_tg)$ in $L^2((T^*)^{\otimes 2}(\X\times\Y))$ as $n\to\infty$ for every $t>0$.

Now put $g_t:=\h_tg$ and observe that
\[
\begin{split}
&\|\H(g_{n,t})(\nabla g_{n,t})-\H(g_t)(\nabla g_t)\|_{L^2}\\
&\leq\|\H(g_{n,t}-g_t)(\nabla g_{n,t}-\nabla g_t)\|_{L^2}+\|\H(g_t)(\nabla g_{n,t}-\nabla g_t)\|_{L^2}+\|\H(g_{n,t}-g_t)(\nabla g_t)\|_{L^2}.
\end{split}
\]
As $n\to\infty$ we see from $(ii)$ and the $L^2$-convergence of Hessians that the first and third addends in the right hand side go to 0. For the second we use  the $L^2$-convergence of gradients and the fact that they are equibounded to conclude with the dominated convergence theorem.
\end{proof}

\begin{proposition}\label{le:w22tp}
Let $\X,\Y$ be two $\RCD(K,\infty)$ spaces,  $f\in \W (\X\times\Y)$ and $A\in L^2((T^*)^{\otimes 2}(\X\times\Y))$ symmetric.

Then $f\in W^{2,2}(\X\times\Y)$ with $\H (f) = A$ if and only if we have
\begin{equation}
\label{eq:testh}
-\int \la\d f,\d g\ra \,\div(h\nabla g)+\tfrac12h\langle\d f,\d {|\d g|^2}\rangle\,\d(\mm_\X\times\mm_Y)=\int h \la A,\d g\otimes d g\ra\,\d(\mm_\X\times\mm_Y)
\end{equation}
for every $g,h\in\Test_{\times}\fct(\X\times\Y)$.
\end{proposition}
\begin{proof}
The `only if' follows from the definitions of $W^{2,2}$ and Hessian and the inclusion  $\Test_\times\fct(\X\times\Y)\subset \test{\X\times\Y}$. For the `if' we start observing that from the symmetry in $g,\tilde g$ of the definition of Hessian \eqref{eq:defhess} and the fact that $A$ is symmetric we deduce that it is sufficient to check that \eqref{eq:testh} holds for $g,h\in \test{\X\times\Y}$.

Thus, fix such $g,h$ and apply Lemma \ref{le:apprtestx} to find corresponding functions $g_{n,t},h_{n,t}$ as in the statement. Hence we know that \eqref{eq:testh} holds with $g_{n,t},h_{n,t}$ in place of $g,h$ respectively and letting $n\to\infty$ in such identity and using the conclusions of Lemma \ref{le:apprtestx} it is immediate to check that \eqref{eq:testh} holds also with  $\h_tg,\h_th$ in place of $g,h$ respectively. It is now easy to see that we can pass to the limit as $t\downarrow0$ in the resulting identity  to conclude that \eqref{eq:testh} holds for the $g,h\in\test{\X\times\Y}$ initially chosen, as desired.
\end{proof}

\begin{theorem}\label{prop:dafatt} Let $\X,\Y$ be two $\RCD(K,\infty)$ spaces and let $f\in \W (\X\times\Y)$ be such that:
\begin{itemize}
\item[i)] for $\mm_\X$-a.e.\ $x\in\X$ the map $\Y\ni y\mapsto f(x,y)$ belongs to $W^{2,2}(\Y)$ with $\iint|\H_\sy (f)|^2_\HS\,\d\mm_\Y\,\d\mm_\X<\infty$,
\item[ii)] for $\mm_\Y$-a.e.\ $y\in\Y$ the map $\X\ni x\mapsto f(x,y)$ belongs to $W^{2,2}(\X)$ with $\iint|\H_\sx (f)|^2_\HS\,\d\mm_\X\,\d\mm_\Y<\infty$,
\item[iii)] we have $\d_\sy f\in W^{1,2}(\X;L^2(T^*\Y))$ (and thus, by Theorem \ref{thm:swap}, also $\d_\sx f\in W^{1,2}(\Y;L^2(T^*\X))$).
\end{itemize}
Then $f\in W^{2,2}(\X\times\ Y)$ with 
\begin{equation}
\label{eq:hprod}
\H(f)=\Phi^{\otimes 2}_\sx(\H_\sx (f))+\Phi^{\otimes 2}_\sy(\H_\sy (f))+\Phi_\sx\otimes\Phi_\sy(\d_\sx\d_\sy f)+\Phi_\sy\otimes\Phi_\sx(\d_\sy\d_\sx f)
\end{equation}
i.e.\ for every $v,w\in L^0(T^*(\X\times \Y))$, writing $v=(v_\sx,v_\sy)$ and $w=(w_\sx,w_\sy)$ (recall \eqref{eq:decv}) we have
\begin{equation}
\label{eq:hprod2}
\H(f)(v,w)=\H_\sx(f)(v_\sx,w_\sx)+\H_\sy(f)(v_\sy,w_\sy)+\la\d_\sx\d_\sy f,v_\sx\otimes w_\sy\ra+\la \d_\sy\d_\sx f,v_\sy\otimes w_\sx\ra
\end{equation}
\end{theorem}
\begin{proof} Thanks to the assumptions the right hand side of \eqref{eq:hprod} defines a symmetric tensor in $L^2((T^*)^{\otimes 2}(\X\times\Y))$. Thus by the very definition of $W^{2,2}(\X\times\Y)$ and of Hessian, Proposition  \ref{le:w22tp} above and keeping in mind the equivalent version \eqref{eq:hprod2} of \eqref{eq:hprod}, to conclude it is sufficient to prove  that: for $g_1,h_1$ (resp.\ $g_2,h_2$) test functions on $\X$ (resp.\ $\Y$) and putting $g:=g_1\otimes g_2$, $h:=h_1\otimes h_2$ we have
\begin{equation}
\label{eq:claimh}
\begin{split}
-\int \la\d f,\d g\ra \,\div(h\nabla g)+\tfrac12h\langle\d f,\d {|\d g|^2}\rangle\,\d(\mm_\X&\times\mm_\Y)\\
=\int  h\Big(\H_\sx (f)(\nabla_\sx g,\nabla_\sx g)+&\H_\sy (f)(\nabla_\sy g,\nabla_\sy g)\\
+\la\d_\sy\d_\sx f,\d_\sy g\otimes\d_\sx g\ra&+ \la\d_\sx\d_\sy f,\d_\sy g\otimes\d_\sx g\ra\Big)\,\d(\mm_\X\otimes\mm_\Y).
\end{split}
\end{equation}
We thus concentrate into proving this. From $\d g = (g_2\,\d_\sx g_1,g_1\,\d_\sy g_2)$ we get
\[
\la\d f,\d g\ra=g_2\la\d_\sx f,\d_\sx g_1\ra+g_1\la\d_\sy f,\d_\sy g_2\ra.
\]
Also, for any $h\in \Test(\X\times \Y)$ from Proposition \ref{divprod} we have
\[
\div(h\nabla g)=\div_\sx(h\nabla_\sx g)+\div_\sy(h\nabla_\sy g)=g_2\div_\sx(h\nabla_\sx g_1)+g_1\div_\sy(h\nabla_\sy g_2)
\]
and thus
\begin{equation}
\label{eq:p1}
\begin{split}
\la\d f,\d g\ra \,\div(h\nabla g)=&\underbrace{|g_2|^2\la\d_\sx f,\d_\sx g_1\ra\div_\sx(h\nabla_\sx g_1)}_A+\underbrace{g_1g_2\la\d_\sx f,\d_\sx g_1\ra\div_\sy(h\nabla_\sy g_2)}_B\\
&+\underbrace{g_1g_2\la\d_\sy f,\d_\sy g_2\ra\div_\sx(h\nabla_\sx g_1)}_C+\underbrace{|g_1|^2\la\d_\sy f,\d_\sy g_2\ra\div_\sy(h\nabla_\sy g_2)}_D.
\end{split}
\end{equation}
Now notice that $ |\d g|^2=|g_2|^2|\d_\sx g_1|^2+|g_1|^2|\d_\sy g_2|^2$, therefore 
\[
\d \frac{|\d g|^2}2=\big(|g_2|^2\d_\sx\frac{|\d_\sx g_1|^2}{2}+g_1\d_\sx g_1|\d_\sy g_2|^2,g_2\d_\sy g_2|\d_\sx g_1|^2+|g_1|^2\d_\sy\frac{|\d_\sy g_2|^2}{2}\big)
\]
and
\begin{equation}
\label{eq:p2}
\begin{split}
h\la\d f,\d \frac{|\d g|^2}2\ra=&\underbrace{h|g_2|^2\la\d_\sx f,\d_\sx\frac{|\d_\sx g_1|^2}{2}\ra}_{A'}+\underbrace{hg_1|\d_\sy g_2|^2\la\d_\sx f,\d_\sx g_1\ra}_{B'}\\
&+\underbrace{hg_2|\d_\sx g_1|^2\la\d_\sy f,\d_\sy g_2\ra}_{C'}+\underbrace{h|g_1|^2\la\d_\sy f,\d_\sy\frac{|\d_\sy g_2|^2}{2}\ra}_{D'}.
\end{split}
\end{equation}
Now observe that, from assumption $(i)$ and the integrability properties of $g,h$, it follows that 
\[
\begin{split}
-\iint A+A'\,\d\mm_\X\,\d\mm_\Y&=\int|g_2|^2\int h\H_\sx (f)(\nabla_\sx g_1,\nabla_\sx g_1)\,\d\mm_\X \,\d\mm_\Y\\
&=\int h\H_\sx (f)(\nabla_\sx g,\nabla_\sx g)\,\d(\mm_\X \otimes\mm_\Y).
\end{split}
\]
Similarly from assumption $(ii)$  we obtain
\[
-\iint D+D'\,\d\mm_\X\,\d\mm_\Y=\int h\H_\sy (f)(\nabla_\sy g,\nabla_\sy g)\,\d(\mm_\X \otimes\mm_\Y).
\]
Also, assumption $(iii)$ and the chain rule in point $(iv)$ of Proposition \ref{thm:closured} with $\X,\Y$ swapped, $\Hil=L^0(T^*\X)$ and $T:L^0(T^*\X)\to L^0(\X)$ given by $T(\cdot)=\la\cdot,\d_\sx g_1\ra$  ensure that $ \la\d_\sx f,\d_\sx g_1\ra\in W^{1,2}(\Y;L^2(\X))$ with
\begin{equation}
\label{eq:dxdy2}
\d_\sy  \la\d_\sx f,\d_\sx g_1\ra= \la\d_\sy\d_\sx f,\d_\sx g_1\ra.
\end{equation}
In particular, for $\mm_\X$-a.e.\ $x\in\X$ the function $ g_2(\cdot)\la\d_\sx f,\d_\sx g_1\ra(x,\cdot)$ is in $W^{1,2}(\Y)$ and  the following integration by parts is justified:
\[
\begin{split}
-\iint B\,\d\mm_\Y\d\mm_\X&=\int g_1\int h\la \d_\sy(g_2\la\d_\sx f,\d_\sx g_1\ra),\d_\sy g_2\ra\,\d\mm_\Y\d\mm_\X\\
&=\iint hg_1|\d_\sy g_2|^2\la\d_\sx f,\d_\sx g_1\ra+hg_1g_2\underbrace{\la\d_\sy\la\d_\sx f,\d_\sx g_1\ra,\d_\sy g_2\ra}_{=\la\d_\sy\d_\sx f,\d_\sy g_2\otimes\d_\sx g_1\ra \text{ by } \eqref{eq:dxdy2}}\,\d(\mm_\X\otimes\mm_\Y). 
\end{split}
\]
Hence we have
\[
-\iint B+B'\,\d\mm_\Y\d\mm_\X=\int h\la\d_\sy\d_\sx f,\d_\sy g\otimes\d_\sx g\ra \,\d(\mm_\X\otimes\mm_\Y)
\]
and analogously $-\iint C+C'\,\d\mm_\X\d\mm_\Y=\int h\la\d_\sx\d_\sy f,\d_\sy g\otimes\d_\sx g\ra \,\d(\mm_\X\otimes\mm_\Y).
$
Collecting what proved so far and recalling \eqref{eq:p1} and \eqref{eq:p2} we obtain \eqref{eq:claimh} and the conclusion.
\end{proof}
One might wonder if improving the requirements $(i),(ii)$ in this last statement by asking that the sections belong to the corresponding $H^{2,2}$ space is enough to conclude that $f\in H^{2,2}(\X\times\Y)$. We don't know if this is the case, but we can point out the following simple result (which is independent from Theorem \ref{prop:dafatt} above). Notice that it is not necessary to assume the existence of `mixed' derivatives.
\begin{proposition}\label{prop:dafatt2}
 Let $\X,\Y$ be two $\RCD(K,\infty)$ spaces and let $f\in \W (\X\times\Y)$ be such that:
\begin{itemize}
\item[i')]  for $\mm_\Y$-a.e.\ $y\in\Y$ we have $f(\cdot,y)\in D(\Delta_\sx)$ and $\iint |\Delta_\sx f|^2\,\d\mm_\Y\,\d\mm_\X<\infty$,
\item[ii')]  for $\mm_\X$-a.e.\ $x\in\X$ we have $f(x,\cdot)\in D(\Delta_\sy)$ and $\iint |\Delta_\sy f|^2\,\d\mm_\X\,\d\mm_\Y<\infty$.
\end{itemize}
Then $f\in H^{2,2}(\X\times\Y)$ and the identities \eqref{eq:hprod} and \eqref{eq:hprod2} hold.
\end{proposition}
\begin{proof}
This is a direct consequence of Corollary \ref{cor:lapprod} and the fact that $H^{2,2}$ is defined as the $W^{2,2}$-closure of $D(\Delta)$, and thus it contains the latter.
\end{proof}

\subsubsection{From regularity on the product to regularity on the factors}
In this section we investigate the validity of the converse implication w.r.t.\ the one proved before, namely we study what the assumption $f\in W^{2,2}(\X\times\Y),H^{2,2}(\X\times\Y)$ implies in terms of regularity of $f(\cdot,y),f(x,\cdot)$.

We start with the following result, which is similar in spirit to Propositions \ref{prop:l2div}, \ref{prop:perdxdy}.
\begin{proposition}\label{prop:hessx}
Let $(\X,\sfd_\X,\mm_\X)$ and $(\Y,\sfd_\Y,\mm_\Y)$ be $\RCD(K,\infty)$ spaces and $f\in W^{1,2}(\X\times\Y)$. Then the following are equivalent:
\begin{itemize}
\item[i)] for $\mm_\Y$-a.e.\ $y\in\Y$ we have $f(\cdot,y)\in W^{2,2}(\X)$ with $\iint|\H_\sx (f)|^2_\HS\,\d\mm_\X\,\d\mm_\Y<\infty$,
\item[ii)] there is $A\in L^2(\Y;L^2((T^*)^{\otimes 2}\X))$  such that the identity
\[
\begin{split}
2\int hA(\nabla_\sx g,\nabla_\sx g')\,\d(\mm_\X&\otimes\mm_\Y)=\int -\la\d_\sx f,\d_\sx g\ra\div(h\nabla(g'\circ\pi_\X))\\
\qquad\qquad&-\la\d_\sx f,\d_\sx g'\ra\div(h\nabla(g\circ\pi_\X))-h\langle\d_\sx f,\d_\sx\la\d_\sx g,\d_\sx g'\ra\rangle \,\d(\mm_\X\otimes\mm_\Y)
\end{split}
\] 
holds for every $g,g'\in \test{\X}$ and $h\in\Lip_\bs(\X\times\Y)$.
\end{itemize}
Moreover, if this holds then the choice $A=\H_\sx (f)$ is the only one for which $(ii)$ holds.
\end{proposition}
\begin{proof}\ \\
\noindent{$(i)\Rightarrow(ii)$} The fact that $(ii)$ holds with the choice $A:=\H_\sx (f)$ is a direct consequence of the definitions of $W^{2,2}(\X)$ and of the Hessian, of the identity \eqref{eq:divdivx} and of the fact that for $h\in\Lip_\bs(\X\times\Y)$ we have that $h(\cdot,y)\in \Lip_\bs(\X)$ for $\mm_\Y$-a.e.\ $y\in\Y$. The fact that $\H_\sx (f)$ is the only choice is a consequence of the density of the linear span of elements of the form $h\d_\sx g\otimes\d_\sx g'$ in $L^2(\Y;L^2((T^*)^{\otimes 2}\X))$ for $g,g',h$ as in the statement, which in turn is a trivial consequence of the fact that differentials of test functions generate the cotangent module.

\noindent{$(ii)\Rightarrow(i)$} Recall from \cite[Theorem 3.3.2]{Gigli14} that the functional $\E_{2,\sx}\colon W^{1,2}(\X)\to[0,\infty]$ defined by $\E_{2,\sx}(f):=\frac12\int|\H_\sx(f)|_\HS^2\,\d\mm_\X$ if $f\in W^{2,2}(\X)$ and $\E_{2,\sx}(f):=+\infty$ otherwise, is convex, lower semicontinuous and satisfies
\begin{equation}
\label{eq:ide2}
\begin{split}
2\E_{2,\sx}(f)=&\sup\bigg(\sum_j\int   -\la\d_\sx f,\d_\sx g_j\ra\div_\sx(h_jh'_j\nabla_\sx g_i')-\la\d_\sx f,\d_\sx g_j'\ra\div_\sx(h_jh_j'\nabla_\sx g_j)\\
&\qquad\qquad-h_jh_j'\big\langle\d_\sx f,\d_\sx\langle\d_\sx g_j,\d_\sx g'_j\rangle\big\rangle \,\d\mm_\X\bigg)-\bigg\|\sum_j(h_j\nabla_\sx g_j)\otimes(h'_j\nabla g_j')\bigg\|^2_{L^2(\X)},
\end{split}
\end{equation}
where the sup is taken among all finite collections of functions $g_j,g'_j,h_j,h'_j\in\test{\X}$. 

Now, for every $n\in\N$, let $(A^n_i)$ be a Borel partition of $\Y$ made of at most countable sets, with $\mm_\Y(A^n_i)\in(0,\infty)$, ${\rm diam}(A^n_i)\leq \frac1n$ and so that $(A^{n+1}_i)$ is a refinement of $(A^n_i)$. Put $f^n_i:=\mm_\Y(A^n_i)^{-1}\int_{A^n_i}f(\cdot,y)\,\d\mm_\Y(y)\in L^2(\X)$ and $f^n(x,y):=\sum_i\nchi_{A^n_i}(y)f^n_i(x)\in L^2(\X\times\Y)$. We have already noticed in the proof of Proposition \ref{prop:l2div} that $f^n\to f$ in $L^2(\X\times\Y)$ and the proof of the same proposition  shows that $\int \E_\sx(f^n(\cdot,y))\,\d\mm_\Y(y)\to\int \E_\sx(f(\cdot,y))\,\d\mm_\Y(y)$. Hence taking into account that $L^2(\Y;W^{1,2}(\X))$ is a Hilbert space (because $W^{1,2}(\X)$ is so) we just proved that $f_n\to f$ in $L^2(\Y;W^{1,2}(\X))$ and thus up to pass to a non-relabeled subsequence we have that $f^n(\cdot,y)\to f(\cdot,y)$ in $W^{1,2}(\X)$ for $\mm_\Y$-a.e.\ $y\in\Y$. Hence the $W^{1,2}(\X)$-lower semicontinuity of $\E_{2,\sx}$ and Fatou's lemma give that
 \begin{equation}
\label{eq:limEX2}
\int \E_{2,\sx}(f(\cdot,y))\,\d \mm_\Y(y)\leq \limi_{n\to\infty }\int \E_{2,\sx}(f^n(\cdot,y))\,\d \mm_\Y(y).
\end{equation}
Now fix $n,i$ and pick a finite family of functions $g_j,g'_j,h_j,h'_j\in\test{\X}$ in the identity \eqref{eq:ide2} written for the function $f^n_i$ in place of $f$. Then for every $j$ and $t>0$ define  $h_{j,t}\in\test{\X\times\Y}$ as
\[
h_{j,t}:=\h_t\big(\nchi_{\X\times A^n_i}(h_jh'_j)\circ\pi_\X\big)\,\eta\circ\pi_\Y\stackrel{\eqref{eq:tenscal2}}=\h^\X_t(h_jh'_j)\otimes \big(\h_t^\Y(\nchi_{A^n_i})\,\eta\big),
\]
where $\eta\in\test{\Y}$ has bounded support and is identically 1 on $A^n_i$. Then from the weak maximum principle \eqref{eq:maxh} and the Bakry-\'Emery gradient estimates \eqref{BE} it easily follows that $(|h_{j,t}-h_j|),(|\d_\sx h_{j,t}-\d_\sx h_j|)$ are uniformly bounded and converge to 0 in $L^2(\X\times\Y)$ as $t\downarrow0$. Hence the dominate convergence theorem, the identity \eqref{eq:divdivx}  and the closure of the differential give that
\begin{equation}
\label{eq:limj}
\begin{split}
\int -\la\d_\sx f,\d_\sx g_j\ra\div(h_{j,t}\nabla(g'_j\circ\pi_\X))-\la\d_\sx f,\d_\sx g_j'\ra\div(h_{j,t}&\nabla(g_j\circ\pi_\X))\\
-h_{j,t}\langle\d_\sx f,&\d_\sx\la\d_\sx g_j,\d_\sx g_j'\ra\rangle \,\d(\mm_\X\otimes\mm_\Y)\\
\to\quad\mm_\Y(A^n_i)\int   -\la\d_\sx f^n_i,\d_\sx g_j\ra\div_\sx(h_jh'_j\nabla_\sx g_i')-&\la\d_\sx f^n_i,\d_\sx g_j'\ra\div_\sx(h_jh_j'\nabla_\sx g_j)\\
-h_jh_j'\langle\d_\sx f^n_i&,\d_\sx\la\d_\sx g_j,\d_\sx g'_j\ra\rangle \,\d\mm_\X
\end{split}
\end{equation}
as $t\downarrow0$ for any $j$. 
%and similarly
%\[
%\Big\|\sum_jh_{j,m}\nabla_\sx g_j\otimes\nabla g_j'\Big\|^2_{L^2(\X\times\Y)}\quad\to\quad \mm_\Y(A^n_i)\Big\|\sum_j(h_j\nabla_\sx g_j)\otimes(h'_j\nabla g_j')\Big\|^2_{L^2(\X)}.
%\]
Similarly
\begin{equation}
\label{eq:ahj}
\int \langle A,h_{j,t}\nabla_\sx g_j\otimes\nabla_\sx g'\rangle \,\d(\mm_\X\otimes\mm_\Y)\quad\to\quad \int_{\X\times A^n_i} \langle A,(h_{j}\nabla_\sx g_j)\otimes(h_j'\nabla_\sx g_j')\rangle \,\d (\mm_\X\otimes\mm_\Y)
\end{equation}
Now observe that from assumption $(ii)$ we have that
\[
\begin{split}
\lim_{t\downarrow0}\sum_j\int -\langle\d_\sx f&,\d_\sx g_j\rangle\div(h_{j,t}\nabla(g'_j\circ\pi_\X))-\la\d_\sx f,\d_\sx g_j'\ra\div(h_{j,t}\nabla(g_j\circ\pi_\X))\\
&\qquad\qquad\qquad\qquad\qquad\qquad\qquad\qquad-h_{j,t}\langle\d_\sx f,\d_\sx\la\d_\sx g_j,\d_\sx g_j'\ra\rangle \,\d(\mm_\X\otimes\mm_\Y)\\
&=\lim_{t\downarrow0}\sum_j2\int \langle A,h_{j,t}\nabla_\sx g_j\otimes\nabla_\sx g'\rangle \,\d(\mm_\X\otimes\mm_\Y)\\
\text{by \eqref{eq:ahj}}\qquad&=2\int \langle\nchi_{\X\times A^n_i} A,\sum_j(h_{j}\nabla_\sx g_j)\otimes(h_j'\nabla_\sx g_j)'\rangle \,\d (\mm_\X\otimes\mm_\Y)\\
&\leq \int_{\X\times A^n_i}|A|_\HS^2 \,\d (\mm_\X\otimes\mm_\Y)+\mm_\Y(A^n_i)\Big\|\sum_j(h_j\nabla_\sx g_j)\otimes(h'_j\nabla g_j')\Big\|^2_{L^2(\X)},
\end{split}
\]
having used Young's inequality in the last step. This inequality, the arbitrariness of $g_j,g_j',h_j,h_j'$, \eqref{eq:limj} and \eqref{eq:ide2} give that
\[
\mm_\Y(A^n_i)\E_{2,\sx}(f^n_i)\leq \tfrac12 \int_{\X\times A^n_i}|A|_\HS^2 \,\d (\mm_\X\otimes\mm_\Y).
\]
Adding up over $i$, then letting $n\to\infty$ keeping in mind \eqref{eq:limEX2} we deduce that
\[
\int \E_{2,\sx}(f(\cdot,y))\,\d \mm_\Y(y)\leq\tfrac12\int_{\X\times \Y}|A|_\HS^2 \,\d (\mm_\X\otimes\mm_\Y),
\]
which is $(i)$.
\end{proof}

\begin{lemma}\label{le:easy}
Let $(\X,\sfd_\X,\mm_\X)$ and $(\Y,\sfd_\Y,\mm_\Y)$ be $\RCD(K,\infty)$ spaces and $f\in W^{2,2}(\X\times\Y)$. Then the identity 
\begin{equation}
\label{eq:hess}
\begin{split}
2\int& h\H(f)(\nabla g_1,\nabla g_2) \,\d(\mm_\X\otimes\mm_\Y)\\
&=-\int \la \d f,\d g_1\ra\div(h\nabla g_2)+\la \d f,\d g_2\ra\div(h\nabla g_1)-h\la \d f,\d\la\d g_1,\d g_2\ra\ra\,\d(\mm_\X\otimes\mm_\Y)
\end{split}
\end{equation}
holds for any $h\in \Lip_\bs(\X\times\Y)$ and $g_1,g_2\in \{g\circ\pi_\X:g\in\test{\X}\}\cup\{g\circ\pi_\Y:g\in\test{\Y}\}$.
\end{lemma}
\begin{proof} Suppose at first that $h\in \test{\X\times\Y}$ has bounded support and let $\eta\in\test{\X\times\Y}$ be with bounded support and such that $\supp(h)\subset\{\eta=1\}$ (see \eqref{eq:testcutoff}). Then for $g_1,g_2\in \{g\circ\pi_\X:g\in\test{\X}\}\cup\{g\circ\pi_\Y:g\in\test{\Y}\}$ we have that $\eta g_1,\eta g_2\in\test{\X\times\Y}$ and the validity of \eqref{eq:hess} follows from the very definition of Hessian and the locality property of the differential, which ensures that $\mm_\X\otimes\mm_\Y$-a.e.\ on $\supp(h)$ it holds $\d(\eta g_i)=\d g_i$, $i=1,2$, and $\d\la\d (\eta g_1),\d (\eta g_2)\ra=\d\la\d g_1,\d g_2\ra$.

Now let $h\in \Lip_\bs(\X\times\Y)$, $\eta$ as before and $(h_n)\subset \test{\X\times\Y}$  a sequence $W^{1,2}$-converging to $h$. Then  $\eta h_n\in \test{\X\times\Y}$ for every $n\in\N$ and $\eta h_n\to \eta h=h$ in $W^{1,2}(\X\times\Y)$. From the fact that \eqref{eq:hess} holds with  $\eta h_n$ and the fact that these functions have uniformly bounded support it is now easy to pass to the limit in $n$ and get the claim.
\end{proof}

\begin{lemma}\label{le:horto} Let $(\X,\sfd_\X,\mm_\X)$ and $(\Y,\sfd_\Y,\mm_\Y)$ be $\RCD(K,\infty)$ spaces, $g:\X\times\Y\to\R$ of the form $g=\tilde g\circ\pi_\X$ for some $\tilde g\in\test{\X}$. Then for every $ h=\tilde h\circ\pi_\Y$ with $\tilde h\in W^{1,2}(\Y)$ we have
\begin{equation}
\label{eq:horto}
\H (g)(\nabla h)=0\qquad\mm_\X\otimes\mm_\Y-a.e..
\end{equation}
Similarly with the roles of $\X$ and $\Y$ swapped.
\end{lemma}
\begin{proof}
Recalling that gradients of test functions on $\Y$ generate $L^2(T^*\Y)$,  we can reduce to prove \eqref{eq:horto} for $h$ as in the statement with $\tilde h\in \test{\Y}$. Then taking into account Theorem \ref{dectang} we see that to conclude it is sufficient to show that
 \begin{equation}
\label{eq:tesieq}
\begin{split}
\H (g)(\nabla h,\nabla f)&=0,\quad\mm_\X\otimes\mm_\Y-a.e.,\quad \forall f=\tilde f\circ\pi_\X\quad\text{with }\tilde f\in\test{\X},\\
\H (g)(\nabla h,\nabla  f)&=0,\quad\mm_\X\otimes\mm_\Y-a.e.,\quad \forall  f=\tilde f\circ\pi_\Y\quad\text{with }\tilde f\in\test{\Y}.
\end{split}
\end{equation}
We start with the first of the above and notice that
\[
\begin{split}
2\H (g)(\nabla h,\nabla f)&=\langle\d\underbrace{\la\d g,\d h\ra}_{\equiv0\atop\text{ by \eqref{eq:orto}}},\d f\rangle+\la\d\la\d g,\d f\ra,\d h\ra-\langle\d g,\d\underbrace{\la\d h,\d f\ra}_{\equiv 0\atop\text{ by \eqref{eq:orto}}}\rangle,
\end{split}
\]
while for the middle term we notice that by Proposition \ref{prop:defPhi} (in particular by polarization from the fact that $\Phi_\sx$ preserves the pointwise norm) we have $\la\d g,\d f\ra=\langle\d_\sx \tilde g,\d_\sx \tilde f \rangle\circ\pi_\X$, so that this function depends only on the $x$ variable and the scalar product of its differential with that of $h$ is 0, again by  \eqref{eq:orto}.

Similarly, for the second in \eqref{eq:tesieq} we have
\[
\begin{split}
2\H (g)(\nabla h,\nabla f)&=\langle\d\underbrace{\la\d g,\d h\ra}_{\equiv0\atop\text{ by \eqref{eq:orto}}},\d \tilde f\rangle+\langle\d\underbrace{\langle\d g,\d  f\rangle}
_{\equiv 0\atop\text{ by \eqref{eq:orto}}},\d h\rangle-\langle\d g,\d{\langle\d h,\d  f\rangle}\rangle,
\end{split}
\]
and in the last addend arguing as above we see that the function ${\langle\d h,\d  f\rangle}$ depends only on $y$, while $ g$ depends only on $x$, so that  \eqref{eq:orto} ensures that it is identically 0, as claimed.
\end{proof}
We now come to the main result of this section:
\begin{theorem}\label{prop:w22prod} Let $(\X,\sfd_\X,\mm_\X)$ and $(\Y,\sfd_\Y,\mm_\Y)$ be $\RCD(K,\infty)$ spaces and $f\in W^{2,2} (\X\times\Y)$. Then 
\begin{itemize}
\item[i)] for $\mm_\Y$-a.e.\ $y\in\Y$ we have $ f(\cdot,y)\in W^{2,2}(\X)$ with $\iint|\H_\sx (f)|^2_\HS\,\d\mm_\X\,\d\mm_\Y<\infty$,
\item[ii)] for $\mm_\X$-a.e.\ $x\in\X$ we have $ f(x,\cdot)\in W^{2,2}(\Y)$ with $\iint|\H_\sy (f)|^2_\HS\,\d\mm_\Y\,\d\mm_\X<\infty$,
\item[iii)] we have $\d_\sy f\in W^{1,2}(\X;L^2(T^*\Y))$ and   $\d_\sx f\in W^{1,2}(\Y;L^2(T^*\X))$.
\end{itemize}
Also, the identities \eqref{eq:hprod} and \eqref{eq:hprod2} hold.
\end{theorem}
\begin{proof} Point $(i)$ follows by a direct application of Lemma \ref{le:easy} above with $g_1=g_2\in \{g\circ\pi_\X:g\in\test{\X}\}$ in conjunction with Proposition \ref{prop:hessx}. Similarly for $(ii)$. 

To get $(iii)$ we start claiming that for $g_1,g_2\in \{g\circ\pi_\X:g\in\test{\X}\}\cup\{g\circ\pi_\Y:g\in\test{\Y}\}$ and $h\in\test{\X\times\Y}$ with bounded support it holds
\[
\begin{split}
&2\int h\H(g_1)(\nabla f,\nabla g_2) \,\d(\mm_\X\otimes\mm_\Y)\\
&=\int- \la \d g_1,\d f\ra\div(h\nabla g_2)+h\la\d\la \d g_1,\d g_2\ra,\d f\ra +\la \d f,\d g_2\ra\div(h\nabla g_1)\,\d(\mm_\X\otimes\mm_\Y).
\end{split}
\]
Indeed, this identity holds for $f\in\test{\X\times\Y}$ by the very definition of Hessian and two integration by parts. Then the claim follows from the continuity of both sides of the identity in $f$ w.r.t.\ the $W^{1,2}$-topology. Adding up this identity with \eqref{eq:hess} we obtain
\[
\int h\H (f)(\nabla g_1,\nabla  g_2)+h \H (g_1)(\nabla f,\nabla  g_2)\,\d(\mm_\X\otimes\mm_\Y)=-\int \la \d f,\d g_1\ra\div(h\nabla g_2)\,\d(\mm_\X\otimes\mm_\Y).
\]
Here we pick $g_1=g\circ \pi_\X$ and $g_2=\tilde g\circ\pi_\Y$ with $g\in\test{\X}$ and $\tilde g\in \test{\Y}$ respectively and take into account Lemma \ref{le:horto} to deduce that
\[
\begin{split}
\int h\H (f)(\nabla( g\circ\pi_\X),\nabla ( \tilde g\circ\pi_\Y))\,\d(\mm_\X\otimes\mm_\Y)&=-\int \la \d f,\d (g\circ\pi_\X)\ra\div(h\nabla(\tilde g\circ\pi_\Y))\,\d(\mm_\X\otimes\mm_\Y)\\
\end{split}
\]
Hence recalling   Proposition \ref{prop:defPhi} (in particular the fact that $\Phi_\sx$ preserves the pointwise norm)   and denoting by $A$ the restriction of the Hessian of $f$ to $\Phi_\sx(L^2(\Y;L^2(T^*\X)))\otimes\Phi_\sy( L^2(\X;L^2(T^*\Y)))\subset L^2((T^*)^{\otimes 2}(\X\times\Y))$ we obtain
\begin{equation}
\label{eq:perdxdy}
\int h\la A,\d_\sx g\otimes \d_\sy  \tilde g\ra=-\int \la \d_\sx f,\d_\sx g \ra\div(h\nabla(\tilde g\circ\pi_\Y))\,\d(\mm_\X\otimes\mm_\Y).
\end{equation}
We have established \eqref{eq:perdxdy} for $\tilde g\in \test{\Y}$ and $h\in \test{\X\times\Y}$ with bounded support, but a simple approximation argument based on the heat flow and a multiplication with a test cut-off function with bounded support shows that it actually holds for $h\in \Lip_\bs(\X\times\Y)$ and $\tilde g\in D(\Delta_\sy)$. 

Since $\{\d_\sx g:g\in\test{\X}\}$ generates $L^2(T^*\X)$, we are now in position to apply Proposition \ref{prop:perdxdy} with $\d_\sx f$ in place of $f$ to conclude that $\d_\sx f\in W^{1,2}(\Y;L^2(T^*\X))$, as desired (with $\d_\sy\d_\sx  f=A$). The argument for $\d_\sy\d_\sx f$   is analogous.

The fact that  the identities \eqref{eq:hprod} and \eqref{eq:hprod2} hold is now a direct consequence of the proof (alternatively, now that we know $(i),(ii),(iii)$, of Proposition \ref{prop:dafatt}).
\end{proof}
If the function $f$ in the previous statement is assumed to be in $H^{2,2}(\X\times\Y)$, also the fibers can be proved to be in the respective $H^{2,2}$ spaces:
\begin{proposition}\label{prop:h22prod} Let $(\X,\sfd_\X,\mm_\X)$ and $(\Y,\sfd_\Y,\mm_\Y)$ be $\RCD(K,\infty)$ spaces and  $f\in H^{2,2} (\X\times\Y)$. 

Then in addition to the conclusions of Theorem \ref{prop:w22prod} above we  have:
\begin{itemize}
\item[i')] for $\mm_\Y$-a.e.\ $y\in\Y$ we have $f(\cdot,y)\in H^{2,2}(\X)$,
\item[ii')] for $\mm_\X$-a.e.\ $x\in\X$ we have $f(x,\cdot)\in H^{2,2}(\Y)$.
\end{itemize}
\end{proposition}
\begin{proof}
Let $(f_n)\subset \test{\X\times \Y}$ be $W^{2,2}$-converging to $f$ and notice that with a diagonalization argument based on the fact that $\h_tf_n\to f_n$ in $W^{2,2}$ as $t\downarrow0 $ (recall \eqref{eq:hesscont}), we see that $h_{t_n}f_n\to f$ in $W^{2,2}$ for some $t_n\downarrow0$.

The conclusion then follows from the following two simple facts (and their analogue with inverted variables):
\[
\text{for $\mm_\X$-a.e.\ $x\in\X$ we have $\h_{t_n}f_n(x,\cdot)\in\test{\Y}$ for every $n\in\N$}
\]
and
\[
\text{for  $\mm_\X$-a.e.\ $x\in\X$ we have $\h_{t_{n_k}}f_{n_k}(x,\cdot)\to f(x,\cdot)$ in $W^{2,2}(\Y)$ for some $n_k\uparrow\infty$.}
\]
The first follows from \eqref{eq:tenscal} and \eqref{eq:l2test}, while for the second we notice that the identities \eqref{Chpr} and \eqref{eq:decA2}, \eqref{eq:hprod} (which is valid thanks to Theorem \ref{prop:w22prod}) imply that
\[
|\d_\sy g(x,\cdot)|(y)\leq |\d g|(x,y)\qquad\text{ and }\qquad |\H_\sy g(x,\cdot)|_\HS (y)\leq|\H g|_\HS(x,y)
\]
$\mm_\X\otimes\mm_\Y$-a.e.\ $(x,y)\in \X\times\Y$. Then we apply these to $g:=f-f_n$ and use the assumption $f_n\to f$ in $W^{2,2}(\X\times\Y)$ to conclude.
\end{proof}

 \subsection{Covariant derivative on the product space}\label{sec:CovDer}

In this part of the paper we investigate the relation between Sobolev regularity of vector fields on base spaces and in the product.

\subsubsection{From regularity on the factors to regularity on the product}
The following is analogue of Proposition \ref{le:w22tp}:
\begin{proposition}\label{prop:WCtp}
Let $\X, \Y$ be two $\RCD(K, \infty)$ spaces, $v \in L^0(T(\X \times \Y))$ and $T \in L^2(T^{\otimes 2}(\X \times \Y))$. Then $v \in W^{1, 2}_C(T(\X \times \Y))$ with $\nabla v = T$ if and only if it holds
\begin{equation}\label{eq:testcov}
-\int \langle v, \nabla \tilde g \rangle \div(h \nabla g) + h \H(\tilde g)(v, \nabla g) \, \d(\mm_\X \otimes \mm_\Y) \, = \, \int h \langle T, \nabla g \otimes \nabla \tilde g \rangle \, \d(\mm_\X \otimes \mm_\Y).
\end{equation}
for every $g, \tilde g, h \in \Test_\times \fct(\X \times \Y)$.
\end{proposition}
\begin{proof}
The `only if' part follows directly from the definitions of $W^{1,2}_C$ and Covariant Derivative, and the inclusion  $\Test_\times \fct(\X\times\Y)\subset \test{\X\times\Y}$. As for the `if' part, fix $g, \tilde g, h\in \test{\X\times\Y}$, and apply Lemma \ref{le:apprtestx} to find corresponding functions $g_{n,t}, \tilde g_{n,t}, h_{n,t}$ as in the statement. Therefore, \eqref{eq:testcov} holds with $g_{n,t}, \tilde g_{n,t}, h_{n,t}$ in place of $g, \tilde g, h$, respectively. Now, letting $n\to\infty$ in such identity, and using the conclusions of Lemma \ref{le:apprtestx}, we obtain that \eqref{eq:testh} holds also with  $\h_tg, \h_t \tilde g_{n}, \h_th$ in place of $g, \tilde g, h$, respectively. In order to conclude, we pass to the limit as $t\downarrow0$ in the resulting identity and we get that \eqref{eq:testcov} actually holds for the $g, \tilde g, h\in\test{\X\times\Y}$ initially chosen, as desired.
\end{proof}
We then the following result, analogue of Theorem \ref{prop:dafatt}.
\begin{theorem}\label{prop:dafattCov} Let $\X,\Y$ be two $\RCD(K,\infty)$ spaces and let $v\in L^0 (T (\X\times\Y))$, $v=(v_\sx, v_\sy)$, be such that:
\begin{itemize}
\item[i)] for $\mm_\X$-a.e.\ $x\in\X$ we have that $v_\sy(x, \cdot) \in W^{1,2}_C(T\Y)$ with $\iint|\nabla_{C, \sy} (v_\sy)|^2_\HS\,\d\mm_\Y\,\d\mm_\X<\infty$,
\item[ii)] for $\mm_\Y$-a.e.\ $y\in\Y$ we have that $v_\sx(\cdot, y) \in W^{1,2}_C(T\X)$ with $\iint|\nabla_{C, \sx} (v_\sx)|^2_\HS\,\d\mm_\X\,\d\mm_\Y<\infty$,
\item[iii)] the map $\Y\ni y\mapsto v_\sx(\cdot, y)$ belongs to $W^{1,2}(\Y; L^2(T \X))$, and the map $\X\ni x\mapsto v_\sy(x, \cdot)$ belongs to $W^{1,2}(\X; L^2(T \Y))$.
\end{itemize}
Then $v\in W^{1,2}_C(T(\X\times\Y))$ with 
\begin{equation}
\label{eq:covprod}
\nabla_C v = \Phi^{\otimes 2}_\sx(\nabla_{C, \sx} v_\sx)+\Phi^{\otimes 2}_\sy(\nabla_{C, \sy} v_\sy)+\Phi_\sx\otimes\Phi_\sy(\d_\sx v_\sy)+\Phi_\sy\otimes\Phi_\sx(\d_\sy v_\sx)
\end{equation}
i.e.\ for every $w, z \in L^0(T(\X\times \Y))$, writing $w=(w_\sx, w_\sy)$ and $z=(z_\sx, z_\sy)$ (recall \eqref{eq:coppia}) we have
\begin{equation}
\label{eq:covprod2}
 \nabla_C v : (w\otimes z)  = \langle \nabla_{C, \sx} v_\sx, w_\sx\otimes z_\sx\rangle + \langle \nabla_{C, \sy} v_\sy, w_\sy\otimes z_\sy\rangle + \la\d_\sx v_\sy, w_\sx\otimes z_\sy\ra+\la \d_\sy v_\sx, w_\sy\otimes z_\sx\ra.
\end{equation}
\end{theorem}

\begin{proof}
First of all, we observe that the assumptions guarantee that the right hand side of \eqref{eq:covprod} defines a tensor in $L^2(T^{\otimes 2}(\X \times \Y))$. Hence, keeping in mind the very definition of $W^{1,2}_C(T(\X\times\Y))$ and of the Covariant Derivative, Proposition  \ref{prop:WCtp} above and the equivalent version \eqref{eq:covprod2} of \eqref{eq:covprod}, in order to conclude we have to prove that for $g_1, \tilde g_1, h_1$ (resp.\ $g_2, \tilde g_2, h_2$) test functions on $\X$ (resp.\ $\Y$), once we set $g:=g_1\otimes g_2$, $\tilde g := \tilde g_1\otimes \tilde g_2$, and $h := h_1\otimes h_2$, we have
\begin{equation}
\label{eq:claimcov}
\begin{split}
-\int \langle v, \nabla \tilde g \rangle \div(h \nabla g) + h \H(\tilde g)(v, \nabla g) \,\d(\mm_\X&\times\mm_\Y)\\
=\int  h\Big(\langle \nabla_{C, \sx} v_\sx, \nabla_\sx g \otimes \nabla_\sx \tilde g\rangle+& \langle \nabla_{C, \sy} v_\sy, \nabla_\sy g \otimes \nabla_\sy \tilde g\rangle\\
+\la\d_\sy v_\sx, \d_\sy g\otimes\d_\sx \tilde g\ra &+ \la\d_\sx v_\sy,\d_\sy g\otimes\d_\sx \tilde g\ra\Big)\,\d(\mm_\X\otimes\mm_\Y).
\end{split}
\end{equation}
From the fact that $\nabla \tilde g = (\tilde g_2 \nabla_\sx \tilde g_1, \tilde g_1 \nabla_\sy \tilde g_2)$, we have
\[
\la v,\nabla \tilde g\ra=\tilde g_2\la v_\sx, \nabla_\sx \tilde g_1\ra + \tilde g_1\la v_\sy,\nabla_\sy \tilde g_2\ra.
\]
Moreover, from Proposition \ref{divprod} we have that for any $h\in \Test(\X\times \Y)$ it holds
\[
\div(h\nabla g)=\div_\sx(h\nabla_\sx g)+\div_\sy(h\nabla_\sy g)=g_2\div_\sx(h\nabla_\sx g_1)+g_1\div_\sy(h\nabla_\sy g_2),
\]
and thus
\begin{equation}
\label{eq:p1cov}
\begin{split}
\la v, \nabla \tilde g\ra \,\div(h\nabla g) =& \underbrace{g_2 \tilde g_2\la v_\sx,\nabla_\sx \tilde g_1\ra\div_\sx(h\nabla_\sx g_1)}_A +\underbrace{g_1\tilde g_2\la v_\sx,\nabla_\sx \tilde g_1\ra\div_\sy(h\nabla_\sy g_2)}_B \\
& +\underbrace{\tilde g_1g_2\la v_\sy,\nabla_\sy \tilde g_2\ra\div_\sx(h\nabla_\sx g_1)}_C  + \underbrace{g_1 \tilde g_1\la v_\sy,\nabla_\sy \tilde g_2\ra\div_\sy(h\nabla_\sy g_2)}_D.
\end{split}
\end{equation}
At this point, a direct application of Theorem \ref{prop:w22prod} ensures that
\begin{equation}\label{eq:p2cov}
\begin{split}
h \H(\tilde g)(v, \nabla g) 
%=& h \big(\H_\sx (\tilde g) (X_\sx, \nabla_\sx g) + \H_\sy (\tilde g) (X_\sy, \nabla_\sy g) + \la \d_\sx \d_\sy \tilde g, X_\sx \otimes \nabla_\sy  g \ra + \la \d_\sy \d_\sx \tilde g, X_\sy \otimes \nabla_\sx  g \ra\big)\\
= & \underbrace{ h g_2 \tilde g_2 \H_\sx (\tilde g_1) (v_\sx, \nabla_\sx g_1)}_{A'} + \underbrace{ h g_1 \la v_\sx \otimes \nabla_\sy  g_2, \nabla_\sx \tilde g_1 \otimes \nabla_\sy \tilde g_2 \ra}_{B'}  \\
 & + \underbrace{ h g_2 \la v_\sy \otimes \nabla_\sx  g_1, \nabla_\sy \tilde g_2 \otimes \nabla_\sx \tilde g_1 \ra}_{C'}  + \underbrace{h g_1 \tilde g_1 \H_\sy (\tilde g_2) (v_\sy, \nabla_\sy g_2)}_{D'}.
\end{split}
\end{equation}
Now, observe that from assumption $(i)$, and the integrability properties of $g, \tilde g, h$ it follows that
\[
\begin{split}
- \int \int A + A' \, \d \mm_\sx \, \d \mm_\sy &= \int g_2 \tilde g_2 \int h \nabla_{C, \sx} v_\sx : (\nabla_\sx g_1 \otimes \nabla_\sx \tilde g_1)\, \d \mm_\sx \, \d \mm_\sy\\
&= \int  h \nabla_{C, \sx} v_\sx : (\nabla_\sx g \otimes \nabla_\sx \tilde g)\, \d (\mm_\sx \otimes \mm_\sy),
\end{split}
\]
and, similarly, from assumption $(ii)$ we obtain
\[
- \int \int D + D' \, \d \mm_\sx \, \d \mm_\sy = \int  h \nabla_{C, \sy} v_\sy : (\nabla_\sy g \otimes \nabla_\sy \tilde g)\, \d (\mm_\sx \otimes \mm_\sy).
\]
Moreover, assumption $(iii)$ and the chain rule in point $(iv)$ of Proposition \ref{thm:closured} with $\X,\Y$ swapped, $\Hil = L^0(T^*\X)$ and $T \colon L^0(T^*\X) \to L^0(\X)$ given by $T(\cdot) = \la\cdot,\d_\sx g_1\ra$  ensure that $ \la v_\sx,\d_\sx g_1\ra\in W^{1,2}(\Y;L^2(\X))$ with
\begin{equation}
\label{eq:dxdyCOV}
\d_\sy  \la v_\sx, \d_\sx g_1\ra= \la\d_\sy v_\sx,\d_\sx g_1\ra.
\end{equation}
In particular, for $\mm_\X$-a.e.\ $x\in\X$ the function $ g_2(\cdot)\la\d_\sx f,\d_\sx g_1\ra(x,\cdot)$ is in $W^{1,2}(\Y)$ and  the following integration by parts is justified:
\[
\begin{split}
-\iint B\,\d\mm_\Y\d\mm_\X&=\int g_1\int h\la \d_\sy(\tilde g_2\la v_\sx,\nabla_\sx \tilde g_1\ra),\d_\sy g_2\ra\,\d\mm_\Y\d\mm_\X\\
&=\iint hg_1 \underbrace{\la \d_\sy \tilde g_2, \d_\sy g_2\ra \la v_\sx, \nabla_\sx \tilde g_1\ra}_{=\la v_\sx \otimes \nabla_\sy \tilde g_2, \nabla_\sx \tilde g_1 \otimes \nabla_\sy g_2 \ra} + h g_1 \tilde g_2\underbrace{\la\d_\sy\la v_\sx,\d_\sx \tilde g_1\ra,\d_\sy g_2\ra}_{=\la\d_\sy v_\sx,\d_\sy g_2\otimes\d_\sx \tilde g_1\ra \text{ by } \eqref{eq:dxdyCOV}}\,\d(\mm_\X\otimes\mm_\Y). 
\end{split}
\]
Therefore we get
\[
-\iint B + B' \, \d \mm_\Y \d \mm_\X = \int h \la\d_\sy v_\sx,\d_\sy g \otimes\d_\sx \tilde g \ra \, \d (\mm_\X \otimes \mm_\Y),
\]
and similarly $-\iint C + C' \, \d \mm_\X \d \mm_\Y = \int h \la\d_\sx v_\sy,\d_\sx g \otimes\d_\sy \tilde g \ra \, \d (\mm_\X \otimes \mm_\Y)$. Summing up \eqref{eq:p1cov} and \eqref{eq:p2cov}, and keeping in mind all these identities, we obtain \eqref{eq:claimcov} and the conclusion.
\end{proof}

\subsubsection{From regularity on the product to regularity on the factors}

The following statement is the analogue of Proposition \ref{prop:hessx}.
\begin{proposition}\label{prop:covx}
Let $(\X,\sfd_\X,\mm_\X)$ and $(\Y,\sfd_\Y,\mm_\Y)$ be $\RCD(K,\infty)$ spaces, and let $v \in \Lint^2(\T(\X \times \Y))$, $v = (v_\sx, v_\sy)$.
Then the following are equivalent:
\begin{itemize}
\item[i)] for $\mm_\Y$-a.e. $y \in \Y$, we have $v_\sx(\cdot, y) \in \W_C(\T \X)$ with $\iint \abs{\nabla_{C, \X} (v_\sx)}^2_\HS \, \d \mm_\X \, \d \mm_\Y < \infty$,
\item[ii)] there exists $T \in \Lint^2(\Y; \Lint^2(\T^{\otimes 2} \X))$ such that the identity
\[
\begin{split}
\int h T : (\nabla g \otimes \nabla \tilde g) \, &\d (\mm_\X \otimes \mm_\Y)\\
& = \int - \la v_\sx, \nabla_x \tilde g \ra \div(h \nabla (g\circ \pi_\X)) - h \H(\tilde g)(v_\sx, g) \, \d (\mm_\X \otimes \mm_\Y)
\end{split}
\]
holds for every $g, \tilde g \in \test{\X}$ and $h \in \Lip_\bs(\X \times \Y)$.
\end{itemize}
Moreover, in this case the choice $T = \nabla_{C, \X} (v_\sx)$ is the only one for which the identity in $ii)$ holds.
\end{proposition}

\begin{proof}
$i) \Rightarrow ii)$ Similarly to the proof of Proposition \ref{prop:hessx}, the validity of the identity in $ii)$ with the only choice $T = \nabla_\sx v$ is guaranteed by the definition of the Covariant Derivative and the space $\W_C(\T \X)$, the identity \eqref{eq:divdivx} and the fact that $h (\cdot, y) \in \Lip_\bs(\X)$ for $\mm_\Y$-a.e. $y \in \Y$.\\
$ii) \Rightarrow i)$ We start recalling that the connection energy functional $\E_{C,\sx}\colon \Lint^2(\T \X)\to[0,\infty]$ defined by $\E_{C, \sx}(v):=\frac12\int|\nabla v|_\HS^2\,\d\mm_\X$ if $v\in W^{1,2}_C(\T \X)$ and $\E_{C,\sx}(v):=+\infty$ otherwise, is convex, lower semicontinuous and satisfies
\begin{equation}
\label{eq:ideCov}
\begin{split}
\E_{C, \sx}(v) = \sup\bigg\{\sum_j & \int   -\la v, z_j\ra\div_\sx(w_j) - \nabla z_j : (w_j \otimes v) \,\d\mm_\X - \frac12 \bigg\|\sum_j w_j \otimes z_ j\bigg\|^2_{\Lint^2(\T^{\otimes 2}\X)} \bigg\},
\end{split}
\end{equation}
where the sup is taken among all finite collections of vector fields $w_j, z_j \in \Test \V(\X)$, and over all finite collections of functions $g_{j, \ell}, h_{j, \ell} \in\test{\X}$ such that $z_\ell = \sum_j h_{j, \ell} \nabla g_{j, \ell}$ (\cite[Theorem 3.4.2]{Gigli14}). 

Hence, arguing similarly as in the proof of Proposition \ref{prop:hessx}, for each $n \in \N$ we denote by $(A_i^n)$ a Borel partition of $\Y$ made of at most countable sets such that $0 < \mm_\Y(A_i^n) < \infty$, $\text{diam}(A_i^n) \le 1/n$ and with the property that $(A^{n+1}_i)$ is a refinement of $(A^n_i)$. Then we define $v^n_i := \mm_\Y(A_i^n)^{-1} \int_{A_i^n} v_\sx \, \d \mm_\Y(y) \in L^2(T\X)$, and $v^n := \sum_i \chi_{A_i^n}(y) \Phi_\sx(v_i^n) \in L^2(T(\X \times \Y))$. It is clear that $(v^n)$ converges to $(v_\sx,0)$ in $L^2(\T(\X \times \Y))$ as $n\to\infty$, thus from the lower semicontinuity of $\E_{C,\sx}$ on $\X$ we easily get
\begin{equation}\label{eq:ECX}
\int \E_{C, \sx}(v_\sx) \, \d \mm_\Y(y) \le \limi_{n\to\infty} \int \E_{C, \sx}(v^n_\sx) \, \d \mm_\Y(y).
\end{equation}
At this point, we fix $n, i \in \N$ and pick a finite family of vector fields $w_j, z_j = \sum_\ell h_{j, \ell} \nabla g_{j, \ell} \in \Test V(\X)$ in the identity \eqref{eq:ideCov} written for $v_i^n$ in place of $v$. Thus, for every fixed $j, \ell$ and $t > 0$, we define $h_{j, \ell, t} \in \test{\X \times \Y}$ by posing
\[
h_{j, \ell, t} := \h_t\big( \chi_{\X \times A_i^n} h_{j, \ell} \circ \pi_\X \big) \eta \circ \pi_\Y \overset{\eqref{eq:tenscal2}}{=} \h_t^\X(h_{j, \ell}) \otimes \big( \h_t^\Y(\chi_{A_i^n}) \eta \big),
\]
for a function $\eta \in \test{\Y}$ with bounded support and identically equal to $1$ on $A_i^n$. Thus $(\abs{h_{j, \ell, t} - h_{j, \ell}}), (\abs{\d_\sx h_{j, \ell, t} - \d_\sx h_{j, \ell}})$ are uniformly bounded and converge to $0$ in $L^2(\X \times \Y)$ as $t \downarrow 0$. Therefore, the dominate convergence theorem, the identities \eqref{eq:divdivx} and \eqref{eq:horto}, and the closure of the differential give that, as $t \downarrow 0$,
\begin{equation}\label{eq:limjCov}
\begin{split}
%\int &-\la v_\sx, h_{j, \ell, t} \nabla (g_{j, \ell} \circ \pi_\X) \ra \div(w_j) \circ \pi_\X - (\la \nabla h_{j, \ell, t}, w_j  \ra  \la \nabla g_{j, \ell}, v_\sx \ra) \circ \pi_\X\\ &\hspace{5em} + h_{j, \ell, t} \H(g_{j, \ell} \circ \pi_\X)(w_j, v_\sx) \circ \pi_\X \, \d (\mm_\X \otimes \mm_\Y)\\
%&= 
\int - &\la v_\sx, \nabla_\sx g_{j, \ell} \ra \div (h_{j, \ell, t} w_j)  - h_{j, \ell, t} \H_\sx(g_{j, \ell})(w_j, v_\sx)\circ \pi_\X  \, \d (\mm_\X \otimes \mm_\Y) \\
&{\to} \quad \mm_\Y(A_i^n) \int -  \la v_\sx, \nabla_\sx g_{j, \ell} \ra \div_\sx (h_{j, \ell} w_j)  - h_{j, \ell} \H_\sx(g_{j, \ell})(w_j, v_\sx)  \, \d \mm_\X
\end{split}
\end{equation}
for every fixed $j, \ell$. In a similar way we get that
\begin{equation}\label{eq:Thj}
\int \la  T, h_{j, \ell, t} \nabla_\sx g_{j, \ell} \otimes w_j   \ra \, \d (\mm_\X \otimes \mm_\Y)\, \overset{t \downarrow 0}{\longrightarrow} \, \int_{\X \times A_i^n} \la  T, h_{j, \ell} \nabla_\sx g_{j, \ell} \otimes w_j   \ra \, \d (\mm_\X \otimes \mm_\Y).
\end{equation}
Moreover, directly from $ii)$, we have that
\[
\begin{split}
%\lim_{t \downarrow 0} \sum_{j,\ell}\int &-\la v_\sx, h_{j, \ell, t} \nabla (g_{j, \ell} \circ \pi_\X) \ra  \div(w_j) \circ \pi_\X - (\la \nabla h_{j, \ell, t}, w_j  \ra  \la \nabla g_{j, \ell}, v_\sx \ra) \circ \pi_\X\\ 
%&\hspace{5em} + h_{j, \ell, t} \H(g_{j, \ell} \circ \pi_\X)(w_j, v_\sx) \circ \pi_\X \, \d (\mm_\X \otimes \mm_\Y)\\
\lim_{t \downarrow 0} \sum_{j,\ell}\int - &\la v_\sx, \nabla_\sx g_{j, \ell} \ra \div (h_{j, \ell, t} w_j)  - h_{j, \ell, t} \H_\sx(g_{j, \ell})(w_j, v_\sx)\circ \pi_\X  \, \d (\mm_\X \otimes \mm_\Y) \\
&\quad= \lim_{t \downarrow 0} \int \langle  T, w_j\otimes \sum_{j, \ell} h_{j, \ell, t} \nabla_\sx g_{j, \ell}    \rangle \, \d (\mm_\X \otimes \mm_\Y)\\
&\overset{\eqref{eq:Thj}}{=}  \int \langle \chi_{\X \times A_i^n}  T,\sum_j w_j\otimes z_j\rangle \, \d (\mm_\X \otimes \mm_\Y)\\
&\quad \le \tfrac12\int_{\X \times A_i^n} \abs{T}_\HS^2 \, \d (\mm_\X \otimes \mm_\Y) + \tfrac12 \mm_\Y(A_i^n)\Big\| \sum_j w_j \otimes z_j  \Big\|^2_{L^2(\X)}
\end{split}
\]
where in the last step follows from Young's inequality. This inequality, the fact that the choice of $w_j$ and $z_j$ is arbitrary, \eqref{eq:limjCov} and \eqref{eq:ideCov}  provide
\[
\mm_\Y(A_i^n) \E_{C, \sx}(v_i^n) \le\tfrac12 \int_{\X \times A_i^n} \abs{T}^2_\HS \, \d (\mm_\X \otimes \mm_\Y).
\]
Finally, adding up over $i$ and letting $n \to \infty$, thanks to \eqref{eq:ECX}, we get
\[
\int \E_{C, \sx}(v_\sx) \, \d \mm_\Y(y) \le\tfrac12 \int_{\X \times \Y} \abs{T}_\HS^2 \, \d (\mm_\X \otimes \mm_\Y),
\]
which is $i)$.
\end{proof}
The following is a natural variant of Lemma \ref{le:easy}:
\begin{lemma}\label{le:easyCov}
Let $(\X, \d_\X, \mm_\X)$ and $(\Y, \d_\Y, \mm_\Y)$ be $\RCD(K, \infty)$ spaces, and $v \in W_C^{1, 2}(\T(\X \times \Y))$. Then the identity
\begin{equation}\label{eq:cov}
\int h \nabla_C v : (\nabla g_1 \otimes \nabla g_2) \, \d (\mm_\X \otimes \mm_\Y) = - \int \la v, \nabla g_1 \ra \div(h \nabla g_2) -  h \H(g_2)(v, \nabla g_1) \, \d (\mm_\X \otimes \mm_\Y)
\end{equation}
holds for any $h \in \Lip_\bs(\X \times \Y)$ and $g_1, g_2 \in  \{ g \circ \pi_\X  : g \in \test{\X}\} \cup \{ g \circ \pi_\Y  : g \in \test{\Y}\}$.
\end{lemma}

\begin{proof}
First of all, let us consider the case in which $h \in \test{\X\times\Y}$ has bounded support, and let $\eta \in \test{\X\times \Y}$ be with bounded support and such that $\supp(h) \subset \{\eta = 1\}$, whose existence is guaranteed by \eqref{eq:testcutoff}. Then for $g_1, g_2 \in \{ g \circ \pi_\X : g \in \test{\X} \} \cup \{ g \circ \pi_\Y : g \in \test{\Y} \}$, we have that $\eta g_1, \eta g_2 \in \test{\X \times \Y}$, and the validity of \eqref{eq:cov} follows from the definition of Covariant Derivative and the locality property of the differential and the Hessian, which ensures that $\mm_\X \times \mm_\Y$-a.e. on $\supp(h)$ it holds $\nabla (\eta g_i) = \nabla g_i$, $i = 1, 2$ and $\H(\eta g_2) = \H(g_2)$.

At this point, in order to get the conclusion, we take $h \in \Lip_\bs(\X \times \Y)$, $\eta$ as before and $(h_n) \subset \test{\X \times \Y}$ a sequence $W^{1, 2}$-converging to $h$. Then $\eta h_n \in \test{\X\times \Y}$ for every $n \in \N$ and $\eta h_n \to \eta h = h$ in $W^{1, 2}(\X \times \Y)$. Since \eqref{eq:cov} holds with $\eta h_n$ and the functions have uniformly bounded support, we can pas to the limit in $n$, and get the claim.
\end{proof}
We are now ready to prove the main result of the section:
\begin{theorem}\label{thm:wCprod}
Let $(\X, \d_\X, \mm_\X)$ and $(\Y, \d_\Y, \mm_\Y)$ be $\RCD(K, \infty)$ spaces, and $v = (v_\sx, v_\sy) \in W_C^{1, 2}(\T(\X \times \Y))$. Then
\begin{itemize}
\item[i)] for $\mm_\Y$-a.e. $y \in \Y$ it holds $v_\sx(\cdot, y) \in W^{1, 2}_C(\T \X)$ with $\iint \abs{\nabla_{C, \sx} (v_\sx)}^2_{\HS}\,\d \mm_\X \, \d \mm_\Y < \infty$,
\item[ii)] for $\mm_\X$-a.e. $x \in \X$ it holds $v_\sy(x, \cdot) \in W^{1, 2}_C(\T \Y)$ with $\iint \abs{\nabla_{C,\sy} (v_\sy)}^2_{\HS}\,\d \mm_\Y \, \d \mm_\X < \infty$,
\item[iii)] we have $y \mapsto v_\sx(\cdot, y) \in W^{1, 2}(\Y; \Lint^2(\T \X))$ and $x \mapsto v_\sy(x, \cdot) \in W^{1, 2}(\X; \Lint^2(\T \Y))$.
\end{itemize}
In particular, the identities \eqref{eq:covprod} and \eqref{eq:covprod2} hold.
\end{theorem}

\begin{proof}
Point $i)$ is a direct consequence of Lemma \ref{le:easyCov} with $g_1, g_2 \in \{ g \circ \pi_\X  : g \in \test{\X}\}$ together with Proposition \ref{prop:covx}. In the same way we get also $ii)$.

In order to prove $iii)$, we start observing that directly from \eqref{eq:cov} it holds
\[
\int h \nabla_C v : (\nabla g_1 \otimes \nabla g_2) + h \H(g_2)(v, g_1) \, \d (\mm_\X \otimes \mm_\Y) = - \int \la v, \nabla g_1 \ra \div(h \nabla g_2) \, \d(\mm_\X \otimes \mm_\Y)
\]
for $g_1, g_2 \in \{ g \circ \pi_\X  : g \in \test{\X}\} \cup \{ g \circ \pi_\Y  : g \in \test{\Y}\}$, and $h \in \test{\X \times \Y}$ with bounded support. Thus, if we take $g_1 = g \circ \pi_\X$ and $g_2 = \tilde g \circ \pi_\Y$, with $g \in \test{\X}$ and $\tilde g \in \test{\Y}$ respectively, and we take into account Lemma \ref{le:horto}, we get
\[
\int h \nabla_C v : (\nabla (g \circ \pi_\X) \otimes \nabla (\tilde g \circ \pi_\Y)) \, \d (\mm_\X \otimes \mm_\Y) = - \int \la v, \nabla (g \circ \pi_\X) \ra \div(h \nabla (\tilde g \circ \pi_\Y)) \, \d(\mm_\X \otimes \mm_\Y).
\]
Recalling Proposition \ref{prop:defPhi} and denoting by $T$ the restriction of the Covariant Derivative of $v$ to $\Phi_\sx(L^2(\Y; L^2(\T \X))) \otimes \Phi_\sy(L^2(\X; L^2(\T \Y))) \subset L^2(T^{\otimes 2}(\X \times \Y))$, we obtain
\begin{equation}\label{eq:covdxdy}
\int h \la T, \d_\sx g \otimes \d_\sy \tilde g \ra \, \d (\mm_\X \otimes \mm_\Y) = - \int \la v_\sx, \d_\sx g \ra \div(h \nabla (\tilde g \circ \pi_\Y)) \, \d(\mm_\X \otimes \mm_\Y).
\end{equation}
By now we have proven the validity of \eqref{eq:covdxdy} for $\tilde g \in \test{\Y}$ and $h \in \test{\X \times \Y}$ with bounded support: a standard approximation argument involving the heat flow and the multiplication with a test cut-off function with bounded support allows to conclude that actually this equality holds for $h \in \Lip_\bs(\X \times \Y)$ and $\tilde g \in D(\Delta_\sy)$.

Since $\{ \d_\sx g : g \in \test{\X} \}$ generates $L^2(\T^* \X)$, we can then apply Proposition \ref{prop:perdxdy} to conclude that $\Y \ni y \mapsto v_\sx(\cdot, y) \in W^{1, 2}(\Y; L^2(\T \X))$ with $\d_\sy v_\sx = T$. An analogous argument applies to the case $\d_\sx v_\sy$.

At this point we have that $i), ii), iii)$ in Proposition \ref{prop:dafattCov} hold true, and this in particular means that the identities \eqref{eq:covprod} and \eqref{eq:covprod2} are valid.
\end{proof}

We conclude this section providing an analogous of Proposition \ref{prop:h22prod}: we prove that if the vector field $v = (v_\sx, v_\sy)$ in the previous statement is assumed to be in $H_C^{1, 2}(\T(\X\times \Y))$, then also the components are in the respective $H_C^{1, 2}$ spaces.
\begin{proposition}\label{prop:hCprod}
Let $(\X, \d_\X, \mm_\X)$ and $(\Y, \d_\Y, \mm_\Y)$ be two $\RCD(K, \infty)$ spaces, and $v=(v_\sx,v_\sy) \in H_C^{1, 2}(\T(\X \times \Y))$. Then in addition to the conclusions of Theorem \ref{thm:wCprod} above it holds:
\begin{itemize}
\item[i')] $v_\sx(\cdot, y) \in H^{1, 2}_C(\T \X)$ for $\mm_\Y$-a.e. $y \in \Y$,
\item[ii')] $v_\sy(x, \cdot) \in H^{1, 2}_C(\T \Y)$ for $\mm_\X$-a.e. $x \in \X$.
\end{itemize}
\end{proposition}
\begin{proof} 
We start claiming that for $f, g \in \test{\X\times \Y}$ if we consider functions $\{f_{n, t}\}, \{ g_{n, t}  \} \subset \Test_\times \fct (\X \times \Y)$ parametrized by $n\in\N$ and $t>0$ as given by Lemma \ref{le:apprtestx} we have that
\begin{equation}
\label{eq:claimtestprod}
\lim_{t\downarrow0}\lim_{n\to\infty}f_{n,t}\nabla g_{n,t}=f\nabla g\qquad\textrm{ in }W^{1,2}_C(T(\X\times\Y)).
\end{equation}
Indeed we have
\[
\begin{split}
\|  f_{n,t}\nabla g_{n,t} - f\nabla g  \|_{L^2} \le \| f_{n, t} \|_{L^\infty} \| \nabla (g_{n, t} - g)  \|_{L^2} + \| \abs{\nabla g} \|_{L^\infty} \| f_{n, t} - f \|_{L^2} 
\end{split}
\]
and $\sup_{n,t} \| f_{n, t} \|_{L^\infty}<\infty$ by $(i)$ of Lemma \ref{le:apprtestx}, $\lim_t\lim_n\|_{L^\infty} \| f_{n, t} - f \|_{L^2}=0$ by  $(iii)$ of Lemma \ref{le:apprtestx} and  $\lim_t\lim_n\|_{L^\infty} \| \nabla (g_{n, t} - g)  \|_{L^2}=0$ by  $(iv)$ of Lemma \ref{le:apprtestx} and  the continuity in $W^{1,2}$ of the heat flow. This proves that the limit in \eqref{eq:claimtestprod} holds in $L^2(T(\X\times\Y))$. To obtain that it holds in $W^{1,2}_C(T(\X\times\Y))$, taking into account the identity $\nabla(f\nabla g)=\nabla f\otimes\nabla g+f\H g^\sharp$, we see that it remains to prove that
\begin{equation}
\label{eq:partegrad}
\lim_{t\downarrow0}\lim_{n\to\infty}\Big\| \Big(\nabla f_{n, t} \otimes \nabla g_{n, t} + f_{n, t} (\hess(g_{n, t}))^\sharp\Big) - \Big(\nabla f \otimes \nabla g + f (\hess(g))^\sharp\Big)  \Big\|_{L^2(T^{\otimes 2}(\X \times \Y))} = 0,
\end{equation}
The convergence of $\nabla f_{n, t} \otimes \nabla g_{n, t}$ to $\nabla f \otimes \nabla g$ in $L^2(T^{\otimes 2}(\X \times \Y))$ is ensured by $iv)$ in Lemma \ref{le:apprtestx}, \eqref{BE}, and a diagonalization argument. For the other term we notice that \eqref{eq:hesscont} together with $iii)$ and $iv)$ in Lemma \ref{le:apprtestx} gives that $\H g_{n,t}\to\H\h_tg$ in $L^2((T^*)^{\otimes 2}(\X\times\Y))$ as $n\to\infty$ for any $t>0$. In particular, for every $t>0$ the sequence $|\H g_{n,t}-\H\h_tg|_\HS$ is dominated in $L^2$ and this fact together with $L^2$-convergence of $(f_{n,t})$ to $\h_tf$, the uniform $L^\infty$-bound on these functions and an application of the dominated convergence theorem gives 
\[
\|(f_{n, t}-\h_tf)( \hess(g_{n, t})- \hess(\h_tg))\|_{L^2}\to 0
\]
as $n\to\infty$ for every $t>0$. It is then clear that
\[
\lim_{n\to\infty}\|f_{n, t} \hess(g_{n, t})-\h_tf \hess(\h_tg)\|_{L^2}= 0.
\]
A similar line of thought gives $\lim_{t\downarrow0}\|\h_tf \hess(\h_tg)-f \hess(g)\|_{L^2}= 0$ hence \eqref{eq:partegrad}, and thus \eqref{eq:claimtestprod}, follows.

From this claim and the definition of $H^{1,2}_C(T(\X\times\Y))$ it follows that for $v\in H^{1,2}_C(T(\X\times\Y))$ there is a sequence $(v_n)$ converging to $v$ in $W^{1,2}_C(T(\X\times\Y))$ so that each $v_n$ is of the form $\sum f_i\nabla g_i$ for $f_i,g_i\in \Test_\times \fct (\X \times \Y)$. Up to pass to a subsequence and using the bounds $|v|\geq |v_\sx|$ and $|\nabla_Cv|\geq |\nabla_{C,\sx}f_\sx|$ (for the latter recall \eqref{eq:covprod}) we can assume that for $\mm_\Y$-a.e.\ $y\in\Y$ we have $v_{n,\sx}(y)\to v_\sx(y)$ in $W^{1,2}(\X)$ as $n\to\infty$ (recall that the component $v_\sx$ of a vector field $v\in L^2(T(\X\times\Y))$ belongs to $L^2(\Y;L^2(T\X))$) and similarly that $v_{n,\sy}(x)\to v_\sy(x)$ in $W^{1,2}(\Y)$ for $\mm_\X$-a.e.\ $x\in\X$.

The conclusion then follows from the observation that vectors $v_n$ as those considered are such that   $v_{n,\sx}(y) \in \Test\V(\X)$ for $\mm_\Y$-a.e. $y \in \Y$ and, similarly, $v_{n,\sy}(x) \in \Test\V(\Y)$ for $\mm_\X$-a.e. $x \in \X$.
\end{proof}

\bibliographystyle{abbrv}
\bibliography{partder}

\def\cprime{$'$} \def\cprime{$'$}
\begin{thebibliography}{10}

\bibitem{ACM14}
L.~Ambrosio, M.~Colombo, and S.~Di~Marino.
\newblock Sobolev spaces in metric measure spaces: reflexivity and lower
  semicontinuity of slope.
\newblock Accepted at Adv. St. in Pure Math., arXiv:1212.3779, 2014.

\bibitem{AmbrosioGigliSavare08}
L.~Ambrosio, N.~Gigli, and G.~Savar{\'e}.
\newblock {\em Gradient flows in metric spaces and in the space of probability
  measures}.
\newblock Lectures in Mathematics ETH Z\"urich. Birkh\"auser Verlag, Basel,
  second edition, 2008.

\bibitem{AmbrosioGigliSavare11}
L.~Ambrosio, N.~Gigli, and G.~Savar{\'e}.
\newblock Calculus and heat flow in metric measure spaces and applications to
  spaces with {R}icci bounds from below.
\newblock {\em Invent. Math.}, 195(2):289--391, 2014.

\bibitem{AmbrosioGigliSavare11-2}
L.~Ambrosio, N.~Gigli, and G.~Savar{\'e}.
\newblock Metric measure spaces with {R}iemannian {R}icci curvature bounded
  from below.
\newblock {\em Duke Math. J.}, 163(7):1405--1490, 2014.

\bibitem{AmbrosioGigliSavare12}
L.~Ambrosio, N.~Gigli, and G.~Savar{\'e}.
\newblock Bakry-\'{E}mery curvature-dimension condition and {R}iemannian
  {R}icci curvature bounds.
\newblock {\em The Annals of Probability}, 43(1):339--404, 2015.

\bibitem{Ambrosio-Pinamonti-Speight15}
L.~Ambrosio, A.~Pinamonti, and G.~Speight.
\newblock Tensorization of {C}heeger energies, the space {$H^{1,1}$} and the
  area formula for graphs.
\newblock {\em Adv. Math.}, 281:1145--1177, 2015.

\bibitem{Cheeger00}
J.~Cheeger.
\newblock Differentiability of {L}ipschitz functions on metric measure spaces.
\newblock {\em Geom. Funct. Anal.}, 9(3):428--517, 1999.

\bibitem{Gigli17}
N.~Gigli.
\newblock Lecture notes on differential calculus on {${\rm RCD}$} spaces.
\newblock Preprint, arXiv: 1703.06829.

\bibitem{Gigli13}
N.~Gigli.
\newblock The splitting theorem in non-smooth context.
\newblock Preprint, arXiv:1302.5555, 2013.

\bibitem{Gigli13over}
N.~Gigli.
\newblock An overview of the proof of the splitting theorem in spaces with
  non-negative {R}icci curvature.
\newblock {\em Analysis and Geometry in Metric Spaces}, 2:169--213, 2014.

\bibitem{Gigli12}
N.~Gigli.
\newblock On the differential structure of metric measure spaces and
  applications.
\newblock {\em Mem. Amer. Math. Soc.}, 236(1113):vi+91, 2015.

\bibitem{Gigli14}
N.~Gigli.
\newblock Nonsmooth differential geometry---an approach tailored for spaces
  with {R}icci curvature bounded from below.
\newblock {\em Mem. Amer. Math. Soc.}, 251(1196):v+161, 2018.

\bibitem{GH15}
N.~Gigli and B.~Han.
\newblock Sobolev spaces on warped products.
\newblock Preprint, arXiv:1512.03177, 2015.

\bibitem{Gigli-Kuwada-Ohta10}
N.~Gigli, K.~Kuwada, and S.-i. Ohta.
\newblock Heat flow on {A}lexandrov spaces.
\newblock {\em Communications on Pure and Applied Mathematics}, 66(3):307--331,
  2013.

\bibitem{GP18}
N.~Gigli and E.~Pasqualetto.
\newblock Differential structure associated to axiomatic sobolev spaces.
\newblock {\em Expositiones Mathematicae}, 2019.

\bibitem{GR17}
N.~Gigli and C.~Rigoni.
\newblock Recognizing the flat torus among {${\rm RCD}^*(0,N)$} spaces via the
  study of the first cohomology group.
\newblock {\em Calc. Var. Partial Differential Equations}, 57(4):Art. 104, 39,
  2018.

\bibitem{GT01}
V.~Gol'dshtein and M.~Troyanov.
\newblock Axiomatic theory of sobolev spaces.
\newblock {\em Expositiones Mathematicae}, (19):289--336, 2001.

\bibitem{Lott-Villani09}
J.~Lott and C.~Villani.
\newblock Ricci curvature for metric-measure spaces via optimal transport.
\newblock {\em Ann. of Math. (2)}, 169(3):903--991, 2009.

\bibitem{Savare13}
G.~Savar{\'e}.
\newblock Self-improvement of the {B}akry-\'{E}mery condition and {W}asserstein
  contraction of the heat flow in {${\rm RCD}(K,\infty)$} metric measure
  spaces.
\newblock {\em Discrete Contin. Dyn. Syst.}, 34(4):1641--1661, 2014.

\bibitem{Shanmugalingam00}
N.~Shanmugalingam.
\newblock Newtonian spaces: an extension of {S}obolev spaces to metric measure
  spaces.
\newblock {\em Rev. Mat. Iberoamericana}, 16(2):243--279, 2000.

\bibitem{Sturm06I}
K.-T. Sturm.
\newblock On the geometry of metric measure spaces. {I}.
\newblock {\em Acta Math.}, 196(1):65--131, 2006.

\end{thebibliography}

\end{document}